\documentclass[a4paper,11pt]{amsart}
\usepackage[T1]{fontenc}
\usepackage{amsmath,amsthm,amssymb,amsfonts,enumerate,color,esint,bbm}
\usepackage[pdftex]{graphicx}
\usepackage{float}
\usepackage{tikz}
\usepackage{tikz-3dplot}
\usepackage[shortlabels]{enumitem}
\usepackage{cancel}

\usepackage[colorlinks=true, allcolors=blue]{hyperref}
\usepackage{amsrefs}
\usepackage{pgfplots}
\pgfplotsset{
compat=newest,
colormap={blackwhite}{gray(0cm)=(.25); gray(1cm)=(1)}
}
\usepackage{subfig}
\usetikzlibrary{hobby}

\newcommand{\N}{\mathbb{N}}
\newcommand{\R}{\mathbb{R}}

\newcommand{\norm}[2][]{\left\|#2\right\|_{#1}}
\newcommand{\abs}[1]{\left\vert#1\right\vert}
\def\({\left(}
\def\){\right)}

\newcommand{\one}{\mathbbm{1}}

\newcommand{\ep}{\varepsilon}

\newtheorem{thm}{Theorem}[section]
\newtheorem{prop}[thm]{Proposition}
\newtheorem{cor}[thm]{Corollary}
\newtheorem{lem}[thm]{Lemma}

\theoremstyle{definition}
\newtheorem{defn}[thm]{Definition}
\newtheorem{rem}[thm]{Remark}

\numberwithin{equation}{section}

\allowdisplaybreaks

\author{Erisa Hasani}
\email{ehasani@utexas.edu}

\author{Stefania Patrizi} 
\address{Department of Mathematics\\
The University of Texas at Austin\\
2515 Speedway, Austin\\
TX 78712, United States of America}
\email{spatrizi@math.utexas.edu}

\keywords{Phase transitions, 
nonlocal integro-differential equations,
mean curvature, 
fractional mean curvature, 
dislocation dynamics.}

\subjclass[2010]{Primary: 35R09, 74N20, 53C44.
Secondary: 35R11, 47G20.}

\begin{document}

\title[The strongly nonlocal Allen--Cahn problem]{The strongly nonlocal Allen--Cahn problem}

\begin{abstract}
We study the sharp interface limit of the fractional Allen–Cahn equation
\begin{equation*}
\ep \partial_t u^{\ep} = \mathcal{I}^s_n [u^{\ep}]  -\frac{1}{\ep^{2s}}  W'(u^\ep) \quad \hbox{in}~(0,\infty)\times \R^n, ~n \geq 2,
\end{equation*} 
where $\ep>0$, 
$\mathcal{I}^s_n=-c_{n,s}(-\Delta )^s$ is  the fractional  Laplacian of order  $2s\in(0,1)$  in $\R^n$,  
and $W$ is a smooth double-well potential with minima at  0 and 1. 
We focus on the singular regime $s\in(0,\frac12)$, corresponding to strongly nonlocal diffusion. For suitably prepared initial data, we prove that the solution 
 $ u^\ep$ converges, as $\ep\to0$, to the minima of $W$ with the interface evolving by fractional mean curvature flow.
  This establishes the first rigorous convergence result in this regime, complementing and completing previous work for $s\geq \frac12$. 

\end{abstract}


\maketitle

\section{Introduction}

We study the fractional Allen--Cahn equation
\begin{equation} \label{eq:pde}
\ep \partial_t u^{\ep} = \mathcal{I}^s_n [u^{\ep}]  -\frac{1}{\ep^{2s}}  W'(u^\ep) \quad \hbox{in}~(0,\infty)\times \R^n, ~n \geq 2,
\end{equation}
where $\ep>0$ is a small parameter, 
$\mathcal{I}^s_n=-c_{n,s}(-\Delta )^s$ denotes, up to a constant, the fractional Laplacian of order  $2s\in(0,1)$  in $\R^n$,  
and $W$ is a smooth double-well potential with wells at 0 and 1 (see \eqref{eq:operator} and \eqref{eq:W} respectively). 

Equation \eqref{eq:pde}  is the (time-rescaled) $L^2$-gradient flow associated with the Allen--Cahn--Ginzburg--Landau--type energy 
\begin{equation}\label{energyintr}E_\ep(u)=\frac12[u]^2_{H^s(\R^n)}+\frac{1}{\ep^{2s}}\int_{\R^n}W(u)\,dx\end{equation}
where the first term represents the nonlocal interaction energy, given by the squared Gagliardo semi-norm in
 $H^{s}(\R^n)$,  and the second term is the potential energy, which forces minimizers to stay close to the wells 0 and 1.  
 
 We specifically consider the case $s\in(0,\frac12)$, which accounts for a strongly nonlocal elastic term: the smaller the value of 
  $s$,  the stronger the contribution of long-range interactions to the energy.

Equation \eqref{eq:pde} arises naturally, for instance, in the study of the Peierls–Nabarro model for crystal dislocations \cite{PN1, PN2}; see also the one-dimensional and higher-dimensional formulations in \cite{PatriziVaughan, MonneauPatrizi, MonneauPatrizi2}. 

 We show that, for well-prepared initial data (see \eqref{initial_data}), the solution $u^\ep$ to \eqref{eq:pde} converges, as $\ep\to0$, to  0 and 1, and that the  interface between the two phases evolves  by fractional mean curvature.
 
 In the stationary setting, this limiting behavior was previously established by Savin--Valdinoci \cite{SavinValdinoci}, who proved that  the energy 
$E_\ep$, when restricted to functions with the same values on the complement of a  bounded  domain $\Omega$, $\Gamma$-converges to the so-called fractional perimeter functional of order $2s$  in $\Omega$.
Minimizers of this limit functional, characteristic functions of nonlocal minimal surfaces,   were  studied by Caﬀarelli, Roquejoﬀre and Savin \cite{CRS}.
In that work, analogously to the classical (local) theory, a natural notion of fractional mean curvature was introduced, and nonlocal minimal surfaces were characterized as those with zero fractional mean curvature.

The evolution  problem \eqref{eq:pde} was previously studied  by Imbert--Souganidis in the preprint \cite{Imbert}, 
where they developed a framework for singular limits of nonlocal reaction--diffusion equations. 
Their approach successfully handled the fractional Allen-Cahn problem in the  case $s\in[\frac12,1)$,  
under certain additional assumptions. This analysis was recently  completed and extended to cover the case of multiple interfaces for
 $s=\frac12$ in \cite{PatriziVaughan2}.  The regime $s\in(0, \frac12)$, though partially addressed in \cite{Imbert}, 
remained open.
Our result fills this gap by rigorously establishing the sharp interface limit and the motion by fractional mean curvature in the previously unresolved regime $s\in(0, \frac12)$.

Before further discussing the significance of our main result and its connections to prior work, we now formalize the problem.
\subsection{Setting of the problem and main result} The operator $\mathcal{I}^s_n$ is a nonlocal integro-differential operator and is defined on functions $u \in C^{0,1}(\mathbb{R}^n)$ by
\begin{equation}\label{eq:operator}
\mathcal{I}_n^s u(x) 
 = \int_{\R^n} \left( u(x+z) - u(x)\right) \,\frac{dz}{\abs{z}^{n+2s}}, \quad x \in \R^n.
\end{equation}
For further background on fractional Laplacians, see for example \cites{Hitchhikers,Stinga}.

The potential $W:[0,1]\to\R$ satisfies
\begin{equation}\label{eq:W}
\begin{cases}
W \in C^{3, \beta} ([0,1]) & \hbox{for some}~0 < \beta <1 \\
W>W(0) = W(1)=0& \hbox{on}~(0,1)\\
W'(0)=W'(1)=0\\
W''(0)=W''(1) >0.
\end{cases}
\end{equation}

We let $u^\ep$ be the solution to \eqref{eq:pde} when the initial datum is given in terms of the layer solution. The layer solution (also called the phase transition) $\phi :\R \to (0,1)$ is the unique solution to 
\begin{equation} \label{eq:standing wave}
\begin{cases}
 C_{n,s}  \mathcal{I}_1^s[\phi] = W'(\phi) & \hbox{in}~\R\\
 \dot{\phi}>0 & \hbox{in}~\R\\
\phi(-\infty) = 0, \quad \phi(+\infty)=1,\quad \phi(0) = \frac{1}{2},
\end{cases}
\end{equation}
where $\mathcal{I}_1^s$ denotes the nonlocal operator in \eqref{eq:operator} with $n=1$ and the constant $C_{n,s}>0$ (given explicitly in \eqref{eq:Cns}) depends only on $s \in (0,\frac{1}{2})$ and on the dimension $n \geq 2$. Further discussion on $\phi$ is presented in Section \ref{sec:phi}.

Let $\Omega_0$ denote a bounded open subset in $\mathbb{R}^n$ with smooth boundary $\Gamma_0 = \partial \Omega_0$, and let $d^0(x)$ be its  signed distance function, given by
\begin{equation}\label{def:signed_distance_function}
    d^0(x) = \begin{cases}
        d(x, \Gamma_0) & \text{ if } x \in \Omega_0 \\
        -d(x, \Gamma_0) & \text{ otherwise}.
    \end{cases}
\end{equation}
For the initial condition to be well-prepared, we set $u_0^\ep = \phi \left( \frac{d^0(x)}{\ep} \right)$, see Figure \ref{fig:initial}.

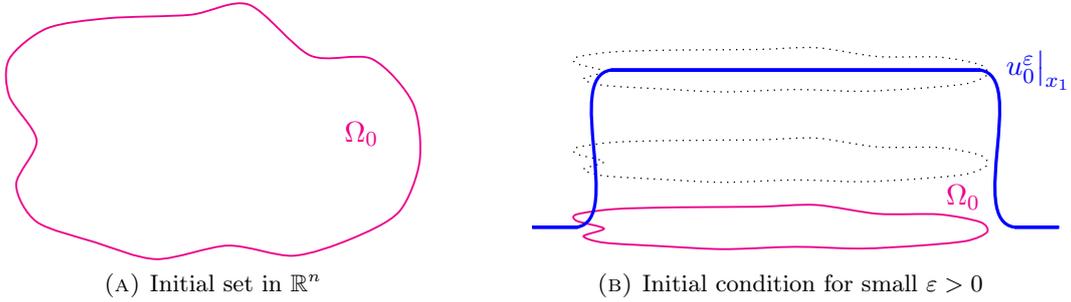
\begin{figure}[h]
\centering
\subfloat[Initial set in $\R^n$]{
 \begin{tikzpicture}[scale=0.45, use Hobby shortcut, closed=true]
    \draw[magenta, line width=.70pt] 
        plot[smooth cycle] coordinates {
            (1.5,3.5)  (3.2,3.8)  (5.0,2.3)  (6.8,2.2)
            (8.0,0.9)  (8.2,-0.8)  (7.6,-2.3)  (6.2,-3.1)
            (4.5,-3.6)  (2.6,-3.3)  (0.5,-3.7)  (-1.3,-3.2)
            (-3.0,-2.6)  (-3.6,-1.5)  (-3.0,-0.2)  (-3.8,1.1)
            (-3.8,2.2)  (-2.7,3.1)  (-0.9,3.4)
        };
    \node[magenta] at (6.5,0) {$\Omega_0$};

\end{tikzpicture}
}
\qquad
\subfloat[Initial condition for small $\ep>0$]{

\tdplotsetmaincoords{80}{0} 
 \begin{tikzpicture}[scale=0.45, tdplot_main_coords]
 
\def\bottom{
    (1.5,-0.5,0)  (3.2,-0.2,0)  (5.0,-1.7,0)  (6.8,-1.8,0)
    (8.0,-3.1,0)  (8.2,-4.8,0)  (7.6,-6.3,0)  (6.2,-7.1,0)
    (4.5,-7.6,0)  (2.6,-7.3,0)  (0.5,-7.7,0)  (-1.3,-7.2,0)
    (-3.0,-6.6,0)  (-3.6,-5.5,0)  (-3.0,-4.2,0)  (-3.8,-2.9,0)
    (-3.8,-1.8,0)  (-2.7,-0.9,0)  (-0.9,-0.6,0)
}

\def\mid{
    (1.5,-0.5,2)  (3.2,-0.2,2)  (5.0,-1.7,2)  (6.8,-1.8,2)
    (8.0,-3.1,2)  (8.2,-4.8,2)  (7.6,-6.3,2)  (6.2,-7.1,2)
    (4.5,-7.6,2)  (2.6,-7.3,2)  (0.5,-7.7,2)  (-1.3,-7.2,2)
    (-3.0,-6.6,2)  (-3.6,-5.5,2)  (-3.0,-4.2,2)  (-3.8,-2.9,2)
    (-3.8,-1.8,2)  (-2.7,-0.9,2)  (-0.9,-0.6,2)
}

\def\top{
    (1.5,3.5,4)  (3.2,3.8,4)  (5.0,2.3,4)  (6.8,2.2,4)
    (8.0,0.9,4)  (8.2,-0.8,4)  (7.6,-2.3,4)  (6.2,-3.1,4)
    (4.5,-3.6,4)  (2.6,-3.3,4)  (0.5,-3.7,4)  (-1.3,-3.2,4)
    (-3.0,-2.6,4)  (-3.6,-1.5,4)  (-3.0,-0.2,4)  (-3.8,1.1,4)
    (-3.8,2.2,4)  (-2.7,3.1,4)  (-0.9,3.4,4)
}

\def\topline{
    (8,0,4) 
    (0,0,4)  
    (-2.5,0,4) 
}

\def\bottomright{
    (9.035,0,-0.7) 
    (9.3,0,-0.7)  
    (10.3,0,-0.7) 
}

\def\bottomleft{
    (-3.8,0,-0.7) 
    (-4.1,0,-0.7)  
    (-5.1,0,-0.7) 
}

\draw[magenta, line width=.70pt] plot[smooth cycle] coordinates {\bottom};
\draw[dotted, line width=.5pt] plot[smooth cycle] coordinates {\top};
\draw[dotted, line width=.5pt] plot[smooth cycle] coordinates {\mid};

\draw[very thick, blue] plot[smooth cycle] coordinates {\topline};
\draw[very thick, blue] plot[smooth cycle] coordinates {\bottomright};
\draw[very thick, blue] plot[smooth cycle] coordinates {\bottomleft};

\draw[very thick, blue]
    (8,0,4) .. controls (9.4,0,3.8) and (7.6,0,-0.5) .. (9.0,0,-0.7);

\draw[very thick, blue]
    (-2.8,0,4) .. controls (-4.2,0,3.8) and (-2.4,0,-0.5) .. (-3.8,0,-0.7);

\node[magenta] at (7.5,0,0.2) {$\Omega_0$};
\node[very thick, blue] at (9.7,0,4) {$u_0^\ep{\big|}_{x_1}$};

\node[opacity=0] at (1,-8,0) {};

\end{tikzpicture}
}
\caption{Initial configuration in dimension $n=2$}
\label{fig:initial}
\end{figure}
Consider a continuous viscosity solution $u(t,x)$ to the fractional mean curvature equation (see \eqref{eq:fmc equation} with $c_0$ as in \eqref{def:c0}) whose positive, zero and negative sets at time $t=0$ are $\Omega_0$, $\Gamma_0$ and $(\overline{\Omega_0})^c$, respectively. If $^+\Omega_t$, $\Gamma_t$, and $^-\Omega_t$ are the positive, zero and negative sets, respectively, of $u(t, \cdot)$ at time $t>0$, then we say that the collection $(^+\Omega_t, \Gamma_t, ^-\Omega_t)_{t\geq0}$ is the level set evolution of $(\Omega_0,\Gamma_0,(\overline{\Omega_0})^c)$. See Section \ref{sec:FMC} for definitions and details on the level set approach to motion by fractional mean curvature.

We now present the main result of the paper.
\begin{thm}\label{thm:main_result}
    Let $u^\ep = u^\ep(t,x)$ be the unique solution of the reaction-diffusion equation \eqref{eq:pde} with initial datum $u_0^\ep:\mathbb{R}^n \to (0,1)$ defined by
    \begin{equation}\label{initial_data}
        u_0^\ep(x) = \phi \left( \frac{d^0(x)}{\ep} \right)
    \end{equation}
    where $\phi$ solves \eqref{eq:standing wave} and $d^0$ is given in \eqref{def:signed_distance_function}. Then, as $\ep \to 0$, the solution $u^\ep$ satisfies
    \begin{align*}
        u^\ep \to 
        \begin{cases}
1 &^+\Omega_t,\\
 \,\, \quad \hbox{locally uniformly in}~&\\
 0 & ^-\Omega_t.
\end{cases}
    \end{align*}
    where $(^+\Omega_t, \Gamma_t, ^-\Omega_t)_{t\geq0}$  denotes the level set evolution of $(\Omega_0,\Gamma_0,(\overline{\Omega_0})^c)$.
\end{thm}

\begin{figure}[h]
\centering
 \begin{tikzpicture}[scale=0.5, use Hobby shortcut, closed=true]
\def\Omeg{(1.3,3.4)  (3.4,3.6)  (5.1,3.3)  (6.5,2.1)
            (8.0,0.9)  (8.2,-0.8)  (7.6,-2.3)  (6.2,-3.1)
            (4.3,-3.2)  (2.6,-3.2)  (0.5,-3.4)  (-1.3,-3.2)
            (-3.3,-2.3)  (-3.6,-1.2)  (-3.4,-0.2)  (-3.8,1.1)
            (-3.8,2.2)  (-2.7,3.1)  (-0.9,3.4)};

\draw[magenta, line width=.65pt] plot[smooth cycle] coordinates {\Omeg};
\fill[magenta!20, opacity=0.2] plot[smooth cycle] coordinates {\Omeg};

\node[] at (1,2.25) { $u^{\ep} \to 1$};
\node[] at (1,4.5) { $u^{\ep} \to 0$};
\node[magenta] at (9,0) {$\Gamma_t$};
\end{tikzpicture}
\caption{Convergence result in dimension $n=2$}
\label{fig:main thm}
\end{figure}
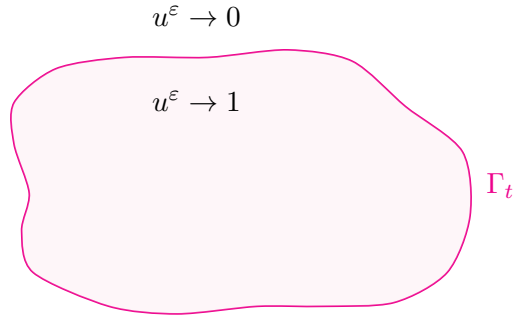

As illustrated in Figure \ref{fig:main thm}, Theorem \ref{thm:main_result}  says that the solution $u^{\ep}$ converges to 0 and 1 ``between'' the interface $\Gamma_t$.
Moreover, $\Gamma_t$ moves by fractional mean curvature. Specifically, it moves in the direction of the interior normal vector with 
 scalar velocity
\begin{equation*}\label{eq:velocity-intro}
v=-\frac{c_0}{2} H_{2s}(^+\Omega_t),
\end{equation*}
where $ H_{2s}(^+\Omega_t)$ is the fractional mean curvature of order $2s$ of $^+\Omega_t$ and $c_0>0$ is explicit (see \eqref{def:c0}).
See  Section \ref{sec:FMC} for the definition and properties of the fractional mean curvature of a set. 

We use the level set approach to handle possible singularities for large times $t>0$.

For the case in which 
the set $\Gamma_t$ doesn't  develop interior, i.e.~$\Gamma_t= \partial({^+}\Omega_t) = \partial ({^-}\Omega_t)$, 
the limiting function in Theorem  \ref{thm:main_result}  makes the jump on the surface $\Gamma_t$ and satisfies
\[
\lim_{\ep \to 0} u^{\ep} = \frac{1}{2} + \frac{1}{2}\( \one_{{^+}\Omega_t^i} - \one_{(\overline{{^+}\Omega_t^i})^c}\)  \quad \hbox{in}~\big((0,\infty)\times  \R^n\big)
\setminus  \bigg(\bigcup_{t>0}\{t\}\times \Gamma_t\bigg)
\]
where $\one_\Omega$ denotes the characteristic function of the set $\Omega \subset \R^n$. 
However,  
it is well known that  $\Gamma_t$ may develop interior in finite time,  even if $\Gamma_0$ has none, see~\cite{CesaroniDipierroNovagaVal}. In this situation,
the discontinuity set at time $t$ of the limiting function is contained in the set $\Gamma_t$, but we cannot say exactly where the jump occurs within this set.

\subsection{Strategies and prior work.}

We now discuss the key aspects of Theorem \ref{thm:main_result}, its proof, and some of the relevant literature.

Theorem \ref{thm:main_result} has been addressed in the literature in the local case.  
For instance, the classical Allen--Cahn equation for which \eqref{eq:pde} is instead driven by the usual Laplacian $\Delta$ was studied famously by Modica--Mortola \cite{MM}  for the stationary case. 
Chen studied the corresponding evolutionary Allen--Cahn problem and proved that the solution exhibits an interface moving by mean curvature \cite{Chen}. 
Using the framework of viscosity solutions and the level set method, Evans--Soner--Souganidis \cite{EvansSonerSouganidis} established convergence to mean curvature flow for all times, including beyond the formation of singularities. See Section \ref{sec:FMC} for more on the phase field theory. 

In the fractional setting,  for any $s\in(0,1)$, the stationary case was studied by  Savin--Valdinoci  \cite{SavinValdinoci} (see also  \cite{Alberti, AmbrosioDePhilippisMartinazzi,Contigarmul,GarroniMuller} for related $\Gamma$-convergence results). 
 They showed that for $s < \frac{1}{2}$, the fractional Allen–Cahn energy \eqref{energyintr}, when restricted to functions that agree outside a bounded domain $\Omega$, $\Gamma$-converges to the fractional perimeter functional in $\Omega$.
  For  $s \geq \frac{1}{2}$, a properly rescaled energy functional is considered:
\begin{equation}\label{energyintr2}E_\ep(u)=\frac{1}{\eta_s}\(\ep^{2s}\frac12[u]^2_{H^s(\R^n)}+\int_{\R^n}W(u)\,dx\)\end{equation}
where $\eta_s=\ep |\ln \ep|$ if $s=\frac12$ and $\eta_s=\ep$ if $s>\frac12$.
Under this rescaling, the energy 
$\Gamma$-converges to  the classical  (local) perimeter functional in $\Omega$.

The evolution problem for $s\in(0,1)$ was studied by Imbert--Souganidis in the preprint \cite{Imbert}.
In the case $s\geq \frac12$, they proved that the interface $\Gamma_t$ evolves according to the classical mean curvature flow.
However, their analysis assumes the existence of suitable one-dimensional solutions necessary for the convergence proof, without proving their existence.
 For the critical case  $s=\frac12$, their  work  was  completed and extended to  cover the case of multiple fronts by Patrizi--Vaughan  \cite{PatriziVaughan2}.  

When  $s<\frac12$,  only partial results were obtained in \cite{Imbert}, and the full convergence result remained an open problem. 
We now explain this in more detail.

The proof of Theorem \ref{thm:main_result}   
relies on the abstract method introduced by Barles--Da Lio \cite{Barles-DaLio} and Barles--Souganidis \cite{Barles-Souganidis} for the study of front propagation, and later extended to the fractional setting by Imbert \cite{Imbert09}.
To apply this method, we construct barriers in the form of strict subsolutions and supersolutions to \eqref{eq:pde}, see Section \ref{barriersection} for details. 
 In \cite{Imbert}, a subsolution is constructed near the interface $\Gamma_t$ using 
 the ansatz $\phi_c(d(t,x)/\ep)$,  
where $d(t,x)$ denotes the signed distance to  the evolving set ${^+}\Omega_t$,   and $\phi_c$  solves a traveling wave equation with speed $c$ (with $c=0$ corresponding to \eqref{eq:standing wave}).  However, the existence and asymptotic behavior of such traveling waves are assumed rather than proven.
In fact, the expected decay at infinity does not hold in the stationary case. 

Since the equation is nonlocal and nonlinear, a difficulty arises in dealing with $d(t,x)$ when $(t,x)$ is far from the front, since $d$ may not be smooth at such points. 
To address this,  \cite{Imbert} truncates and extends the subsolution away from the front, taking particular care when truncating from below in order to remain a subsolution. 
However,   when  $s<\frac12$ the equation is strongly nonlocal, and their method fails to produce a valid extension far from the interface in this case.

In \cite{PatriziVaughan2}, global subsolutions are  constructed for the critical case
 $s=\frac12$. Their construction uses the form  $\phi(\tilde{d}(t,x)/\ep)$ where $\tilde{d}$ is a smooth bounded extension of the signed distance function $d$ to ${^+}\Omega_t$ and $\phi$ solves the stationary equation \eqref{eq:standing wave}, whose existence, uniqueness, and asymptotic behavior are known  (see Lemma \ref{lem:asymptotics}). 

In this paper, for the case $s<\frac12$, we adopt the approach developed in    \cite{PatriziVaughan2}. 
However, additional difficulties arise because, roughly speaking, equation \eqref{eq:pde} is more singular than its $s = \frac12$ counterpart.
Specifically, the fractional Allen--Cahn equation with $s = \frac12$ includes an additional logarithmic term that is absent in our case (compare \eqref{energyintr} with   \eqref{energyintr2} with  $s = \frac12$). 

To prove the convergence result, it is necessary to introduce  a lower-order corrector to control the error as $\ep \to0$. 
This corrector is the solution   $\psi= \psi_\ep$ to the linearized equation
\begin{equation*}\label{eq:corrector-intro}
 -C_{n,s}\mathcal{I}^s_1[\psi] + {W}''(\phi) \psi =g
 \end{equation*}
 for some right-hand side $g = g_\ep$ depending on $\ep>0$ and on the signed distance function $d$. 
The explicit form of $g$ is somewhat technical (given in Section \ref{sec:phi}) and differs from the one in \cite{PatriziVaughan2} as well as in \cite{Imbert}. In  \cite{Imbert}, the existence of such correctors is assumed rather than proved.

Since the correctors depend on the parameter $\ep$ that tends to zero, a delicate analysis is required to obtain sharp estimates on their derivatives. These estimates, which  are essential for establishing  Theorem \ref{thm:main_result},   are specific for the case $s<\frac12$ and blow up when $s\to\frac12$.

The derivation of the corrector equation, a comparison with \cite{PatriziVaughan2}, and a heuristic proof of Theorem \ref{thm:main_result} are presented in 
 Section \ref{sec:Heuristics}.

 Another difficulty in constructing subsolutions and supersolutions to \eqref{eq:pde} is the presence of additional terms that do not decay far from the front. To control the resulting error in these regions, we introduce suitable auxiliary functions,  a step  not required for $s=\frac12$.

The one dimensional case with multiple fronts was studied by Gonzalez--Monneau  \cite{GonzalezMonneau}  for $s=\frac12$. The cases    $s\in(0,\frac12)$ and $s\in(\frac12,1)$
were later addressed in \cite{DipierroFigalliVald} and \cite{DipierroPalatucciValdinoci}, respectively. In this setting, the interfaces (i.e., the transition points between phases) evolve according to a long-range interaction potential determined by the fractional nature of the operator.
The case in which the solution is not monotone is investigated in \cite{MeursPatrizi,patval1, patval4}. The long time behavior of solutions is studied  in \cite{CozziDavilaDelPino, patval3}. Furthermore, the regime where the number of interfaces tends to infinity is explored in \cite{patsan, patsan2}.

The motion by fractional mean curvature has also been extensively studied in recent years. We refer the reader to  \cite{CesaroniDeLucaNovagaPonsiglione, CesaroniDipierroNovagaVal, CesaroniNovaga, ChambolleMoriniNovagaPonsiglione, ChambolleMoriniPonsiglione, CintiSinestrariValdinoci2,  CintiSinestrariValdinoci,SaezVal} and the references therein.
\subsection{Future directions}
We plan to extend the analysis in this paper to the case of multiple interfaces.
This will involve considering initial  datum given by the  superposition of functions of the form \eqref{initial_data},  and a multi-well potential $W$. 
Due to the nonlinearity and the strong nonlocality  of the problem, this extension is highly nontrivial.  
We expect that the limit configuration will consist of a superposition of characteristic  functions, with each interface evolving by fractional mean curvature plus an interaction potential depending on the distances to other interfaces. This behavior contrasts with the $s=\frac12$ case,  where fronts evolve independently by mean curvature, as shown in \cite{PatriziVaughan2}.

\subsection{Organization of the paper.}
The rest of the paper is organized as follows. In Section~\ref{sec:FMC} 
we recall the definition of fractional mean curvature and provide the necessary background on motion by fractional mean curvature and the level set formulation.  Section~\ref{sec:Heuristics} presents the heuristics for the proof of Theorem~\ref{thm:main_result} and for the equation satisfied by the corrector.  Section \ref{sec:ae-frac} contains preliminary results on fractional Laplacians and the 
solutions $u^\ep$. In Section \ref{sec:phi} we recall some  preliminary results for  the phase transition $\phi$, and establish 
preliminary results for the corrector $\psi$ and other auxiliary functions needed for the rest of the paper.  The construction of barriers is presented in Section~\ref{barriersection}. Section~\ref{sec:proof main_result} contains the proof of Theorem \ref{thm:main_result}. Lastly, since the proofs of some  auxiliary results in
Section~\ref{sec:phi} are rather technical; they are presented separately Sections~\ref{sec:proof of b_ep fmc}, \ref{sec: proof of a_ep and frac laplacians}, \ref{sec:proof of a_estimates}, and \ref{sec:proof of ae psi estimate}. 
\subsection{Notations}
Throughout the paper, we denote by $C>0$ any  constant independent of $\ep$ and the parameters $\delta$, $\sigma$, and $R$,  which will be introduced later. 

We write $B(x_0,r)$ and $\overline{B}(x_0,r)$ for the open and closed balls of radius $r>0$ centered at $x_0\in\R^n$, respectively, and $S^n$ for the unit sphere in $\R^{n+1}$. 

For $\beta \in (0,1]$,  $k \in \N \cup \{0\}$ and $m\in\N$, we denote by $C^{k,\beta}(\R^m)$ the usual class of functions with bounded $C^{k,\beta}$ norm over $\R^m$. 
For $\beta=0$ we simply write $C^{k}(\R^m)$.
For multi-variable functions $v(\xi;t,x)$, we write $v \in C_\xi^{k,\beta}(\R)$ if $v(\cdot;t,x) \in C^{k,\beta}(\R)$ for all $t,x$ in the domain of $v$. 
Moreover, we use the dot notation for derivatives with respect to the variable $\xi$, namely $\dot{v}(\xi;t,x) = v_\xi(\xi;t,x)$.

Given a function $\eta = \eta(t,x)$, defined on a set $A$,  we write $\eta = O(\ep)$ if there is $C>0$ such that
$|\eta(t,x)| \leq C \ep$ for all $(t,x)\in A$, and we write $\eta = o_\ep(1)$ if  $\lim_{\ep \to 0} \eta(t,x) = 0$, uniformly in $(t,x)\in A$. 

Given a sequence of functions $u^\ep(t,x)$, we define
\[
\liminf_{\ep \to 0}{_*} u^\ep(t,x) := \inf \left\{ \liminf_{\ep \to 0} u^\ep(t_\ep,x_\ep): ( t_\ep,x_\ep) \to (t,x)\right\}
\]
and
\[
\limsup_{\ep \to 0}{^*} u^\ep(t,x) :=\sup \left\{ \limsup_{\ep \to 0} u^\ep(t_\ep,x_\ep): ( t_\ep,x_\ep) \to (t,x)\right\}.
\]

For a set $A$, we denote by $\one_A$ the characteristic function of the set $A$.

\section{Motion by fractional mean curvature}\label{sec:FMC}

In this section, we present preliminary results concerning the evolution of fronts by fractional mean curvature.
\subsection{The fractional mean curvature} 

Let $\Omega$ be a smooth bounded subset of $\R^n$. For a point $x\in \partial \Omega$, the fractional mean curvature of order $2s$ of $ \Omega$ at $x$ is defined by
$$H_{2s}(\Omega)(x) =P.V.\int_{\R^n}\frac{\one_\Omega(z)-\one_{\Omega^c}(z)}{|z-x|^{n+2s}}\,dz,$$
where P.V.  denotes the Cauchy principal value. This quantity can also be expressed in terms of the signed  distance function $d$ to $\Omega$.
Indeed, since $$\Omega=\{z\,:\, d(z)>0\}$$ and using that 
$$P.V. \int_{\R^n}\frac{\one_{\{\nabla d(x)\cdot z>0\}}-\one_{\{\nabla d(x)\cdot z<0\}}}{|z|^{n+2s}}\,dz=0,$$ we can write
\begin{align*}
H_{2s}(\Omega)(x) &=P.V.\int_{\R^n}\frac{\one_\Omega(x+z)-\one_{\Omega^c}(x+z)}{|z|^{n+2s}}\,dz\\&
=P.V. \int_{\R^n}\frac{\one_{\{d(x+z)>0\}}-\one_{\{d(x+z)<0\}}+\one_{\{\nabla d(x)\cdot z<0\}}-\one_{\{\nabla d(x)\cdot z>0\}}}{|z|^{n+2s}}\,dz\\& 
=2 \int_{\{d(x+z)>0,\, \nabla d(x)\cdot z<0\}}\frac{dz}{|z|^{n+2s}}-2 \int_{\{d(x+z)<0,\, \nabla d(x)\cdot z>0\}}\frac{dz}{|z|^{n+2s}},
\end{align*}
where the last two integrals converge in the standard sense, as stated in Proposition \ref{prop:fmc_finite_result} below.
Assume $d$ is smooth in $Q_{2\rho}:=\{z\,:\,|d(z)|<2\rho\}$ for some $\rho>0$, then for $x\in Q_\rho$, define
\begin{equation}\label{kappa+de}\kappa^+[x,d]:=\int_{\{d(x+z)>d(x),\, \nabla d(x)\cdot z<0\}}\frac{dz}{|z|^{n+2s}},\quad 
\kappa^-[x,d] :=\int_{\{d(x+z)<d(x),\, \nabla d(x)\cdot z>0\}}\frac{dz}{|z|^{n+2s}},\end{equation}
and 
\begin{equation}\label{kappade}\kappa[x,d]:=\kappa^+[x,d]-\kappa^-[x,d].\end{equation}
From the discussion above, we obtain the identity
$$\kappa[x,d]=\frac12 H_{2s}(\{d>d(x)\})(x).$$ 
Roughly speaking, $\kappa[\cdot,d]$ plays the role of  $\Delta d$ in the local setting. 

Notice that, if $u$ is a smooth function such that 
$$\Omega=\{u>0\}\quad\text{and}\quad (\overline\Omega)^c=\{u<0\},$$ then  for all $x\in\partial \Omega$, 
\begin{equation*}\label{kappau=kappad}\kappa[x,u]=\kappa[x,d]. \end{equation*}
A proof of the following result can be found, for instance, in \cite[Lemma 7.3]{PatriziVaughan}. 
 \begin{prop}\label{prop:fmc_finite_result}
    Assume $d$ of class $C^2(Q_{2\rho})$, then for all $x\in Q_\rho$,  the  quantities  $ \kappa^+[x,d] $ and $\kappa^-[x,d]$ are finite. 
\end{prop}
The fractional mean curvature of balls can be explicitly computed, see \cite[Lemma 2]{SaezVal} for a proof.
\begin{prop}\label{ballFMC}For $r>0$, let $d(x)=r-|x|$. Then, for $x\neq 0$, 
$$\kappa[x,d]=-\frac{\omega}{|x|^{2s}},$$
for some $\omega>0$.
\end{prop}
\subsection{The level set approach} 
We review the level set approach for fractional mean curvature flows, a method originally introduced by Osher–Sethian \cite{OsherSethian}, Evans–Spruck \cite{EvansSpruck}, and Chen–Giga–Goto \cite{CGG} for the evolution of fronts under classical mean curvature flow.

Let  $u=u(t,x)$ be a smooth function, and  consider the level set $\Gamma_t = \{ x \in \mathbb{R}^n: u(t,x) = \ell \}$ of $u(t, \cdot)$ at the level $\ell \in \mathbb{R}$. Assume that $\Gamma_t$ is bounded and $\nabla u$ does not vanish on $\Gamma_t$. 
Then $n(t, x) := \nabla u(t, x)/|\nabla u(t,x)|$ is interior unit normal to  $\{ u > \ell \}$.
In a time-space neighborhood  $\mathcal{N}$ of $\Gamma_t$, the level set $\Gamma_t$,  as well as all the level sets of $u$  in $\mathcal{N}$ of $\Gamma_t$,    move in the direction of $n(t, x) $ with scalar velocity 
\begin{equation}\label{eq:velocity_of_fonts}
  v(t,x) =- c_0 \kappa[x,u(t,\cdot)]=-\frac{c_0}{2}H_{2s}(\{u>u(x)\})(x),
\end{equation} 
where $c_0$>0, if and only if $u$  satisfies  the fractional mean curvature equation
\begin{equation}\label{eq:fmc equation}
    \partial_t u = c_0 |\nabla u| \kappa[x,u],
\end{equation}
in $\mathcal{N}$. 

As in the classical case, the evolution of level sets by fractional mean curvature may develop singularities in finite time, see for instance \cite{CintiSinestrariValdinoci}. To account for such singularities and generalize the notion of evolving fronts,  Imbert \cite{Imbert09} introduced a weak formulation of fractional mean curvature flows based on the level set method and the theory of viscosity solutions for nonlocal degenerate equations.


More precisely,  for a bounded, open set $\Omega_0 \subset \mathbb{R}^n$, set $\Gamma_0 = \partial\Omega_0$ and consider the initial triplet $(\Omega_0, \Gamma_0, (\Omega_0)^c)$. Let $u_0(x)$ be a bounded and Lipschitz continuous function such that
\[
\Omega_0 = \{x : u_0(x) > 0\}, \quad \Gamma_0 = \{x : u_0(x) = 0\}, \quad (\Omega_0)^c = \{x : u_0(x) < 0\}. 
\]
Then, there exists a unique bounded uniformly continuous viscosity solution $u$ to \eqref{eq:fmc equation} in $(0, \infty) \times \mathbb{R}^n$ with initial datum $u(0, x) = u_0(x)$, see  \cite[Theorem 3]{Imbert09}. For the definition of viscosity solution of \eqref{eq:fmc equation}, see \cite[Definition 1]{Imbert09}. We define the time-evolving triplet as
\begin{equation}\label{eq:triplets}
{}^+\Omega_t := \{x : u(t, x) > 0\}, \quad \Gamma_t := \{x : u(t, x) = 0\}, \quad {}^-\Omega_t := \{x : u(t, x) < 0\}.
\end{equation}

The collection $({}^+\Omega_t, \Gamma_t, {}^-\Omega_t)_{t\geq 0}$ is called the level set evolution of the initial configuration $(\Omega_0, \Gamma_0, (\Omega_0)^c)$. As shown in \cite[Theorem 6]{Imbert09}, the interface $\Gamma_t$ depends only on the initial zero level set $\Gamma_0$, and not on the specific choice of $u_0$.

Assume  $\Gamma_t$  smooth for $t \in [t_0, t_0+h]$, and let 
$d(t,x)$ denote the signed distance function to the set $\{ u > 0 \}$, defined by 
\begin{align*}
d(t, x) =
\begin{cases}
d(x, \Gamma_t) & \text{for } u(t, x) \geq 0\\
-d(x, \Gamma_t) & \text{for } u(t, x) < 0.
\end{cases}
\end{align*}
If $u$ solves 
\begin{align*}
    \partial_t u = c_0 |\nabla u| \kappa[x,u] + \sigma,
\end{align*}
in a neighborhood of $\Gamma_t$ for some $\sigma = \sigma(t,x)$, then since 
$$\partial_t d = \dfrac{\partial_t u}{|\nabla u|}\quad\text{and}\quad \kappa[x,d]=\kappa[x,u]
\quad \text{in} \quad \bigcup_{t \in [t_0, t_0+h]} \{ t \} \times \Gamma_t,$$
the function $d$ solves 
\begin{equation}\label{meancurvature_intro_sigma}
    \partial_t d = c_0 \kappa[x, d] + \frac{\sigma}{|\nabla u|} \quad \text{in} \quad \bigcup_{t \in [t_0, t_0+h]} \{ t \} \times \Gamma_t.
\end{equation}

\subsection{Generalized Flows}\label{sec:flows}
We now present the definition of generalized flows for our problem as introduced in \cite{Imbert09}. 
Let us first define the singular measure 
\begin{equation*}\label{def:singular_measure}
\nu(dz) = \frac{dz}{|z|^{n+2s}}.
\end{equation*}
Let $D \subset \mathbb{R}^n$ be open and $E \subset \mathbb{R}^n$ be closed. For all $x,p\in \mathbb{R}^n$, let $F^\ast(x,p,D)$ and $F_\ast(x,p,E)$ be defined, respectively, as 
\begin{equation}\label{def:F_ast_definition}
\begin{aligned}
    F^\ast(x,p,D) &=
    \begin{cases}
        -c_0\Big[ \nu\left( D \cap \left\{ p \cdot z < 0 \right\} \right) - \nu\left( D^c \cap \left\{ p \cdot z \geq 0 \right\} \right) \Big] |p| & \text{ if } p \neq 0, \\[5pt]
        0 & \text{ if } p =0,
    \end{cases} \\
    F_\ast(x,p,E) &=
    \begin{cases}
        -c_0\Big[ \nu\left( E \cap \left\{ p \cdot z \le 0 \right\} \right) - \nu\left( E^c \cap \left\{ p \cdot z > 0 \right\} \right) \Big] |p| & \text{ if } p \neq 0, \\[5pt]
        0 & \text{ if } p =0.
    \end{cases} 
\end{aligned}
\end{equation}
\begin{defn}\label{def:generalized_flows} 
A family $ (D_t)_{t >0}$ (resp., $(E_t)_{t >0}$) of open (resp., closed)  subsets of $\R^n$ is a \emph{generalized super-flow (resp., sub-flow)} of the fractional mean curvature equation \eqref{eq:fmc equation} 
if for all $(t_0,x_0) \in (0,\infty) \times \R^n$,  $h,\, r>0$, and for all smooth functions $\varphi:(0,\infty) \times \R^n \to \R$ such that
\begin{enumerate}[start=1,label={(\roman*)}]
\item \label{item:i}
(Boundedness) for all $t\in[t_0,t_0+h]$, the set
 $$
\{ x \in \R^n : \varphi(t,x) >0\}\quad (\text{resp. }\{ x \in \R^n : \varphi(t,x) < 0\})
$$ is bounded and 
$$
\{ x \in \overline{B}(x_0,r)  : \varphi(t,x) >0\}\quad (\text{resp. }\{ x\in \overline{B}(x_0,r)  : \varphi(t,x) < 0\})
$$
is non-empty,
\item \label{item:ii}
 (Speed) 
there exists $\tau =\tau(\varphi)>0$ such that
\begin{align*}
\partial_t\varphi + F^*(x,\nabla\varphi, \{ z: \varphi(t,x+z) > \varphi(t,x) \} ) &\leq -\tau \quad \hbox{in}~[t_0, t_0 + h] \times \overline{B}(x_0,r),\\
(\text{resp., }\partial_t\varphi + F_\ast(x,\nabla\varphi, \{ z: \varphi(t,x+z) \geq \varphi(t,x) \} ) &\geq -\tau)
\end{align*}
\item  \label{item:iii}
(Non-degeneracy)
\[
\nabla\varphi \not= 0 \quad \hbox{on}~\{ (t,x) \in [t_0, t_0 + h] \times  \overline{B}(x_0,r) : \varphi(t,x) = 0\},
\] 
\item \label{item:iv}
(Initial condition)
\begin{align*}
\{x \in \R^n: \varphi(t_0,x) \geq 0\} &\subset D_{t_0}\\
(\text{resp., }\{x \in \R^n: \varphi(t_0,x) \leq 0\} &\subset \R^n \setminus E_{t_0}),
\end{align*}
\item \label{item:v}
(Boundary condition)  for all $t\in[t_0,t_0+h]$,
\begin{align*} 
\{x \in \R^n\setminus B(x_0,r) : \varphi(t,x) \geq 0\} &\subset  
D_t\\
(\text{resp., }\{x \in  \R^n\setminus B(x_0,r) : \varphi(t,x) \leq 0\} &\subset \R^n \setminus E_{t}),
\end{align*}
\end{enumerate}
it holds that
\begin{align*}
\{x \in B(x_0,r): \varphi(t_0+h,x)>0\}& \subset D_{t_0+h}\\
(\text{resp., }\{x \in B(x_0,r): \varphi(t_0+h,x)<0\}& \subset \R^n \setminus E_{t_0+h}).
\end{align*}
\end{defn}
In this paper, we will apply the abstract method developed in \cite{Barles-DaLio}-\cite{Barles-Souganidis} for generalized flows of local geometric equations and extended in \cite{Imbert09} to  flows of equation \eqref{eq:fmc equation}. Precisely, let $\Omega_0$ be the open set defined in \eqref{def:signed_distance_function}. We will show that there exist families of open sets $(D_t)_{t \geq 0}$ and $(E_t)_{t \geq 0}$ such that $(D_t)_{t \geq 0}$ and $((E_t)^c)_{t \geq 0}$ are generalized super and sub-flows of the fractional mean curvature equation \eqref{eq:fmc equation}, respectively. Moreover, $\Omega_0 \subset D_0$, $(\overline{\Omega}_0)^c \subset E_0$, and
\begin{align*}
    u^\ep(t,x) \to 1 \quad \text{if } x \in D_t \quad \text{and}\quad   u^\ep(t,x) \to 0 \quad \text{if } x \in E_t.
\end{align*}
Therefore, if $(^+\Omega_t, \Gamma_t, ^-\Omega_t)_{t \geq 0}$ denotes the level set evolution of $(\Omega_0, \Gamma_0, (\overline{\Omega_0})^c)$, then, by \cite[Corollary 1]{Imbert09}, 
 \[
{^+}\Omega_t \subset D_t \subset {^+}\Omega_t \cup \Gamma_t
\quad \hbox{and} \quad
{^-}\Omega_t \subset E_t \subset {^-}\Omega_t \cup \Gamma_t.
\]
In particular,  if the set $\Gamma_t$ doesn't  develop interior, 
 then 
 $$ {^+}\Omega_t=D_t \quad \hbox{and} \quad {^-}\Omega_t = E_t.$$
 Our main result, Theorem \ref{thm:main_result}, immediately follows.

\subsection{Extension of the signed distance function}

Recall that the signed distance function $d = d(t,x)$ associated to the front $\Gamma_t$ in \eqref{eq:triplets} is smooth in some neighborhood $Q_\rho = \{|d|<\rho\}$ of the front, provided $\Gamma_t$ is smooth. 
However, in general, $d$ is not smooth away from the front. 
Throughout the paper, we will use the following smooth extension of the distance function away from $\Gamma_t$. 

\begin{defn}[Extension of the signed distance function]\label{defn:extension}
For $t\in[t_0,t_0+h]$, let $\tilde{d}$ be the signed distance function from a bounded domain  $\Omega_t$ with boundary  $\Gamma_t$  and let   $\rho>0$ be such that $\tilde{d}(t,x)$  is smooth in 
\[
Q_{2\rho} := \{(t,x)\in [t_0,t_0+h] \times \R^n: |\tilde{d}(t,x)| < 2\rho\}.
\]
Let $\eta(t, x)$ be a smooth, bounded function such that 
\[
\eta = 1~\hbox{in}~\{|\tilde{d}|\leq \rho\}, \quad \eta = 0~\hbox{in}~\{|\tilde{d}|\geq 2\rho\},\quad 0 \leq \eta \leq 1.
\]
We extend $\tilde{d}(t,x)$ in the set $\{(t,x)\in [t_0,t_0+h] \times \R^n: |\tilde{d}(t,x)|\geq \rho\}$ with the smooth bounded function $d(t,x)$ given by
\[
d(t,x) 
= \begin{cases}
\tilde{d}(t,x) & \hbox{in}~Q_\rho=\{|\tilde{d}(t,x)|< \rho\}\\
\tilde{d}(t,x) \eta(t,x) + 2\rho(1-\eta(t,x)) & \hbox{in}~\{\rho \leq  \tilde{d}(t,x) \leq 2 \rho\}\\
\tilde{d}(t,x) \eta(t,x) -2\rho(1-\eta(t,x)) & \hbox{in}~\{-2\rho\leq\tilde{d}(t,x) \leq - \rho\}\\
2\rho &\hbox{in}~\{\tilde{d}(t,x)> 2\rho\}\\
-2\rho &\hbox{in}~\{\tilde{d}(t,x)< -2\rho\}.
\end{cases}
\]
Notice that, in $\{\rho \leq  \tilde{d} \leq  2\rho\}$, the function $d$ satisfies
\[
d = 2\rho + (\tilde{d}-2\rho)\eta \geq  2 \rho - \rho \eta \geq \rho,
\]
and, in $\{-2\rho \leq  \tilde{d} \leq -\rho\}$, the function $d$ satisfies
\[
d = -2\rho + (\tilde{d}+2\rho)\eta \leq - 2 \rho + \rho \eta \leq -\rho.
\]
\end{defn} 
\begin{rem}\label{Kdextensionrem}
By the definition of $\tilde d$,  we observe  that  for $(t,x)\in Q_\rho$, the following identities hold
 $$\partial_t d (t,x)=\partial_t\tilde d (t,x),\quad\nabla d (t,x)=\nabla \tilde d(t,x),$$
 $$\{z\,:\, d(t, x+z)> d(t, x)\}=\{z\,:\,\tilde d(t, x+z)>\tilde d(t, x)\}.$$
In particular, this implies 
$$\kappa[x, d(t,\cdot)]=\kappa[x, \tilde  d(t,\cdot)].$$
\end{rem}

\section{Heuristics}\label{sec:Heuristics}
Here, we give two formal computations relating to Theorem \ref{thm:main_result} and its proof. 
We use the notation $\simeq$ to denote equality up to adding terms that vanish as $\ep \to 0$.

\subsection{Derivation of the fractional mean curvature equation}\label{sec:heuristicspart1}
For the following formal computations, assume that the signed distance function $d(t,x)$ associated to $\Omega_t$ is smooth and   $\abs{\nabla d} = 1$.

Consider the following  ansatz for the solution of \eqref{eq:pde}-\eqref{initial_data}
\begin{equation} \label{eq:ansatz1}
u^{\ep}(t,x) \simeq \phi\( \frac{d(t,x)}{\ep}\),
\end{equation}
with $\phi$ the solution of \eqref{eq:standing wave}. 
Plugging the ansatz into \eqref{eq:pde}, the left-hand side gives
\begin{equation}\label{eq:antatz1-left}
\ep\partial_t u^{\ep}
 \simeq  \dot{\phi}\(\frac{d}{\ep}\) \partial_td.
\end{equation}
On the other hand, 
we use the equation for $\phi$  in  \eqref{eq:standing wave}
    to write 
the fractional Laplacian of the ansatz as
\begin{equation}\label{eq:antatz1-right}
\begin{aligned}
 \mathcal{I}^s _n [u^{\ep}]  
 &\simeq  \mathcal{I}^s _n\bigg[\phi\(\frac{d}{\ep}\)\bigg] \\&= \( \mathcal{I}^s_n\bigg[\phi\(\frac{d}{\ep}\)\bigg]  - \frac{C_{n,s}}{\ep^{2s}}\mathcal{I}^s_1[\phi]\(\frac{d}{\ep}\)  \)  +\frac{C_{n,s}}{\ep^{2s}}\mathcal{I}^{s}_1[\phi]\(\frac{d}{\ep}\)  \\
 & =\bar a_{\ep}
    +\frac{1}{\ep^{2s}}W'\(\phi\( \frac{d}{\ep}\)\)\\&\simeq \bar a_{\ep} 
       +\frac{1}{\ep^{2s}}W'\(u^{\ep}\),
\end{aligned}
\end{equation}
where 
\begin{equation}\label{oldaep}
\bar a_{\ep}:=\mathcal{I}^s_n\bigg[\phi\(\frac{d}{\ep}\)\bigg]  - \frac{C_{n,s}}{\ep^{2s}}\mathcal{I}^s_1[\phi]\(\frac{d}{\ep}\).
\end{equation}
By Lemma \ref{lem:1 to n facrional Laplacian}, applied to $v=\phi$ with $e=\nabla d(t,x)$, we have 
\begin{equation*}C_{n,s}\mathcal{I}^s_1[\phi]\(\frac{d(t,x)}{\ep}\)=\int_{\R^n}\(\phi\(\frac{d(t,x)}{\ep}+\nabla d(t,x)\cdot z\)-\phi\(\frac{d(t,x)}{\ep}\)\)\frac{dz}{|z|^{n+2s}}.\end{equation*}
Hence, since 
\begin{equation*}\mathcal{I}^s_n\bigg[\phi\(\frac{d(t,\cdot)}{\ep}\)\bigg] (x)=\int_{\R^n} \(\phi\(\frac{d(t,x+ z)}{\ep}\)-\phi\(\frac{d(t,x)}{\ep}\)\)\frac{dz}{|z|^{n+2s}},\end{equation*}
after a change of variables, we can write $\bar a_\ep$ as follows
\begin{align*}\label{aep:heuristicse}\bar a_\ep=\int_{\R^n} \(\phi\(\frac{d(t,x+ z)}{\ep}\)-\phi\(\frac{ d(t,x) + \nabla d(t,x)\cdot z}{\ep}\)\)\frac{dz}{|z|^{n+2s}}.
\end{align*}

Now freeze a point $(t,x)$ such that $x$ is  near the front $\Gamma_t$, and let 
 $\xi = d(t,x)/\ep$. Since $d$ grows linearly away from $\Gamma_t$, we can assume, at least formally, separation of scales. That is, assume that $\xi\in\R$ and $(t,x)$ are independent variables.

Define  
\begin{equation*}\label{oldaepbis}
A_{\ep}(\xi;t,x):=\frac{1}{\ep^{2s}}\int_{\R^n} \(\phi\(\xi+\frac{d(t,x+\ep z)-d(t,x)}{\ep}\)-\phi\(\xi +\nabla d(t,x)\cdot z\)\)\frac{dz}{|z|^{n+2s}}.
\end{equation*}
By making a change of variables, it is not hard to see that  
$$\bar a_\ep(t,x)=A_\ep\(\frac{d(t,x)}{\ep};t,x\).$$ 
The following convergence result  for $A_{\ep}$ is proven in \cite{PatriziVaughan}.
\begin{thm}\cite[Theorem 1.3]{PatriziVaughan}\label{aepconvergenceoldthm}
Assume $d(t,\cdot)$ smooth for $|d(t,x)| < 2\rho$, for   some $\rho>0$. Then, for $|d(t,x)| <\rho$, it holds that 
$$\lim_{\ep\to0}\int_{\R}  A_{\ep}(\xi;t,x)\dot \phi(\xi)\,d\xi=\kappa[x,d],$$
where $\kappa$ is defined in \eqref{kappade}. 
\end{thm}

Using this result and assuming separation of scales, we can now complete our formal derivation.  
Equalities \eqref{eq:antatz1-left} and \eqref{eq:antatz1-right} become, respectively,
\begin{equation}\label{eq:antatz1-left_bis}
\ep\partial_t u^{\ep}
 \simeq  \dot{\phi}(\xi) \partial_td, 
\end{equation}
and 
\begin{equation}\label{eq:antatz1-right_bis}
\begin{aligned}
 \mathcal{I}^s _n [u^{\ep}] \simeq  A_\ep(\xi;t,x)+\frac{1}{\ep^{2s}}W'\(u^{\ep}\).
 \end{aligned}
 \end{equation}
Since the ansatz $u^{\ep}$  approximates  a solution of  \eqref{eq:pde}, we can multiply the equation
 by $\dot{\phi}(\xi)$ and integrate over $\xi \in \R$, obtaining 
\begin{equation}\label{eq:freeze}
\int_{\R} \ep \partial_t u^{\ep} \dot{\phi}(\xi) \, d \xi \simeq \int_{\R} \( \mathcal{I}^s_n[ u^{\ep}]  - \frac{1}{\ep^{2s} } W'(u^\ep) \) \dot{\phi}(\xi) \, d \xi. 
\end{equation}
For convenience, we will consider the left and right-hand sides separately again. First, the left-hand side of \eqref{eq:freeze} with \eqref{eq:antatz1-left_bis} gives
\begin{align*}
\int_{\R} \ep \partial_t u^{\ep}\, \dot{\phi}(\xi) \, d \xi
 &\simeq \partial_td(t,x) \int_{\R} [\dot{\phi}\(\xi\)]^2  \, d \xi= c_0^{-1}\partial_td(t,x),
\end{align*}
where
\begin{equation}\label{def:c0}
c_0^{-1} := \int_\R [\dot{\phi}(\xi)]^2 \, d \xi.
\end{equation}
Then, we look at the right-hand side of \eqref{eq:freeze} with \eqref{eq:antatz1-right_bis} and write
\begin{align*}
& \int_{\R}  \left[\mathcal{I}^s_n [u^{\ep}] - \frac{1}{\ep^{2s} }W'(u^{\ep})\right] \dot{\phi}(\xi) \, d \xi\simeq \int_{\R} A_{\ep}\(\xi;t,x\) \dot{\phi}(\xi)\, d \xi. \end{align*}
Combining these estimates with Theorem \ref{aepconvergenceoldthm}, from \eqref{eq:freeze} we finally obtain 
\begin{align*}
\partial_td(t,x) \simeq c_0 \kappa[x,d]\quad  \text{near}~\Gamma_t,
 \end{align*}
that is the front $\Gamma_t$ moves by fractional mean curvature.

\subsection{ Derivation of equation \eqref{eq:linearized wave}}\label{Ansatzsection}

 It is actually  necessary to add a lower-order correction to  \eqref{eq:ansatz1} for the ansatz to solve  the fractional Allen--Cahn equation \eqref{eq:pde}. This was already observed in the one-dimensional case in  \cite[Section 3.1]{GonzalezMonneau}.  The correction involves  a function of the form  $\psi=\psi(\xi;t,x)$  belonging  to the space 
 $\{v\in H^s(\R)\,:\,\int_\R v(\xi)\dot\phi(\xi)\,d\xi=0\}$, which satisfies an equation in the variable $\xi$ involving the linearized operator $\mathcal{L}$ associated with \eqref{eq:standing wave} around $\phi$, defined by
\begin{equation*}\label{eq:linearized operator}
\mathcal{L}[\psi] = -C_{n,s}\mathcal{I}^s_1[\psi] + {W}''(\phi) \psi.
\end{equation*}
The derivation of this corrector for the case $s = \frac12$ is carried out in  \cite[Section 5]{PatriziVaughan2}. 
However, their approach doesn't apply here because, as explained in the introduction, our equation \eqref{eq:pde} is more singular than its $s = \frac12$ counterpart.
 
 In order to showcase the equation for the corrector, for $\sigma\in\R$,  let $v^\ep$ be the solution to 
\begin{equation}\label{eq:pde-sub}
\ep \partial_t v^{\ep} = \mathcal{I}^s_n v^{\ep}  -\frac{1}{\ep^{2s}} W'(v^\ep)- \sigma.
\end{equation}
Assume that $d(t,x)$ is smooth with $\abs{\nabla d}=1$,  and solves
\begin{equation}\label{eq:MC for d}
 \partial_td
  = c_0 \kappa[x,d] -\ c_0 \sigma,
\end{equation}
in a neighborhood of $\Gamma_t$. 

We now use a modified version of Theorem \ref{aepconvergenceoldthm}. 
 As shown in Lemma~\ref{thm: b_ep fmc},   the quantity $\bar a_\ep$ defined in \eqref{oldaep} satisfies
\begin{equation}\label{baraepconheuristic}\bar a_\ep(t,x)\simeq \kappa[x,d]\quad\text{if }d(t,x)\simeq0.
\end{equation}
Consider the new ansatz 
\[
v^{\ep}(t,x) \simeq \phi\( \frac{d(t,x)}{\ep}\) + \ep^{2s} \psi\( \frac{d(t,x)}{\ep};t,x\) + \ep^{2s} w(t,x),
\]
where  $\psi(\xi;t,x)$ and $w(t,x)$  are smooth functions  to  be determined.  

Plugging the ansatz into \eqref{eq:pde-sub}, the left-hand side gives
\begin{equation*}\label{eq:ansats2-left}
\begin{aligned}
\ep \partial_t v^{\ep}
 &\simeq \dot{\phi} \partial_td
  +  \ep^{2s}   (\dot{\psi} \partial_td+ \ep\partial_t \psi+ \ep\partial_t w). 
\end{aligned}
\end{equation*}
Assuming 
\begin{equation}\label{timederivest_heursec}
\ep^{2s}   (\dot{\psi} \partial_td+\ep \partial_t \psi+\ep \partial_t w)\simeq 0, 
\end{equation}
we obtain 
\begin{equation}\label{eq:ansats2-left}
\begin{aligned}
\ep \partial_t v^{\ep}
 \simeq  \dot{\phi}\partial_td.
\end{aligned}
\end{equation}
Next, we look at the right-hand side of \eqref{eq:pde-sub} for the ansatz. We compute 
\begin{equation*}\label{eq:ansats2-right1}
\begin{aligned}
\mathcal{I}^s_n[v^\ep]
 &\simeq \mathcal{I}^s_n\left[ \phi\( \frac{d(t,\cdot)}{\ep}\) \right]
  + \ep^{2s} \mathcal{I}^s_n\left[ \psi\( \frac{d(t,\cdot)}{\ep};t,\cdot\) \right]+\ep^{2s} \mathcal{I}^s_nw\\
 &= \mathcal{I}^s_n\left[ \phi\( \frac{d(t,\cdot)}{\ep}\) \right]- \frac{C_{n,s}}{\ep^{2s}} \mathcal{I}^s_1[\phi]\( \frac{d}{\ep}\) 
 + \frac{C_{n,s}}{\ep^{2s}} \mathcal{I}^s_1[\phi]\( \frac{d}{\ep}\)\\
 &\quad +  \(\ep^{2s} \mathcal{I}^s_n\left[ \psi \( \frac{d(t,\cdot)}{\ep};t,\cdot\) \right]-C_{n,s} \mathcal{I}^s_1[\psi]\( \frac{d}{\ep}\) \)
 +C_{n,s} \mathcal{I}^s_1[\psi]\( \frac{d}{\ep}\)+\ep^{2s} \mathcal{I}^s_nw.
\end{aligned}
\end{equation*}
Using  the equation for $\phi$ in \eqref{eq:standing wave},  definition \eqref{oldaep} and assuming that
\begin{equation}
\label{psilemmaheuristics}
\ep^{2s} \mathcal{I}^s_n\left[ \psi\( \frac{d(t,\cdot)}{\ep};t,\cdot\) \right]-C_{n,s} \mathcal{I}^s_1[\psi]\( \frac{d}{\ep}\) \simeq 0,\quad \ep^{2s} \mathcal{I}^s_nw\simeq 0, 
\end{equation}
we find
\begin{equation}\label{eq:ansats2-right1}
\mathcal{I}^s_n[v^\ep]\simeq \bar a_{\ep}
  +  \frac{1}{\ep^{2s}}W'\(\phi\( \frac{d}{\ep}\)\)+C_{n,s} \mathcal{I}^s_1[\psi]\( \frac{d}{\ep}\).
\end{equation}
Next, we do a Taylor expansion for $W'$ around $\phi(d/\ep)$ to estimate
\begin{equation}\label{eq:ansats2-right2}
\begin{aligned}
\frac{1}{\ep^{2s}}W'(v^{\ep})
 &\simeq \frac{1}{\ep^{2s}} \left[W'\(\phi\( \frac{d}{\ep}\)\)+ W''\(\phi\( \frac{d}{\ep}\)\) \(\ep^{2s} \psi\( \frac{d}{\ep};t,x\) + \ep^{2s}  w(t,x)\)\right].
\end{aligned}
\end{equation}
Plugging \eqref{eq:ansats2-left},  \eqref{eq:ansats2-right1} and \eqref{eq:ansats2-right2} into \eqref{eq:pde-sub}, we get
\begin{equation}\label{eq:first corrector eqn}
\begin{aligned}
 \dot{\phi}\( \frac{d}{\ep}\) \partial_td
  &\simeq \bar a_{\ep}(t,x) +C_{n,s} \mathcal{I}^n_1[\psi]\( \frac{d}{\ep}\)\\
 &\quad- W''\(\phi\( \frac{d}{\ep}\)\)   \psi\( \frac{d}{\ep};t,x\)-W''\( \phi \(\frac{d}{\ep}\)\)w(t,x)-\sigma.
\end{aligned}
\end{equation}
Rearranging terms and using \eqref{eq:MC for d} and \eqref{baraepconheuristic},  for  $d(t,x)\simeq 0$ we find that   $\psi$ must satisfy 
\begin{align*}
\mathcal{L}[\psi]\( \frac{d}{\ep}\)&=-C_{n,s} \mathcal{I}^s_1[\psi]\( \frac{d}{\ep}\)
+W''\(\phi\( \frac{d}{\ep}\)\)  \psi\( \frac{d}{\ep};t,x\) \\
 &\simeq \bar a_{\ep}(t,x)- \dot{\phi}\( \frac{d}{\ep}\) \partial_td
  -W''\( \phi \(\frac{d}{\ep}\)\)w(t,x)-\sigma\\
&\simeq  \bar a_{\ep}(t,x)
  - \dot{\phi}\( \frac{d}{\ep}\)c_0 \bar{a}_{\ep} (t,x)
  + c_0 \sigma  \dot{\phi}\(\frac{d}{\ep}\)-W''\( \phi \(\frac{d}{\ep}\)\)w(t,x)-\sigma.
\end{align*}
Evaluating the equation when $|d(t,x)|>> \ep$, and using the asymptotic behavior of  $\phi$ and $\dot\phi$ (see  \eqref{eq:asymptotics for phi}, \eqref{eq:asymptotics for phi dot}),  as well as the fact  that $W''(0)=W''(1)$, we find that if $\psi(\pm\infty)=0$,  then  $w$ must satisfy
$$W''\( 0\)w(t,x) = \bar a_{\ep}(t,x)-\sigma.$$
Substituting this into the earlier expression yields the following equation  for $\psi=\psi(\xi;t,x)$ in the variable $\xi$, 
\begin{equation}\label{psiequationheuristic}\mathcal{L}[\psi](\xi) =\left( c_0\dot{\phi}(\xi) 
  + \frac{W''\left( \phi (\xi)\right) - W''(0)}{ W''(0)} \right)\left(\sigma -\bar{a}_\ep (t,x) \right).\end{equation}
        Thus, the final ansatz near the front takes the form
\begin{equation}\label{finalansatz:heuristic}
v^{\ep}(t,x) \simeq \phi\( \frac{d(t,x)}{\ep}\) + \ep^{2s} \psi\( \frac{d(t,x)}{\ep};t,x\) + \frac{\ep^{2s}}{W''(0)}(\bar  a_\ep(t,x)-\sigma). 
\end{equation}
Since equation  \eqref{psiequationheuristic} is in the variable $\xi$, we see that we can write $\psi$  as following 
\begin{equation*}
\psi(\xi;t,x)=v_\ep(t,x)\tilde\psi(\xi), \end{equation*} 
where $\tilde\psi$ solves 
\begin{equation}\label{psiequationheuristiconlyxi}
\mathcal{L}[\tilde\psi](\xi) = c_0\dot{\phi}( \xi) + \frac{W''\left( \phi (\xi)\right) - W''(0)}{W''(0)},
       \end{equation}
and for $d(t,x)\simeq 0$,
$$c_0v_\ep(t,x)=-c_0(\bar{a}_\ep (t,x)-\sigma)\simeq -c_0(\kappa[x,d]-\sigma)=-\partial d_t$$
is, up to vanishing errors, the scalar velocity of the front. 
Existence of a solution to \eqref{psiequationheuristiconlyxi} such that $\tilde\psi(\pm\infty)=0$ is proven in  \cite[Theorem 9.1]{DipierroFigalliVald}. 
Hence, the  $\ep^{2s}$-correction to the original ansatz \eqref{eq:ansatz1} is given by  
 \begin{equation*}\label{psiseparationheuristin}v_\ep(t,x)\tilde\psi(\xi)-\frac{v_\ep(t,x)}{ W''(0)}.\end{equation*}
This decomposition, separating the fast variable $\xi$ and the slow variables $(t,x)$, is a new feature not present in the case $s=\frac12$, and is essential for proving the main result, Theorem~\ref{thm:main_result}.

To rigorously justify the approximations used (in particular, \eqref{timederivest_heursec} and \eqref{psilemmaheuristics}), precise and delicate estimates on the derivatives of $\bar a_\ep$ and $\psi$ are required. These are established in Lemmas~\ref{lem:a_estimates}, \ref{lem:psi-reg}, \ref{lem:ae psi estimate}, and Corollary~\ref{cor: fractional a_ep estimate}.

The final corrected ansatz  \eqref{finalansatz:heuristic} will be used to construct subsolutions and supersolutions to  \eqref{eq:pde}, depending on the sign of $\sigma$ (see  Section~\ref{barriersection}).
In order to construct global in space subsolutions and supersolutions, it will be necessary to modify the definition of $\psi$ and $\bar a_\ep $ to control additional error terms that arise far from the interface $\Gamma_t$. This will involve the auxiliary function $\mu$
 defined in \eqref{def:mu_definition} and the refined definition of 
 $\bar{a}_\ep$ in \eqref{aepsilondef}.

\section{Preliminary results on the fractional Laplacian}\label{sec:ae-frac}
In this section, we recall a few basic properties of the operator $\mathcal{I}^s_n $, which will be used later in the paper.
Let $u\in C^{0,1}(\R^n)$.  Then, 
for any $R>0$, we can write
\begin{align*}
\mathcal{I}^s_n u(x) 
 &= \int_{\{\abs{z}<R\}} \left( u(x+z) - u(x)\right) \,\frac{dz}{\abs{z}^{n+2s}}
  + \int_{\{\abs{z}>R\}} \left( u(x+z) - u(x)\right) \,\frac{dz}{\abs{z}^{n+2s}}.
\end{align*}
In particular,  both integrals above are finite, and we can bound $\mathcal{I}^s_n u$ as follows,

\begin{equation}\label{Iubound-C11_bis}
|\mathcal{I}^s_n u(x)|\leq C\left(\|Du\|_\infty R^{1-2s}+\frac{\|u\|_\infty}{R^{2s}}\right), 
\end{equation}
where $C>0$ is a constant depending only on $n$ and $s$. 
The estimate \eqref{Iubound-C11_bis} follows from the following lemma, whose proof is a straightforward computation in polar coordinates.

\begin{lem}\label{kernellemma3}
There exist $C_1,\, C_2>0$ such that for any $R>0$,
$$\int_{\{|z|<R\}}\frac{dz}{|z|^{n+2s-1}}= C_1 R^{1-2s}
\quad \hbox{and} \quad \int_{\{|z|>R\}}\frac{dz}{|z|^{n+2s}}= \frac{C_2}{R^{2s}}.$$
\end{lem}
We will frequently use Lemma \ref{kernellemma3}  throughout the paper without further reference.

We will also need   the following result, which provides a representation of  the one-dimensional fractional Laplacian of a function defined on $\mathbb{R}$ as an $n$-dimensional fractional Laplacian. 

\begin{lem} \cite[Lemma 3.2]{PatriziVaughan}\label{lem:1 to n facrional Laplacian}
For a vector $e \in \R^n$ and a function $v\in C^{1,1}(\R)$, let $v_e(x) = v(e\cdot x): \R^n \to \R$. Then,
\[
\mathcal{I}_n^s[v_e](x) =|e|^{2s} C_{n,s} \mathcal{I}_1^s[v](e\cdot x)
\]
where 
\begin{equation}\label{eq:Cns}
C_{n,s} =  \int_{\R^{n-1}} \frac{1}{(\abs{z}^2 + 1)^{\frac{n+2s}{2}} } \, dz.
\end{equation}
Consequently,
\begin{equation*}\label{eq:one to n}
|e|^{2s} C_{n,s} \mathcal{I}_1^s[v](\xi) = \int_{\R^n} \left( v(\xi + e \cdot z) - v(\xi)\right) \frac{dz}{\abs{z}^{n+2s}}, \quad \xi \in \R.
\end{equation*}
\end{lem}

\subsection{Properties of solutions to \eqref{eq:pde}}

Here, we state existence, uniqueness, and comparison principles for viscosity solutions to \eqref{eq:pde} for a fixed $\ep>0$. 

First, the following comparison principles can be found in \cite{JakobsenKarlsen} and will be used throughout the paper without reference. 
For the definition of viscosity subsolutions, supersolutions, and solutions, see also \cite{UsersGuide}.
For ease, we denote by $USC_b([t_0,t_0+h] \times \R^n)$ (resp.~$LSC_b([t_0,t_0+h] \times \R^n)$) the set of upper (resp.~lower) semicontinuous functions on $[t_0,t_0+h] \times \R^n$ which are bounded on $[t_0,t_0+h] \times \R^n$. 

\begin{prop}[Comparison principle in $\R^n$]
Fix $\ep>0$. 
If $u \in USC_b([t_0,t_0+h] \times \R^n)$ is a viscosity subsolution and $v \in LSC_b([t_0,t_0+h] \times \R^n)$ is a viscosity supersolution of \eqref{eq:pde} such that $u(t_0,\cdot) \leq v(t_0,\cdot)$
on $\R^n$, then $u \leq v$ on $[t_0,t_0+h] \times \R^n$. 
\end{prop}

\begin{prop}[Comparison principle in bounded domains]
Fix $\ep>0$ and let $\Omega \subset \R^n$ be a bounded domain. 
If $u \in USC_b([t_0,t_0+h] \times \R^n)$ is a viscosity subsolution and $v \in LSC_b([t_0,t_0+h] \times \R^n)$ is a viscosity supersolution of \eqref{eq:pde} such that $u(t_0,\cdot ) \leq v(t_0,\cdot )$ on $\R^n$ and $u \leq v$ on $[t_0, t_0+h] \times (\R^n \setminus \Omega)$, then $u \leq v$ on $[t_0,t_0+h] \times \R^n$. 
\end{prop}

Next, we prove existence and uniqueness of viscosity solutions. 

\begin{prop}[Existence and uniqueness]
Fix $\ep>0$ and let $u_0 \in C^{0,1}(\R^n)$. 
There exists a unique viscosity solution $u^\ep \in C([0,\infty) \times \R^n) \cap L^{\infty}([0,\infty)\times \R^n)$  to \eqref{eq:pde} with initial datum $u^\ep(0,x) = u_0(x)$. 
\end{prop}

\begin{proof}
Since $u_0 \in  C^{0,1}(\R^n)$,  by \eqref{Iubound-C11_bis} with $R=1$, the functions
\[
u^{\pm}(t,x) := u_0(x) \pm \frac{Ct}{\ep^{2s+1}}
\]
are supersolutions and subsolutions of \eqref{eq:pde}, respectively, if $C \geq 
\ep^{2s} c_{n,s}\|u_0\|_{C^{0,1}(\R^n)} + \|W'\|_{L^\infty(\R)}$. 
Noting that $u^\pm(0,x) = u_0(x)$, existence of a unique continuous viscosity  solution $u^\ep$ follows by Perron's method and the above comparison principle in $\R^n$. 
\end{proof}

\section{The phase transition, the corrector, and the auxiliary function}\label{sec:phi}

In this section, we will introduce the phase transition $\phi$ and the corrector $\psi$.
Along the way, we will also define the auxiliary function $\bar{a}_{\ep}$ and describe its connection to the fractional Laplacians of $\phi$ and $\phi(d(t,x)/\ep)$, as well as 
its relation to the fractional mean curvature operator. 

\subsection{The phase transition $\phi$}

Let $\phi$ be the solution to \eqref{eq:standing wave} and let $H(\xi)$ be the Heaviside function. 

\begin{lem}\label{lem:asymptotics}  
There is a unique solution $\phi \in C^{2, \beta}(\R)$ of \eqref{eq:standing wave}, for some $\beta\in (0,1)$. Moreover, there exists a constant $C>0$ 
and $\kappa>2s$ (depending only on~$s$) such that
\begin{equation}\label{eq:asymptotics for phi}
\abs{\phi(\xi) - H(\xi) +  \frac{C_{n,s}}{2sW''(0)} \frac{\xi}{|\xi|^{2s+1}}} \leq \frac{C}{\abs{\xi}^{\kappa}}, \quad \text{for }\abs{\xi} \geq 1,
\end{equation}
with $C_{n,s}$ as in \eqref{eq:Cns}, and
\begin{equation}\label{eq:asymptotics for phi dot}
\frac{1}{C\abs{\xi}^{2s+1}}\leq \dot{\phi}(\xi) \leq \frac{C}{\abs{\xi}^{2s+1}},\quad |\ddot{\phi}(\xi)| \leq \frac{C}{\abs{\xi}^{2s+1}}
 \quad \text{for }\abs{\xi} \geq 1.
\end{equation}
\end{lem}

\begin{proof}
The existence of a unique solution of \eqref{eq:standing wave} is established  in  \cite{CabreSola} for $s=\frac12$, and in 
\cite{CabreSire,PalatucciSavinValdinoci} for any $s\in(0,1)$. 
The estimate \eqref{eq:asymptotics for phi}, as well as the estimate on $\dot\phi$  in \eqref{eq:asymptotics for phi dot}, are proven in \cite{GonzalezMonneau} for  $s=\frac{1}{2}$,  and 
in \cite{DipierroFigalliVald} and \cite{DipierroPalatucciValdinoci}, respectively, when $s\in\left(0,\frac{1}{2}\right)$ and $s\in\left(\frac{1}{2},1\right)$. 
Finally, the estimate on $\ddot\phi$  in \eqref{eq:asymptotics for phi dot} is  established in \cite{MonneauPatrizi2}. 
\end{proof}


\subsection{The auxiliary function $\bar{a}_{\ep}$}\label{aepsubsection}

We now introduce the auxiliary function $\bar{a}_{\ep}$,  which will play a crucial role in our analysis. 
Let $\Omega_t$ be a bounded domain with smooth boundary $\Gamma_t$, for $t\in[t_0,t_0+h]$.
Throughout this section, let $d=d(t,x)$ denote the smooth extension of the signed distance function $\tilde{d}$ to $\Omega_t$ outside of $Q_\rho$ (see Definition~\ref{defn:extension}). We also introduce a new parameter $R > 1$, which will be chosen later. 
Define the following auxiliary functions for  $(t,x) \in [t_0,t_0+h] \times \R^n$, 
\begin{equation}\label{b_epsilon}
   \bar  b_\ep[d](t,x): = \int_{\{|z|<R\}} \left[ \phi\left(\frac{d(t,x+ z)}{\ep}\right) - \phi\left( \frac{d(t,x)+ \nabla d(t,x) \cdot z}{\ep} \right) \right] \frac{dz}{\abs{z}^{n+2s}},
\end{equation}
\begin{equation}\label{c_epsilon}
   \bar  c_\ep[d](t,x) :=\frac{1}{\ep^{2s}}\left[ \left(|\nabla d(t,x)|^{2} + \ep^{2+\frac{2s}{1-2s}} \right)^s -1 \right]W' \left( \phi \left( \frac{d(t,x)}{\ep} \right) \right).
\end{equation}
By Lemma~\ref{kernellemma3}, and due to the regularity of $\phi$ and $d$, the integral in \eqref{b_epsilon} is well defined. 
We then define the function $\bar a_\ep= \bar a_\ep[d]( t,x)$,   by
\begin{equation}\label{aepsilondef}
\bar a_{\ep}:= \bar b_\ep+\bar  c_\ep. 
\end{equation}


The following result states that for points $x$ sufficiently close to $\Gamma_t$, $\bar{a}_\ep[d](t,x)$ approximates  the fractional mean curvature of the smooth set $\{d(t,\cdot)>d(t,x)\}$ at the point $x$. The proof, which is delayed until Section \ref{sec:proof of b_ep fmc}, follows the proof of   \cite[Theorem 3.1]{PatriziVaughan}. 
\begin{lem}\label{thm: b_ep fmc} 
For $t\in[t_0,t_0+h]$, let $\Omega_t$ be a bounded domain with smooth boundary $\Gamma_t$. 
Let $d$ be as in Definition \ref{defn:extension} and $\kappa[t,d]$ as in \eqref{kappade}. There exists $\delta_0>0$ such that if  $0<\delta<\delta_0$,   
 and $|d(t,x)|< \delta$, then 
\begin{align*}
\bar{a}_\ep[d](t,x) =\kappa[x,d] + o_\ep(1) + o_\delta(1)+ O(R^{-2s}). 
\end{align*}   
\end{lem}
Moreover, 
$\bar{a}_\ep$ is, up to small errors,  the difference between an $n$-dimensional and a $1$-dimensional fractional Laplacian, as stated in the following lemma
whose proof can be found in Section \ref{sec: proof of a_ep and frac laplacians}.
\begin{lem}\label{lem:a_ep and frac laplacians}
For all  $(t,x) \in [t_0,t_0+h] \times \mathbb{R}^n$,
\begin{equation}
   \mathcal{I}_n^s \left[ \phi \left( \frac{d(t, \cdot)}{\ep} \right) \right](x) - \frac{C_{n,s}}{\ep^{2s}}\mathcal{I}_1^s \phi \left( \frac{d(t, x)}{\ep}\right)= \bar a_\ep[d](t,x) + O(R^{-2s}) + o_\ep(1). 
\end{equation} 
\end{lem}

Next, we establish estimates for $\bar a_\ep$ and its  derivatives. The following estimates hold, with proofs provided in Section   \ref{sec:proof of a_estimates}. 

\begin{lem}\label{lem:a_estimates}
There exists $C>0$  such that for all  $(t,x) \in [t_0,t_0+h] \times \mathbb{R}^n$, 
    \begin{equation}\label{a_near}
    |\bar a_\ep[d](t,x)| \le  C, 
    \end{equation}   
 and 
    \begin{equation}\label{partial_a_near}
        |\nabla_x \bar a_\ep[d](t,x)|, |\partial_t \bar a_\ep[d](t,x)|   = \ep^{-1}o_{\ep}(1)R.
    \end{equation}
\end{lem}

\begin{cor}\label{cor: fractional a_ep estimate}
For all  $(t,x) \in [t_0,t_0+h] \times \mathbb{R}^n$, 
\begin{equation}\label{fractional_laplacian_a_ep_near}
    \left|\mathcal{I}_n^s[\bar a_\ep] \right|  =\ep^{-2s}  o_{\ep}(1)R.
\end{equation}
\end{cor}
\begin{proof}
    Let $\alpha > 0$, to be determined.  Then
    \begin{align*}
        \mathcal{I}_n^s[\bar a_\ep] &= \int_{\mathbb{R}^n} (\bar a_\ep[d](x+z) - \bar a_\ep[d](x)) \frac{dz}{|z|^{n+2s}} \\
        & = \int_{ \{ |z|< \alpha \} } (\ldots) + \int_{ \{ |z|> \alpha \} } (\ldots) \\
        & =: I + II.
    \end{align*}
    We have
    \begin{align*}
        |I| \le \norm[\infty]{ \nabla_x \bar a_\ep}  \int_{ \{ |z|< \alpha \} }  \frac{dz}{|z|^{n+2s-1}} \le C \norm[\infty]{ \nabla_x \bar a_\ep}  \alpha^{1-2s},
    \end{align*}
    and 
    \begin{align*}
        |II| \le C \norm[\infty]{\bar a_\ep} \int_{ \{ |z|> \alpha \} } \frac{dz}{|z|^{n+2s}} \le C \norm[\infty]{\bar a_\ep} \alpha^{-2s}. 
    \end{align*}
 By  \eqref{partial_a_near}  there exists $\tau=o_{\ep}(1)$, such that 
      $\norm[\infty]{ \nabla_x \bar a_\ep} \leq  \ep^{-1}\tau R$. Choosing $\alpha=\ep/\tau$ and using also \eqref{a_near}, we get   
    \begin{align*}
    \left|\mathcal{I}_n^s[\bar a_\ep] \right| \le C \ep^{-1} \tau R \frac{\ep^{1-2s}}{\tau^{1-2s}}+ C  \frac{\tau^{2s}}{\ep^{2s}} \leq C \ep^{-2s} \tau^{2s}R.
\end{align*}
Estimate  \eqref{fractional_laplacian_a_ep_near}, follows. 
 
  \end{proof}

 \subsection{The corrector $\psi$}\label{sec:corrector}

We now introduce two additional small parameters to be chosen later: 
$0<\delta<1$ and $\sigma\in(-1,1)$. 
Let $\mu$ be a smooth function such that
 \begin{equation}\label{def:mu_definition}\begin{split}
\mu[d](t,x) = 
\begin{cases}
\sigma, & |d(t,x)| \le\delta, \\[8pt]
\dfrac{\sigma}{\delta^{2s}}, & |d(t,x)| \ge 2 \delta,
\end{cases}\\[5pt]
|\sigma|\leq\operatorname{sgn}(\sigma)\mu[d](t,x)  \leq \dfrac{|\sigma|}{\delta^{2s}},\quad \delta  < |d(t,x)| <2 \delta,
\end{split}
\end{equation}
and 
\begin{equation}\label{mu_estimates}
    |\partial_t\mu[d](t,x)|, \, |\nabla_x \mu[d](t,x)| \le \frac{C}{\delta^{2s+1}}.
    \end{equation}
\begin{lem}\label{lem:mu_estimates}
There exist $C>0$, such that for all  $(t,x) \in [t_0,t_0+h] \times \mathbb{R}^n$,
    \begin{equation*}\left| \mathcal{I}_n^s \left[ \mu[d](t, \cdot) \right] (x)\right| 
    \le \frac{C}{\delta^{4s}}.
\end{equation*}
\end{lem}
\begin{proof}
    The estimate on  $ \mathcal{I}_n^s \left[ \mu[d](t, \cdot)\right] $ follows from \eqref{def:mu_definition} and \eqref{mu_estimates}
    by a similar argument as in Corollary~\ref{cor: fractional a_ep estimate}.
\end{proof}

The linearized operator $\mathcal{L}$ associated to \eqref{eq:standing wave} around $\phi$ is given by
\begin{equation}\label{eq:linearized operator}
\mathcal{L}[\psi] = -C_{n,s}\mathcal{I}^s_1[\psi] + {W}''(\phi) \psi,
\end{equation}
with $C_{n,s}$ as in \eqref{eq:Cns}.
In the constructions of barriers, we will need the corrector  $\psi = \psi(\xi; t,x)$ that solves
\begin{equation}\label{eq:linearized wave}
\begin{cases}
\displaystyle{
\mathcal{L}[\psi](\xi) 
 =   \left( c_0 \dot{\phi}( \xi) 
  + \frac{W''\left( \phi (\xi)\right) - W''(0)}{W''(0)} \right)\left[\mu[d](t,x) -\bar{a}_\ep[d](t,x) \right],
  } &\xi \in \R\\
\psi(\pm\infty;t,x) = 0,
\end{cases}
\end{equation}
where $c_0$ is given by \eqref{def:c0} and the functions $\bar a_\ep$ and $\mu$ are defined in \eqref{aepsilondef} and \eqref{def:mu_definition}, respectively. 

 Note that $\psi$ depends on $(t,x)$ through the dependence of the function $d$, which appears on the right-hand side 
 of \eqref{eq:linearized wave}. Moreover, although not explicitly indicated, $\psi$ also depends on the parameters $\ep$ and $R$ through the function 
 $\bar a_\ep$ and $\delta$,  $\sigma$ through the function  $\mu$. 

\begin{lem}\label{lem:psi-reg}
There is a  solution $\psi = \psi(\xi;t,x)\in C_\xi^{1,\beta}(\R)$ to \eqref{eq:linearized wave} 
for some  $\beta\in(0,1)$,  and $C>0$ such that, for all  $0<\ep,\,\delta<1$, $\sigma \in (-1,1)$, $R>1$, and 
 $(\xi, t,x)\in \R\times [t_0,t_0+h] \times\R^n$, the following holds.\\
 If $|d(t,x)| < \delta$, then
    \begin{equation}\label{psi_near}
    |\psi(\xi;t,x)|,|\dot{\psi}(\xi;t,x)|  \le   C,
    \end{equation}
    \begin{equation}\label{partial_psi_near}
        |\nabla_x \psi(\xi;t,x)|, |\partial_t \psi(\xi;t,x)| = \ep^{-1} o_\ep(1) R.
    \end{equation}
If $|d(t,x)| \geq\delta $, then
    \begin{equation}\label{psi_far}
    |\psi(\xi;t,x)|, |\dot{\psi}(\xi;t,x)| \le \frac{C}{\delta^{2s}(1+|\xi|^{2s})},
    \end{equation}
    and 
    \begin{equation}\label{partial_psi_far}
        |\nabla_x \psi(\xi;t,x)|, |\partial_t \psi(\xi;t,x)| \leq \(  \frac{o_\ep(1) R}{\ep}+\frac{C}{\delta^{2s+1}}\)\frac{1}{1+|\xi|^{2s}}.
    \end{equation}
 
\end{lem}

\begin{proof}
Under assumptions \eqref{eq:W} on the potential 
$W$, it is shown in \cite{DipierroFigalliVald}*{Theorem 9.1} that there exists a function  $\tilde\psi=\tilde \psi(\xi)\in C^{1,\beta}(\R)$, for some $\beta\in (0,1)$, solving 
$$\begin{cases}
\displaystyle{
\mathcal{L}[\tilde\psi] (\xi)
 =    c_0 \dot{\phi}(\xi) 
  + \frac{W''\left( \phi (\xi)\right) - W''(0)}{W''(0)} ,
  } &\xi \in \R\\
\tilde\psi(\pm\infty) = 0.
\end{cases}
$$
Moreover, \cite[Lemma 3.2]{MonneauPatrizi2} shows that there exists a constant 
$C>0$ such that 

\begin{equation}\label{tildepsiest} 
|\tilde\psi(\xi)|,\,|\dot{\tilde\psi}(\xi)|\le \frac{C}{1+|\xi|^{2s}}\quad\text{for all }\xi\in\R. 
\end{equation}
We define
$$\psi(\xi;t,x):=\tilde \psi(\xi)\left[\mu[d](t,x) -\bar{a}_\ep[d](t,x) \right].$$
Then $\psi\in C_\xi^{1,\beta}(\R)$ is solution to  \eqref{eq:linearized wave}. Moreover, recalling the definition of $\mu$ in \eqref{def:mu_definition}, 
and using Lemma \ref{lem:a_estimates} together with estimates \eqref{mu_estimates} and  \eqref{tildepsiest}, we obtain the bounds stated in  \eqref{psi_near}-\eqref{partial_psi_far}.
\end{proof}

We conclude this section with the following estimate for the difference between the $n$- and the $1$-dimensional fractional Laplacians for the function $\psi$. 
The proof of the lemma is given in Section \ref{sec:proof of ae psi estimate}.
\begin{lem}\label{lem:ae psi estimate}
 Assume $\ep/\delta^2=o_\ep(1)$. Then  for all $(t,x) \in [t_0,t_0+h] \times \mathbb{R}^n$,
 \begin{align*}
\abs{\ep^{2s} \mathcal{I}_n^s\left[ \psi \( \frac{d(t,\cdot)}{\ep};t,\cdot\) \right](x) - C_{n,s} \mathcal{I}_1^s[\psi\(\cdot;t,x\)]\(\frac{d(t,x)}{\ep}\)}= Ro_\ep(1). 
\end{align*}
\end{lem}

\section{Constructions of barriers}\label{barriersection}

We now construct  local and global  strict subsolutions (supersolutions) to \eqref{eq:pde} needed for the proof of Theorem \ref{thm:main_result}. 
We will focus on the construction of subsolutions, since the construction of supersolutions is analogous. 
We will start with the global ones.

 \subsection{Global subsolutions}

Fix $t_0 \in (0,\infty)$ and $h>0$. For $t\in[t_0,t_0+h]$, let $\Omega_t$ be a bounded open set with boundary $\Gamma_t= \partial \Omega_t$. Let $\tilde{d}(t,x)$ be the signed distance function associated to the set $ \Omega_t $, then $\Gamma_t = \{ x \in \R^n : \tilde{d}(t,x) = 0\}$. Assume that there exists  $\rho>0$ such that, $\tilde{d}(t,x)$ is smooth in the set 
\begin{equation}\label{eq:Q2rho}
Q_{2\rho} := \{ (t,x) \in[t_0,t_0+h]\times\R^n: |\tilde{d}(t,x)| < 2\rho\},
\end{equation}
and let $d$ be the smooth, bounded extension of $\tilde{d}$ outside of $Q_{\rho}$, as defined in Definition \ref{defn:extension}. Assume in addition that there exists $\sigma>0$ such that 
\begin{equation}\label{eq:fmc for d}
\partial_td \leq c_0 \kappa[x,d(t, \cdot)] - c_0 \sigma \quad \hbox{in}~Q_\rho.
\end{equation}
Let $\alpha,\,\tilde{\sigma}>0$ be defined by
\begin{equation}\label{alphasigma}\alpha:= W''(0),\quad \tilde{\sigma}:=\frac{\sigma}{\alpha}.\end{equation}
By eventually making $\sigma$ smaller, we may assume $\tilde\sigma<\rho/2$. We define the smooth barrier $v^{\ep}(t,x)$ by 
\begin{equation}\label{eq:barrier defn}
v^{\ep}(t,x) =\phi\left( \frac{d(t,x)- \tilde{\sigma} }{\ep}\right) +\ep^{2s} \psi \left( \frac{d(t,x)-\tilde{\sigma}}{\ep};t,x\right) + \frac{\ep^{2s}}{\alpha} \left(\bar{a}_\ep\left[d - \tilde{\sigma} \right](t,x) - \mu[d-\tilde{\sigma}](t,x)\right),
\end{equation}
where $\phi(\xi)$ is the solution to \eqref{eq:standing wave}, and $\psi(\xi; t, x)$ solves \eqref{eq:linearized wave} for the distance function $d(t, x) - \tilde{\sigma}$, with $\tilde{\sigma}$ defined in \eqref{alphasigma}.
Recall that $\psi$ also depends on the parameters $\ep$, $R$, $\delta$ and $\sigma$ through the functions $\bar{a}_\varepsilon$ and $\mu$, which are defined in 
\eqref{aepsilondef} and \eqref{def:mu_definition}, respectively, and appear on the right-hand side of \eqref{eq:linearized wave}. In the definition of $\mu$, the parameter  $\sigma>0$ is chosen as in   \eqref{eq:fmc for d}, and we assume the following condition on $\delta$:
\begin{equation}\label{delta:def}\delta=o_\ep(1),\quad\frac{\ep}{\delta^2}=o_\ep(1).
\end{equation}
\begin{lem}[Global subsolutions to \eqref{eq:pde}] \label{lem:barrier}
 Assume \eqref{eq:fmc for d} with $c_0$ as in \eqref{def:c0}. Let $v^{\ep}$ be defined as in \eqref{eq:barrier defn} 
 with $0<\tilde\sigma<\rho/2$, $R >1$ and $\delta$ satisfying \eqref{delta:def}. 
  Then there exists $R_0=R_0(\sigma)$ and $\ep_0=\ep_0(\sigma)>0$ such that for all $R>R_0$ and $0<\ep<\ep_0$, $v^{\ep}$ satisfies 
\begin{equation}\label{eq:pde sub}
\ep \partial_t v^{\ep} - \mathcal{I}_n^s[v^{\ep} ] +\frac{1}{\ep^{2s}} W'(v^{\ep} ) \leq -\frac{\sigma}{2} \quad \hbox{in}~[t_0,t_0+h] \times \R^n.
\end{equation}
Moreover, there is a constant $\tilde{C}>0$ such that for all $0<\ep < \ep_0$,
\begin{equation}\label{eq:barrier_est}
v^\ep(t,x) \geq 1 - \tilde C \tilde \sigma^{2s}\frac{\ep^{2s} }{\delta^{2s} }\quad \hbox{in}~\left\{ (t,x) \in [t_0,t_0+h] \times \R^n: d(t,x) -\tilde{\sigma} \geq \frac{\delta }{\tilde{\sigma}}\right\}.
\end{equation} 
\end{lem}

\begin{proof} 
For convenience, and with a slight abuse of notation, we shall use the following notation throughout the proof:
\begin{equation*}\label{notation_for_main_lemma}
\begin{aligned}
\phi & := \phi\left( \frac{d(t,x)- \tilde{\sigma} }{\ep}\right) \\
\psi & := \psi\left( \frac{d(t,x)- \tilde{\sigma} }{\ep}; t,x\right)\\
\bar a_\ep& := \bar a_\ep\left[d- \tilde{\sigma} \right](t,x) \\
  \mu& := \mu \left[d- \tilde{\sigma} \right](t,x).
\end{aligned}
\end{equation*}
We note that it will be important for the reader to remember the dependence of $\psi$ on the variables $t,x$, and $\xi=(d(t,x)-\tilde\sigma)/\ep$
when taking derivatives in $t$ and $x$. We begin by computing the time derivative of $v^\ep$ at $(t,x)$, which is given by
\begin{align*}
\ep \partial_t v^\ep(t,x) = &\dot{\phi} \partial_t d(t,x)+ \ep^{2s} \dot{\psi} \partial_td(t,x) + \ep^{2s+1} \partial_t \psi + \frac{\ep^{2s+1}}{\alpha} \partial_t \bar{a}_\ep - \frac{\ep^{2s+1}}{\alpha} \partial_t \mu.
\end{align*}
By Lemmas  \ref{lem:a_estimates} and \ref{lem:psi-reg},  and \eqref{mu_estimates},  we have
\begin{align*}
 \ep^{2s} \dot{\psi} \partial_td(t,x) + \ep^{2s+1} \partial_t \psi + \frac{\ep^{2s+1}}{\alpha} \partial_t \bar{a}_\ep- \frac{\ep^{2s+1}}{\alpha} \partial_t \mu&= O\(\frac{\ep^{2s}}{\delta^{2s}}\)+
 O\( \ep^{2s}Ro_\ep(1)\)+O\(\frac{\ep^{2s+1}}{\delta^{2s+1}}\)\\&
 =Ro_\ep(1),
\end{align*}
where we used that  $\ep/\delta^2=o_\ep(1)$ in the last equality. Therefore,  
we have 
\begin{equation}\label{eq:time_derivative_of_v_ep}
\ep \partial_t v^\ep (t,x)= \dot{\phi} \partial_t d (t,x)+Ro_\ep(1).
\end{equation}
Next, we consider the nonlocal term. We compute
\begin{align*}
\mathcal{I}_n^s[v^\ep(t,\cdot)](x) =&  \mathcal{I}_n^s\left[\phi  \right](x) + \ep^{2s} \mathcal{I}_n^s \left[ \psi\right](x) + \frac{\ep^{2s}}{\alpha} \mathcal{I}_n^s[\bar{a}_\ep]
(x) - \frac{\ep^{2s}}{\alpha} \mathcal{I}_n^s[\mu](x).
\end{align*}
Using that $\phi$ satisfies \eqref{eq:standing wave} and  Lemma \ref{lem:a_ep and frac laplacians},   we get
\begin{align*}
\mathcal{I}_n^s\left[\phi \right](x) &=  \mathcal{I}_n^s\left[\phi \right](x) 
 - \frac{C_{n,s}}{\ep^{2s}} \mathcal{I}_1^s[\phi]\left( \frac{d(t,\cdot) -\tilde{\sigma}}{\ep} \right) +\frac{C_{n,s}}{\ep^{2s}} \mathcal{I}_1^s[\phi]\left( \frac{d(t,\cdot) -\tilde{\sigma}}{\ep} \right) \\
& = \bar a_\ep + O(R^{-2s}) +  o_\ep(1) + \frac{1}{\ep^{2s}}W'(\phi).
\end{align*}
Recalling  \eqref{eq:linearized operator} and that $\alpha = W''(0)$, and using that $\psi$ solves  \eqref{eq:linearized wave}, we find that 
\begin{align*}
\ep^{2s}\mathcal{I}_n^s[\psi](x) =&\ep^{2s}\mathcal{I}_n^s[\psi](x) - C_{n,s}\mathcal{I}^s_1[\psi]\left( \frac{d(t,x) - \tilde{\sigma}}{\ep} \right) + {W}''(\phi) \psi - \mathcal{L}[\psi]\left( \frac{d(t,x) - \tilde{\sigma}}{\ep} \right)\\
=& \ep^{2s}\mathcal{I}_n^s[\psi](x) - C_{n,s}\mathcal{I}^s_1[\psi]\left( \frac{d(t,x) - \tilde{\sigma}}{\ep} \right) + {W}''(\phi) \psi 
\\&- \left(c_0 \dot{\phi} + \frac{W''\left( \phi \right) }{\alpha} -1\right)  [\mu-\bar{a}_{\ep}]. 
\end{align*}
Since $\ep/\delta^2=o_\ep(1)$, we can apply Lemma \ref{lem:ae psi estimate} and thereby obtain
\begin{align*}
\ep^{2s}\mathcal{I}_n^s[\psi](x) =  {W}''(\phi) \psi -  \left(c_0 \dot{\phi} + \frac{W''\left( \phi \right)}{\alpha}-1 \right)  [\mu-\bar{a}_{\ep}]+ Ro_\ep(1).
\end{align*}
Using Corollary \ref{cor: fractional a_ep estimate}, we also get
\begin{align*}
\ep^{2s}\mathcal{I}_n^s[\bar{a}_\ep](x) = Ro_\ep(1).
\end{align*}
Finally, by Lemma \ref{lem:mu_estimates}, and using again that $\ep/\delta^2=o_\ep(1)$, we have
\begin{align*}
\ep^{2s}\mathcal{I}_n^s[\mu](x) =  O\left(\frac{\ep^{2s}}{\delta^{4s}}\right)=o_\ep(1).
\end{align*}
Therefore, the fractional Laplacian of $v^\ep$ can be written as
\begin{equation}\label{eq:fractional_Laplacian_of_v_ep}
\begin{split}
 \mathcal{I}_n^s[v^\ep(t,\cdot)](x) =\, & \bar{a}_\ep + \frac{1}{\ep^{2s}}W'(\phi) +   {W}''(\phi) \psi -  \left(c_0 \dot{\phi} + \frac{W''\left( \phi \right)}{\alpha} -1\right)  [\mu-\bar{a}_{\ep}]\\& +  O(R^{-2s}) +Ro_\ep(1). 
\end{split}
\end{equation}
Next, we  compute $W'(v^\ep(t,x))$. 
To this end, we perform a Taylor expansion of $W'$ around $\phi$,  which yields
\begin{align*}
W'(v^\ep(t,x)) =& W'(\phi) + \ep^{2s}W''(\phi) \left( \psi + \frac{\bar{a}_\ep}{\alpha} - \frac{\mu}{\alpha}  \right) +
 O \left( \ep^{4s}\left( \psi + \frac{\bar{a}_\ep }{\alpha} - \frac{\mu}{\alpha}  \right)^2 \right).
\end{align*}
By  the estimates for $\bar{a}_\ep$ and $\psi$ in Lemmas \ref{lem:a_estimates} and  \ref{lem:psi-reg}, respectively, and recalling the definition of $\mu$ in \eqref{def:mu_definition},  we get
\begin{equation}\label{eq:W'_of_v_ep}
\begin{split}\frac{1}{\ep^{2s}}W'(v^\ep(t,x))& = \frac{1}{\ep^{2s}}W'(\phi) + W''(\phi) \left( \psi + \frac{\bar{a}_\ep}{\alpha} - \frac{\mu}{\alpha}  \right) +O\left( \frac{\ep^{2s}}{\delta^{4s}} \right)\\&
=\frac{1}{\ep^{2s}}W'(\phi) + W''(\phi) \left( \psi + \frac{\bar{a}_\ep}{\alpha} - \frac{\mu}{\alpha}  \right) +o_\ep(1). 
\end{split}
\end{equation}
Combining \eqref{eq:time_derivative_of_v_ep}, \eqref{eq:fractional_Laplacian_of_v_ep} and \eqref{eq:W'_of_v_ep},  we get
\begin{align*}
\mathcal{J}[v^\ep](t,x):= & \ep \partial_t v^{\ep} (t,x)- \mathcal{I}_n^s[v^{\ep}(t,\cdot) ](x) +\frac{1}{\ep^{2s}} W'(v^{\ep}(t,x))  \\
= & \dot{\phi} \partial_t d(t,x)  \\&
 -\bar{a}_\ep - \frac{1}{\ep^{2s}}W'(\phi) -  {W}''(\phi) \psi +  \left(c_0 \dot{\phi} + \frac{W''\left( \phi \right) }{\alpha} -1\right)  [\mu-\bar{a}_{\ep}]\\&
+ \frac{1}{\ep^{2s}}W'(\phi) + W''(\phi) \left( \psi + \frac{\bar{a}_\ep}{\alpha} - \frac{\mu}{\alpha}  \right) \\&
+ O(R^{-2s})+Ro_\ep(1). 
\end{align*}
Grouping and canceling terms, we obtain
\begin{equation}\label{eq:final_J_of_v_ep}
\begin{split}
\mathcal{J}[v^\ep](t,x) = \dot{\phi}[\,\partial_t d(t,x)
        - c_0 \,\bar{a}_\ep
        + c_0\, \mu]  - \mu+ O(R^{-2s})+Ro_\ep(1).
\end{split}
\end{equation}
We now consider two cases:  $|d(t,x)-\tilde{\sigma}| < \delta$ and $|d(t,x)-\tilde{\sigma}| \geq \delta$.
\medskip

\noindent
\textbf{Case 1:} $|d(t,x)-\tilde{\sigma}|<\delta$. 

Since $\tilde\sigma<\rho/2$, the level set  $\{d(t,\cdot)=\tilde\sigma\}$ is a smooth surface,  and we are in a position to apply Lemma  \ref{thm: b_ep fmc} to $d-\tilde{\sigma}$ which is 
a   smooth extension of its distance function.  Recalling the definition of $\mu$ in \eqref{def:mu_definition}, we also  have that 
 $\mu = \sigma$.  Using that $\dot{\phi} \geq 0$, that $d$ solves \eqref{eq:fmc for d},   Lemma \ref{thm: b_ep fmc}, and that $\delta=o_\ep(1)$, we obtain
\begin{align*}
     \dot{\phi} \left[ \partial_t d(t,x) - c_0 \bar{a}_\ep + c_0 \sigma \right] = & \dot{\phi} \left( \partial_t d(t,x) - c_0 \kappa[x, d(t,\cdot)-\tilde{\sigma}]+ c_0 \sigma+O(R^{-2s})+o_\ep(1)\right) \\
      =&\dot{\phi} \left( \partial_t d(t,x) - c_0 \kappa[x, d(t,\cdot)]+ c_0 \sigma+O(R^{-2s})+o_\ep(1)\right)\\
      \le & \dot{\phi} \left(O(R^{-2s})+o_\ep(1) \right).
\end{align*}
Thus, by \eqref{eq:final_J_of_v_ep}, 
\begin{align*}
\mathcal{J}[v^\ep](t,x) &\le Ro_\ep(1) + O(R^{-2s}) - \sigma.
\end{align*}
Choosing $R_0=R_0(\sigma)$ sufficiently large so that for all $R>R_0$ we have 
 $ |O(R^{-2s}) |\leq \sigma/4$,  then selecting  $\ep_0=\ep_0(R_0,\sigma)=\ep_0(\sigma)$ small enough so that for all $0< \ep< \ep_0$, we have   $|Ro_\ep(1)|\leq \sigma/4$, we obtain
\begin{align*}
\mathcal{J}[v^\ep](t,x) \le -\frac{\sigma}{2}.
\end{align*}
This proves \eqref{eq:pde sub} for Case 1.

\medskip

\noindent
\textbf{Case 2:} $|d(t,x)-\tilde{\sigma}| \geq \delta$. 

By  estimate \eqref{eq:asymptotics for phi dot} for $\dot{\phi}$, we have
\begin{align*}
\dot{\phi}\left( \frac{d(t,x)-\tilde{\sigma}}{\ep} \right) \le C \frac{\ep^{2s+1}}{\delta^{2s+1}},
\end{align*}
which combined with \eqref{eq:final_J_of_v_ep}, the estimate for $\bar{a}_\ep$  in Lemma~\ref{lem:a_estimates} and the definition of $\mu$ in  \eqref{def:mu_definition},  gives 
\begin{align*}
\mathcal{J}[v^\ep](t,x) & \leq C\frac{\ep^{2s+1}}{\delta^{4s+1}}+
Ro_\ep(1) +  O( R^{-2s}) - \mu
\le Ro_\ep(1) +  O( R^{-2s}) - \sigma,
\end{align*}
 where we also used that  $\mu \geq \sigma$ and $\ep/\delta^2=o_\ep(1)$.  Arguing as in Case 1, \eqref{eq:pde sub}  follows. 
\medskip

We finally show \eqref{eq:barrier_est}. Let $\tilde{\rho}:= \delta /\tilde \sigma$.
 Fix $(t,x)$ such that $d(t,x) - \tilde{\sigma} \geq \tilde{\rho} $. Using \eqref{eq:asymptotics for phi},
  \eqref{a_near},  \eqref{psi_far}, condition \eqref{delta:def}, and that 
  $\mu/\alpha\leq \sigma/(\alpha\delta^{2s})= \tilde \sigma/\delta^{2s}$, we get
  
\begin{align*}
v^\ep(t,x) &\geq H\left(\frac{d(t,x)- \tilde{\sigma}}{\ep} \right) - C\frac{\ep^{2s}}{|d(t,x) - \tilde{\sigma}|^{2s}}  - C\frac{\ep^{4s}}{\delta^{2s}|d(t,x) - \tilde{\sigma}|^{2s}} \\
& \qquad - C\ep^{2s} - \frac{\tilde\sigma \ep^{2s}}{\delta^{2s}}\\
& \geq 1 - C \frac{\ep^{2s}}{\tilde\rho^{2s}}- C\frac{\ep^{4s}}{\delta^{2s}\tilde\rho^{2s}} - C\ep^{2s}- 
\frac{\tilde\sigma\ep^{2s}}{\delta^{2s}}\\
& = 1-C\tilde\sigma^{2s}\frac{\ep^{2s}}{\delta^{2s}}-\(C\frac{\tilde\sigma^{2s}\ep^{2s}}{\delta^{2s}}-C\delta^{2s} -\tilde\sigma\)\frac{\ep^{2s}}{\delta^{2s}}\\
&\geq 1 - C\tilde\sigma^{2s}\frac{\ep^{2s}}{\delta^{2s}}-\frac{\tilde\sigma\ep^{2s}}{2\delta^{2s}}\\
&\geq 1 - \tilde C\tilde\sigma^{2s}\frac{\ep^{2s}}{\delta^{2s}}, 
\end{align*} 
for some $\tilde C>0$ and $\ep$ sufficiently small.
\end{proof}

 \subsection{Local subsolutions}

We now construct local subsolutions to \eqref{eq:pde}. For $t\in[t_0,t_0+h]$, let  $\Omega_t$ be a bounded open set. Assume that there exists positive constants $r',\,\rho$ and $x_0\in\R^n$  such that the signed distance function $\tilde{d}(t,x)$ associated to the set  $ \Omega_t $ is smooth in the set 
 $Q_{2\rho}\cap ([t_0,t_0+h]\times B(x_0,r'))$ (recall \eqref{eq:Q2rho}).
Let $d$ denote the bounded extension of $\tilde{d}$ outside of $Q_{\rho}$  as defined in Definition \ref{defn:extension},
then $d$ is smooth in  $[t_0,t_0+h]\times B(x_0,r')$. Assume that 
there exist $\sigma>0$ and  $0<r<r'$ such that,
\begin{equation}\label{eq:fmc for d-local}
\partial_td \leq c_0 \kappa[x,d(t, \cdot)] - c_0 \sigma  \quad \hbox{in}~Q_\rho\cap ([t_0,t_0+h]\times  B(x_0,r)).
\end{equation}
Let $\alpha$ and $\tilde{\sigma}$ be defined  as in \eqref{alphasigma}. 
The following lemma is the local version of Lemma \ref{lem:barrier}. 

\begin{lem}[Local subsolutions to \eqref{eq:pde}]  \label{lem:barrier-barrier}
 Assume that $d$ is smooth in  $[t_0,t_0+h]\times B(x_0,r')$ and that \eqref{eq:fmc for d-local} holds with $c_0$ defined in \eqref{def:c0} and $0<r<r'$. 
Let   $v^{\ep}$ be defined as in \eqref{eq:barrier defn} 
 with $0<\tilde\sigma<\rho/2$, $R>1$ and $\delta$ satisfying \eqref{delta:def}. 
  Then there exists $R_0=R_0(\sigma)$ and $\ep_0=\ep_0(r,r', \sigma)>0$ such that for all $R>R_0$ and $0<\ep<\ep_0$, $v^{\ep}$ satisfies 
\[
\ep \partial_t v^{\ep} -\mathcal{I}^s_n[v^{\ep} ] - \frac{1}{\ep^{2s}} W'(v^{\ep} ) \leq  -\frac{\sigma}{2}
\quad \hbox{in}~[t_0,t_0+h] \times B(x_0,r). 
\]
\end{lem}
\begin{proof} 
The proof follows similarly as in the proof of Lemma \ref{lem:barrier}.

\end{proof}

\section{Proof of Theorem \ref{thm:main_result}}\label{sec:proof main_result}
\begin{proof} 
    We apply an adaptation of the abstract method in \cite{Barles-Souganidis,Barles-DaLio} as described in Section \ref{sec:flows}.
    Let $ \delta >0$  satisfy \eqref{delta:def}. Define the open sets
    \begin{align*}
    D &= \operatorname{Int}\left\{ (t,x) \in (0,\infty) \times\R^n : \liminf_{\ep \to 0}{_*} \frac{u^{\ep}(t,x) - 1}{\ep^{2s}\delta^{-2s}} \geq 0 \right\} \subset (0,\infty)\times\R^n\\[.5em]
    E &= \operatorname{Int}\left\{ (t,x) \in (0,\infty) \times\R^n : \limsup_{\ep \to 0}{^*}\frac{u^{\ep}(t,x)}{\ep^{2s}\delta^{-2s}} \leq 0 \right\} \subset (0,\infty)\times\R^n.
    \end{align*}

To define the traces of $D$ and $E$ in $\{0\} \times\R^n$, we first define the functions $\underline{\chi}, \overline{\chi}:(0,\infty) \times \R^n \to \{-1,1\}$, respectively, by
\[
\underline{\chi} = \one_{D} - \one_{(D)^c} \quad \hbox{and} \quad 
\overline{\chi}= \one_{(E)^c} - \one_{E}.
\]
Since $D$ is open, $\underline{\chi}$ is lower semicontinuous, and since $(E)^c$ is closed, $\overline{\chi}$ is upper semicontinuous. To ensure that $\overline{\chi}$ and $\underline{\chi}$ remain lower and upper semicontinuous, respectively, at $t=0$, we set
\[
\underline{\chi}(0,x) = \liminf_{t\to 0,~y\to x} \underline{\chi}(t,y) \quad \hbox{and} \quad
\overline{\chi}(0,x) = \limsup_{t\to 0,~y\to x} \overline{\chi}(t,y).
\]
Define the traces $D_0$ and $E_0$ by
\begin{align*}
D_0 = \{x \in \R^n : \underline{\chi}(0,x) = 1 \} \quad \hbox{and} \quad
E_0 = \{x \in \R^n : \overline{\chi}(0,x) = -1 \}.
\end{align*}
Note that $D_0 $ and $E_0 $ are open sets. 
For $t > 0$, define the sets $D_t$ and $E_t$ by
\[
D_t = \{x \in \R^n : (t,x) \in D\}  
\quad \hbox{and} \quad
E_t = \{x \in \R^n : (t,x) \in E\}. 
\]
We need the following propositions for the abstract method. Their proofs are delayed until the end of the section. 

\begin{prop}[Initialization]\label{prop:initialization}
\[
\Omega_0 \subset D_0
\quad \hbox{and} \quad
(\overline{\Omega}_0)^c \subset E_0.
\]
\end{prop}

\begin{prop}[Propagation]\label{prop:sub/sup flows}
$(D_t)_{t>0}$ is a generalized super-flow, 
and  $((E_t)^c)_{t>0}$ is a generalized sub-flow, according to Definition \ref{def:generalized_flows}.
\end{prop}
By \cite[Corollary 1]{Imbert09}, 
it follows from Propositions \ref{prop:initialization} and \ref{prop:sub/sup flows} that
\[
{^+}\Omega_t \subset D_t \subset {^+}\Omega_t \cup \Gamma_t
\quad \hbox{and} \quad
{^-}\Omega_t \subset E_t \subset {^-}\Omega_t \cup \Gamma_t.
\]
The conclusion readily follows; we write the details for completeness.

First, since ${^+}\Omega_t \subset D_t$, we use the definition of $D_t$ to see that
\begin{equation}\label{eq:liminf from below}
\liminf_{\ep \to 0}{_*} u^{\ep}(t,x) \geq 1 \quad \hbox{for}~x \in {^+}\Omega_t.
\end{equation}
Using that $ {^-}\Omega_t \subset E_t$, we similarly get
\begin{equation} \label{eq:limsup from above}
\limsup_{\ep \to 0}{^*} u^{\ep}(t,x) \leq 0 \quad \hbox{for}~x \in {^-}\Omega_t.
\end{equation}
Now, since the constant functions 0 and 1  solve equation \eqref{eq:pde} and $0 \leq u_0^{\ep} \leq 1$, the comparison principle implies that  $0 \leq u^{\ep} \leq 1$. In particular,
\[
0 \leq \liminf_{\ep \to 0}{_*} u^{\ep} \quad \hbox{and} \quad 
\limsup_{\ep \to 0}{^*} u^{\ep} \leq 1.
\]
Together with \eqref{eq:limsup from above} and respectively \eqref{eq:liminf from below} we have
\[
\lim_{\ep \to 0} u^{\ep}(t,x) = 0 \quad \hbox{in}~ {^-}\Omega_t
\quad \hbox{and} \quad
\lim_{\ep \to 0} u^{\ep}(t,x) = 1 \quad \hbox{in}~ {^+}\Omega_t.
\]
\end{proof}
It remains to prove Propositions \ref{prop:initialization} and \ref{prop:sub/sup flows}. 
\subsection{Proof of Proposition \ref{prop:initialization} }
\begin{proof}
We will prove that $\Omega_0 \subset D_0$. The proof of $\overline{\Omega}_0^c \subset E_0$ is similar. Fix a point $x_0 \in \Omega_0$. To prove that $x_0 \in D_0$, it is enough to show that for all $(t,x)$ in a neighborhood of $(0,x_0)$ in $ [0, \infty) \times \mathbb{R}^n$, the following inequality holds: 
\[
\liminf_{\ep \to 0}{_*} \frac{u^{\ep}(t,x) - 1}{\ep^{2s}\delta^{-2s}} \geq 0.
\]
We will use  Lemma \ref{lem:barrier}  to construct a suitable (global in space) subsolution $v^{\ep} \leq u^{\ep}$.  
Let $\tilde{\sigma}$ be such that $0< \tilde{\sigma} < d^0(x_0)$, where $d^0$ is defined in \eqref{def:signed_distance_function} and let $r>0$ be given by
\begin{equation}\label{def:radius_of_ball}
    r = d^0(x_0) - \tilde{\sigma}.
\end{equation}
Note that $B(x_0, r) \subset \subset \Omega_0$. Let $C>0$ be a constant to be determined. For $t\le r\!/(2C)$, let $\tilde{d}(t,x)$ be the signed distance function associated to the ball $B(x_0, r-Ct)$, namely
\begin{align*}
    \tilde{d}(t,x) = r- Ct - |x-x_0|.
\end{align*}
Notice that by \eqref{def:radius_of_ball},
\begin{equation}\label{eq:prop_distances_ineq}
    d^0(x) - \tilde{\sigma} \geq \tilde{d}(0,x).
\end{equation}
For $0<\rho < r\!/4$, let $d$ be the smooth, bounded extension of $\tilde{d}$ outside of 
\begin{align*}
    Q_\rho = \left\{ (t,x) \in \left[0, \frac{r}{2C} \right] \times \mathbb{R}^n: |d(t,x)| <\rho \right\},
\end{align*}
as in Definition \ref{defn:extension}. For $(t,x) \in Q_\rho$, we have that  $|x-x_0|\geq r/2-\rho\geq r/4$. Moreover, recalling  Remark \ref{Kdextensionrem}, and by  Proposition \ref{ballFMC}, we  have that  
$$\partial_t d(t,x) =\partial_t \tilde d(t,x) = -C$$
 and 
 $$\kappa [x, d(t, \cdot)]=\kappa [x, \tilde d(t, \cdot)] =\frac{ -\omega }{|x-x_0|^{2s}},$$ 
for  some $\omega>0$. This implies that, for $c_0$ as in \eqref{def:c0}, 
\begin{equation}\label{eq:equation_satisfied}
    \partial_t d(t,x) - c_0 \kappa[x, d(t, \cdot)] = -C + \frac{c_0 \omega }{|x-x_0|^{2s}}\leq  -C + \frac{4^{2s}c_0 \omega }{r^{2s}}\leq -c_0\sigma,
\end{equation}
for $C>0$ sufficiently large and with  $\sigma = W''(0) \tilde{\sigma}$. Moreover, we can assume, by possibly taking  $\tilde{\sigma}$ smaller, that $2 \tilde{\sigma} < \rho$. 

Let $v^\ep(t,x)$ be defined as in \eqref{eq:barrier defn}. Then, by \eqref{eq:equation_satisfied}, and Lemma \ref{lem:barrier}, for $R=R(\sigma)$ sufficiently  large, $\delta$ as in  \eqref{delta:def},
and $\ep=\ep(\sigma)$ sufficiently small, the function  $v^\ep$ solves  \eqref{eq:pde sub} in $[0, r\!/(2C)] \times \R^n$. We claim that, by eventually taking  $\ep$ smaller with respect to $\sigma$ if necessary,
\begin{equation}\label{eq:prop_initial_condition}
    v^{\ep}(0,x)\leq u^{\ep}(0,x)=u_0^\ep(x)\quad\text{for all }x\in\R^n,
\end{equation}
with $u_0^\ep$ as in \eqref{initial_data}. We split the proof of \eqref{eq:prop_initial_condition} into  two cases: when
 $x$ is near the boundary $\partial B(x_0,r)$,  and when $x$ far from 
it. 

\medskip
\noindent
{\bf Case 1:} \emph{$|\tilde{d}(0,x)-\tilde{\sigma}| < 2\delta$.}

 Since $\delta=o_\ep(1)$, we may assume $2\delta<\tilde \sigma$, so that $0 \le \tilde{d}(0,x) \le 2 \tilde{\sigma} <\rho$. Recalling Definition \ref{defn:extension}, we have that $d(0,x) = \tilde{d}(0,x)$. By the monotonicity of $\phi$ and estimate \eqref{eq:asymptotics for phi},  
 \begin{align*}
    \phi \left( \frac{d(0,x)-\tilde{\sigma}}{\ep} \right)  \le \phi \left( \frac{2\delta}{\ep} \right) \le 1- C\dfrac{\ep^{2s}}{\delta^{2s}}.
\end{align*}
If $|\tilde{d}(0,x)-\tilde{\sigma}| < \delta$, then  by \eqref{psi_near}
$$\ep^{2s} \psi \left(\frac{d(0,x)-\tilde{\sigma}}{\ep};0,x \right)\leq 
C\ep^{2s},$$
while if $\delta\leq |\tilde{d}(0,x)-\tilde{\sigma}| < 2\delta$, then by \eqref{psi_far}
 and using that $\ep/\delta^2=o_\ep(1)$, 
$$\ep^{2s} \psi \left(\frac{d(0,x)-\tilde{\sigma}}{\ep};0,x \right)\leq 
\frac{C\ep^{4s}}{\delta^{2s}|d(0,x)-\tilde{\sigma}|^{2s}}\leq C\frac{\ep^{4s}}{\delta^{4s}}\leq C\ep^{2s}. $$
From the above estimates on $\psi$,  estimate \eqref{a_near}, and recalling that $\mu>0$, we get
\begin{align*}
 \ep^{2s} \psi \left(\frac{d(0,x)-\tilde{\sigma}}{\ep};0,x \right)+  \frac{\ep^{2s}}{\alpha} \left( \bar{a}_\ep \left[ d-\tilde{\sigma} \right](0,x) - \mu[d-\tilde{\sigma}](0,x) \right) \le C \ep^{2s}.
\end{align*} 
On the other hand, using \eqref{eq:prop_distances_ineq}, that $d(0,x) \geq 0$, and \eqref{eq:asymptotics for phi},  we have
\begin{align*}
    \phi \left( \frac{d^0(x)}{\ep} \right) \geq \phi \left( \frac{\tilde{\sigma}}{\ep} \right) \geq 1 - C\frac{\ep^{2s}}{\tilde{\sigma}^{2s}}.
\end{align*}
Putting it all together, since $\delta=o_\ep(1)$,  for $\ep$ sufficiently small, we obtain
\begin{align*}
    v^\ep(0,x) \le 1 -  C \frac{\ep^{2s}}{\delta^{2s}}+C\ep^{2s}  \le 1 - C\frac{\ep^{2s}}{\tilde{\sigma}^{2s}} \le \phi \left( \frac{d^0(x)}{\ep} \right) = u^\ep(0,x),
\end{align*}
which proves \eqref{eq:prop_initial_condition} for Case 1.



\medskip
\noindent
{\bf Case 2:} \emph{$|\tilde{d}(0,x)-\tilde{\sigma}| \geq 2 \delta$.} 

Recalling the definition of $\mu$ in \eqref{def:mu_definition}, and that $\tilde\sigma=\sigma/\alpha$, we have that 
$\mu[d-\tilde{\sigma}](t,x)/\alpha = \sigma/(\alpha\delta^{2s})=\tilde \sigma/\delta^{2s}$. 

By \eqref{a_near} and \eqref{psi_far},  we have 
 
\begin{equation}\label{psicomputationcase2prop1}\begin{split}
 \ep^{2s} &\psi \left(\frac{d(0,x)-\tilde{\sigma}}{\ep};0,x \right)+  \frac{\ep^{2s}}{\alpha} \left( \bar{a}_\ep \left[ d-\tilde{\sigma} \right](0,x) - \mu[d-\tilde{\sigma}](0,x) \right)\\& 
\leq  \frac{C\ep^{4s}}{\delta^{2s}|d(0,x)-\tilde{\sigma}|^{2s}}+ C \ep^{2s}-\frac{\tilde \sigma\ep^{2s}}{\delta^{2s}}
 \\&\le 
 \frac{C\ep^{4s}}{\delta^{4s}}+ C \ep^{2s}-\frac{\tilde \sigma\ep^{2s}}{\delta^{2s}}\\&
 =\left(\frac{\ep^{2s}}{\delta^{2s}}+C\delta^{2s}-\tilde\sigma\right) \frac{\ep^{2s}}{\delta^{2s}}\\&
 \le -\frac{\tilde \sigma\ep^{2s}}{2\delta^{2s}}\\&
 \leq 0,
\end{split}\end{equation}
if $\ep$ is taken sufficiently small, where we used  that $\delta=o_\ep(1)$ and $\ep/\delta^2=o_\ep(1)$. 

Assume first $|\tilde{d}(0,x)|\leq \rho$. Then $d(0,x)=\tilde d(0,x)$, and by
 \eqref{eq:prop_distances_ineq} together with the monotonicity of $\phi$, we have 
\begin{align*}
    \phi \left( \frac{d(0,x)-\tilde{\sigma}}{\ep} \right) \le \phi \left( \frac{d^0(x)}{\ep} \right)=u^\ep_0(x).
\end{align*}
Combining the inequality above with \eqref{psicomputationcase2prop1} yields 
\eqref{eq:prop_initial_condition}.

Next, assume $|\tilde{d}(0,x)|\geq \rho$. Then $|d(0,x)|\geq \rho$ and, since $2\tilde\sigma<\rho$, we have that $|d(0,x)-\tilde\sigma|\geq \rho/2$.
If $\tilde{d}(0,x)\geq \rho$,  then  \eqref{eq:prop_distances_ineq} implies that $d_0(x)\geq \rho$ and by estimates \eqref{eq:asymptotics for phi} and  \eqref{psicomputationcase2prop1}, 
$$v^\ep(0,x)\leq 1-\frac{\tilde\sigma\ep^{2s}}{2\delta^{2s}}\leq 1-C\frac{\ep^{2s}}{\rho^{2s}}\leq \phi \left( \frac{d^0(x)}{\ep} \right)=u^\ep_0(x),  $$
for $\ep$, thus $\delta$,  small enough.

If $\tilde{d}(0,x)\leq -\rho$, then again from \eqref{eq:asymptotics for phi} and  \eqref{psicomputationcase2prop1}, and for  $\ep$ small enough, 
$$v^\ep(0,x)\leq C\frac{\ep^{2s}}{\rho^{2s}}-\frac{\tilde\sigma\ep^{2s}}{2\delta^{2s}}\leq 0\leq u^\ep_0(x).$$
This concludes the proof of \eqref{eq:prop_initial_condition}  in Case 2.

\medskip

By \eqref{eq:prop_initial_condition} and the comparison principle, 
\begin{align*}
    u^\ep(t,x) \geq v^\ep(t,x) \quad \text{for all } (t,x) \in \left[ 0, \frac{r}{2C} \right] \times \mathbb{R}^n.
\end{align*}
Since $2 \tilde{\sigma} < \rho<r/4$, and $\delta=o_\ep(1)$, for $t\in [0, r/(2C)]$ and $\ep$ sufficiently small, we have that 
$$\{x :d(t,x) \geq 2\tilde{\sigma} \}= \{x :\tilde{d}(t,x) \geq 2\tilde{\sigma} \}\subset \left\{x :\tilde{d}(t,x) -\tilde{\sigma}\geq \frac{\delta}{\tilde\sigma}\right \},$$ 
and 
$$ \{x:\tilde{d}(t,x) \geq 2\tilde{\sigma} \}=\{|x-x_0|\leq r-Ct-2\tilde \sigma\}\supset\left \{|x-x_0|\leq \frac{r}{4}\right\}.$$
Consequently, by  \eqref{eq:barrier_est} for $t\in [0, r/(2C)]$ and $|x-x_0|\leq r/4$, 
we obtain
\begin{align*}
    \liminf_{\ep \to 0}{_*} \frac{u^{\ep}(t,x) - 1}{\ep^{2s}\delta^{-2s}}
&\geq \liminf_{\ep \to 0}{_*} \frac{v^{\ep}(t,x) - 1}{\ep^{2s}\delta^{-2s}} 
\geq -\tilde C \tilde{\sigma}^{2s}.
\end{align*} 
Letting $\tilde{\sigma} \to 0$, the result follows.
\end{proof}

\subsection{Proof of Proposition \ref{prop:sub/sup flows} }

\begin{proof}
We will show that $(D_t)_{t>0}$ is a generalized super-flow. 
The proof that $((E_t)^c)_{t>0}$ is a generalized sub-flow is similar. 

Let $(t_0,x_0) \in (0,\infty) \times \R^n$, $h, \, r>0$, and $\varphi:(0,\infty)\times \R^n \to \R$ be a smooth function satisfying \ref{item:i}-\ref{item:v} in Definition \ref{def:generalized_flows}
in $[t_0,t_0+h]$ with $F^\ast$  given in \eqref{def:F_ast_definition} and $c_0$ as in \eqref{def:c0}.
Then, there exists $h'>h$ such that $\varphi$ satisfies  \ref{item:i}-\ref{item:v} in $[t_0,t_0+h']$, with an eventually smaller $\tau$ in \ref{item:ii}. 
For $t\in [t_0,t_0+h']$, let us denote 
$$\Omega_t=\{x\in\R^n\,:\,\varphi(t,x)>0\}\quad\text{and}\quad\Gamma_t=\partial \Omega_t.$$
By  \ref{item:iii}  there exists $r'>r$ such that $\nabla \varphi\neq 0$ on  $\Gamma_t\cap  \overline{B}(x_0,r')$ which is therefore a smooth (and, by \ref{item:i},  non-empty) set. 
Let $\tilde{d}(t,x)$ be the signed distance function associated to $\Omega_t$, and let
$Q_\rho =\{(t,x)\in \left[t_0,t_0+h'\right]\times\R^n\,:\,|\tilde{d}(t,x)|<\rho\}$ for $\rho>0$. 
Then, there exists $\rho>0$ such that $\tilde d$  is smooth 
in $Q_{2\rho}\cap  ([t_0,t_0+h']\times B(x_0,r'))$
 and by \ref{item:ii} (recall \eqref{meancurvature_intro_sigma}),  
\begin{equation}\label{eq:mc for d-locali_0prop72}
\partial_t\tilde d \leq c_0 \kappa[x, \tilde{d}(t,\cdot)] - \tilde \tau\quad \hbox{in}~Q_{\rho}\cap  ([t_0,t_0+h']\times B(x_0,r)), 
\end{equation}
for some $\tilde\tau>0$. 
Moreover, by \eqref{eq:mc for d-locali_0prop72}, and recalling  Remark \ref{Kdextensionrem},
  if   $d(t,x)$ is the  bounded extension of $\tilde{d}(t,x)$ outside of 
$ Q_\rho $ as in Definition \ref{defn:extension}, then  $d$ is smooth 
in $[t_0,t_0+h']\times B(x_0,r')$ and 
\begin{equation*}\label{eq:mc for d-locali_0prop72-i}
\partial_td \leq c_0 \kappa[x,d(t,\cdot)]  - c_0 \sigma \quad \hbox{in}~Q_{\rho}\cap  ([t_0,t_0+h']\times B(x_0,r)),
\end{equation*}
for some $\sigma>0$.  

Let $v^\ep(t,x)$ be defined as in \eqref{eq:barrier defn}. 
Then, Lemma \ref{lem:barrier-barrier}  implies that, for $R=R(\sigma)$ sufficiently  large, $\delta$ as in  \eqref{delta:def}, $\alpha$ and $\tilde{\sigma}$ as in \eqref{alphasigma}, 
and $\ep=\ep(\sigma)$ sufficiently small, the function $v^\ep$ 
is a solution to \eqref{eq:pde sub} in $[t_0,t_0+h'] \times B(x_0,r)$.

We  will show that for $\tilde\sigma<\rho/2$, and by eventually taking  $\ep$ smaller with respect to $\sigma$ if necessary,
\begin{equation}\label{initialpropo2}
v^{\ep}(t_0,x)\leq u^{\ep}(t_0,x)\quad\text{for all }x\in  \R^n
\end{equation}
and
\begin{equation}\label{initialpropo2-boundary}
v^{\ep}(t,x)\leq u^{\ep}(t,x)\quad\text{for all }(t,x) \in [t_0,t_0+h'] \times( \R^n \setminus B(x_0,r)). 
\end{equation}
We start with \eqref{initialpropo2}. 
Since $\varphi$ satisfies \ref{item:i} and \ref{item:iv} in Definition \ref{def:generalized_flows}, we have that $\Omega_{t_0} \subset \subset D_{t_0}$.
 Therefore, there exists a compact set $K$   
 such that 
$$\Omega_{t_0}\subset K \subset D_{t_0},$$
and, by possibly taking  $\tilde\sigma>0$ smaller, we may assume that 
\begin{equation}\label{didi+1proppflow-K}
d_{K}(x)-2\tilde\sigma \geq \tilde d(t_0,x),
\end{equation}
where $d_K$ denotes the signed distance function from $K$. 

The proof of \eqref{initialpropo2} is broken into three cases:  we first consider the case when $x$ is close to $\Gamma_{t_0}$, then  when $x$ is in $K$ but far from $\Gamma_{t_0}$, and finally when $x$ is not in $K$.

\medskip

\noindent
{\bf Case 1}:~\emph{ $|\tilde{d} (t_0,x)-\tilde\sigma| < 2\delta$.}

Like in Case 1 in the proof of Proposition  \ref{prop:initialization}, we can show that, for $\ep$ small enough,
\[
v^\ep(t_0,x) \leq 1 - C \frac{\ep^{2s}}{\delta^{2s}}.
\]
Since $\tilde{d}(t_0,x) > \tilde\sigma-2\delta \geq 0$  for $\ep$ small enough,  by \eqref{didi+1proppflow-K}  we know that $x \in K\subset D_{t_0}$. 
 Since $K$ is compact,  and by the definition of  $D_{t_0}$,  given $\tau_0>0$, for $\ep$ small enough and $y\in K$,  
\begin{equation}\label{eq:ue-delta}
\frac{u^\ep(t_0,y) - 1}{\ep^{2s} \delta^{-2s}}  \geq - \tau_0.  
\end{equation}
In particular, 
\[
u^\ep(t_0,x) \geq 1 -\tau_0 \frac{\ep^{2s}}{\delta^{2s}} \geq 1 - C \frac{\ep^{2s}}{\delta^{2s}} \geq v^\ep(t_0,x), 
\]
if $\tau_0 \leq C$. Therefore, \eqref{initialpropo2} holds for Case 1.

\medskip

\noindent
{\bf Case 2}: \emph{$x \in K$ and  $|\tilde d(t_0,x)-\tilde\sigma|\geq 2 \delta$.}

We note that
\[
\phi \left(\frac{d(t_0,x) - \tilde{\sigma}}{\ep}\right) \leq 1.
\]
 Proceeding  as in Case 2 in the proof of Proposition  \ref{prop:initialization},  as in \eqref{psicomputationcase2prop1} we find
\begin{align*}
 \ep^{2s} &\psi \left(\frac{d(t_0,x)-\tilde{\sigma}}{\ep};t_0,x \right)+  \frac{\ep^{2s}}{\alpha} \left( \bar{a}_\ep \left[ d-\tilde{\sigma} \right](t_0,x) - \mu[d-\tilde{\sigma}](t_0,x) \right)
 \leq -\frac{\tilde\sigma\ep^{2s}}{2\delta^{2s}},
\end{align*}
for $\ep$ small enough.

On the other hand,  since $x \in K$, we know that \eqref{eq:ue-delta} holds  at $x$ for  given $\tau_0$ and $\ep$ small enough. 
Thus, for $\tau_0\leq \tilde\sigma \!/2$, 
\[
    v^\ep(t_0,x) \le 1-\frac{\tilde\sigma\ep^{2s}}{2\delta^{2s}}\leq 1 - \tau_0 \frac{\ep^{2s}}{\delta^{2s}} \le u^\ep(t_0,x).
\]
We now have \eqref{initialpropo2} in Case 2.

\medskip

\noindent
{\bf Case 3}: \emph{$x \notin K $.}

Since $d_{K}(x) \leq 0$, by \eqref{didi+1proppflow-K} we have that $\tilde{d}(t_0,x) \leq - 2\tilde{\sigma}$. In particular, $d(t_0,x)-\tilde{\sigma}<-2\tilde\sigma<-2\delta $ and 
recalling the definition of $\mu$ in  \eqref{def:mu_definition}, and that $\tilde\sigma=\sigma/\alpha$, we have that $\mu[d(t_0,x)-\tilde{\sigma}]/\alpha = \sigma/(\alpha\delta^{2s})=\tilde\sigma/\delta^{2s}$ . 
Moreover, 
 by \eqref{eq:asymptotics for phi}, 
 $$\phi \left(\frac{d(t_0,x) - \tilde{\sigma}}{\ep}\right)\leq \frac{C\ep^{2s}}{|d(t_0,x) - \tilde{\sigma}|^{2s}}\leq C\frac{\ep^{2s}}{\tilde{\sigma}^{2s}}. $$
 As in \eqref{psicomputationcase2prop1}  of Proposition  \ref{prop:initialization},
 \begin{align*}
 \ep^{2s} &\psi \left(\frac{d(t_0,x)-\tilde{\sigma}}{\ep};t_0,x \right)+  \frac{\ep^{2s}}{\alpha} \left( \bar{a}_\ep \left[ d-\tilde{\sigma} \right](t_0,x) - \mu[d-\tilde{\sigma}](t_0,x) \right)
 \leq -\frac{\tilde\sigma\ep^{2s}}{2\delta^{2s}},
\end{align*} for $\ep$ small enough.  
 Therefore, 
\begin{align*}
    v^\ep(t_0,x) \le C \frac{\ep^{2s}}{\tilde{\sigma}^{2s}}   -  \frac{\tilde\sigma\ep^{2s}}{2\delta^{2s}} \le 0,
\end{align*}
for  $\ep$ sufficiently small, since $\delta=o_\ep(1)$. 
 Now, since the zero function is a solution to \eqref{eq:pde} and $u_0^{\ep} \geq 0$, the comparison principle implies $u^\ep(t_0,x) \geq 0$. Therefore, 
\[
u^\ep(t_0,x) \geq 0
 \geq v^\ep(t_0,x),
\]
and \eqref{initialpropo2} holds for Case 3.

\medskip
This proves \eqref{initialpropo2}. Inequality \eqref{initialpropo2-boundary} follows with a similar argument using that 
 $\varphi$ satisfies \ref{item:i} and \ref{item:v}  in Definition~\ref{def:generalized_flows}.

With \eqref{initialpropo2} and \eqref{initialpropo2-boundary}, the comparison principle then implies
\begin{equation}\label{eq:comp-claim}
u^{\ep} (t,x)\geq v^{\ep}(t,x) \quad \hbox{for all}~ (t,x)\in [t_0,t_0+h'] \times \R^n. 
\end{equation}
By  \eqref{eq:barrier_est},  we have that, for all $t\in[t_0,t_0+h']$, 
\begin{align*}
 \frac{u^{\ep}(t,x) - 1}{\ep^{2s} \delta^{-2s}} 
  &\geq  \frac{v^{\ep}(t,x) - 1}{\ep^{2s} \delta^{-2s}} 
 \geq -\tilde C\tilde{\sigma}^{2s} \quad \hbox{in}~\left\{x \in \R^n: d(t,x) - \tilde{\sigma} \geq \delta \tilde{\sigma}^{-1}\right\}.
\end{align*}
Letting $\ep \to 0$ (and so $\delta \to 0$), it follows that, for all $t\in[t_0,t_0+h']$, 
\begin{equation}\label{lastinclusion_prop72}
\{x \in \R^n :d(t,x) - \tilde{\sigma} \geq 0\} \subset \left\{ x \in \R^n : \liminf_{\ep \to 0}{_*} \frac{u^{\ep}(t,x) - 1}{\ep^{2s}\delta^{-2s}} \geq -\tilde C \tilde{\sigma}^{2s} \right\}.
\end{equation}
Now, let $x_1 \in \{ x \in B(x_0,r) : \varphi(t_0+h,x) >0\}$, so that  $d(t_0+h,x_1)>0$.
Then, there exist $r_1>0$ and $0<\tau<h'-h$ such that for $|t-(t_0+h)|<\tau$, it holds that $B(x_1,r_1)\subset \{x \in B(x_0,r) :d(t,x) > 0 \}$ 
and by \eqref{lastinclusion_prop72}, for  $\tilde\sigma<r_1/2$, 
$$[t_0+h-\tau,t_0+h+\tau]\times B\left(x_1,\frac{r_1}{2}\right)\subset\left\{ (t,x): \liminf_{\ep \to 0}{_*} \frac{u^{\ep}(t,x) - 1}{\ep^{2s}\delta^{-2s}} \geq -\tilde C \tilde{\sigma}^{2s} \right\}.$$
Taking $\tilde{\sigma} \to 0$, we see that $(t_0+h,x_1)$ is an interior point of the set 
$$\left\{ (t,x) \in (0,\infty) \times\R^n : \liminf_{\ep \to 0}{_*} \frac{u^{\ep}(t,x) - 1}{\ep^{2s}\delta^{-2s}} \geq 0 \right\} ,$$
namely, it belongs to $D$. This proves the desired inclusion
\[
\{x \in B(x_0,r) :\varphi(t_0+h,x) > 0 \} 
 = \{x \in B(x_0,r) :d(t_0+h,x) >0 \} \subset  D_{t_0+h}. 
\]

\end{proof}

\section{Proof of Lemma \ref{thm: b_ep fmc}}\label{sec:proof of b_ep fmc}
For ease of notation, throughout this section we omit the dependence on $t$. Moreover, we write $y = (y', y_n)$ with $y' \in \mathbb{R}^{n-1}$. 
 Recall that if $|d(x)|<\rho$, with $\rho$ as in Definition \ref{defn:extension}, then $|\nabla d(x)| = 1$. Thus, there exists an orthonormal matrix $T$ such that
\begin{equation}\label{changevar_s<12}
\nabla d(x) \cdot (Ty) = y_n.
\end{equation}

We begin with some preliminary results that will be needed for the proof
of Lemma \ref{thm: b_ep fmc}. The following lemma is proven in \cite{PatriziVaughan}, 
 see Lemmas 7.1 and 7.2 therein.

\begin{lem}\label{lemm2s<1/2_1}
There exist $\tau_0,\,C>0$ such that for all $0<\tau\leq \tau_0$, $0\leq\sigma<\tau/2$,  if $|d(x)|<\rho$, then
 \begin{equation*}\label{lemm2s<1/2_est1}\int_{\{d(x+z) > d(x)-\sigma,~-\tau<\nabla d(x) \cdot z < -2\sigma\}}\frac{dz}{|z|^{n+2s}}\leq C\tau^{\frac12-s},
 \end{equation*}
and 
 \begin{equation*}\label{lemm2s<1/2_est2}\int_{\{d(x+z) < d(x)+\sigma,~2\sigma<\nabla d(x) \cdot z  <\tau\}}\frac{dz}{|z|^{n+2s}}\leq  C\tau^{\frac12-s}. 
 \end{equation*}
\end{lem}

\begin{lem}\label{0<y_n<tau_lemma}
Assume $|\nabla d(x)|=1$. Then, there exist $\tau_0>0$ and $C>0$ such that for all $0<\tau\leq\tau_0$, 
\begin{equation*}\label{0<y_n<tau_lemma_eq}
\int_{\{|\nabla d(x) \cdot z |<\tau\}} 
 \left|  \phi\( \frac{ d(x+z)}{\ep}\) - \phi\( \frac{ d(x) + \nabla d(x) \cdot z}{\ep}\) \right| \frac{dz}{\abs{z}^{n+2s}}\leq C\tau^{\frac12-s}. 
\end{equation*}
\end{lem}
\begin{proof}
 By the monotonicity of $\phi$, we have that, for some $C_0>0$, 
\begin{align*}
 \phi&\( \frac{ d(x+z)}{\ep}\) - \phi\( \frac{ d(x) + \nabla d(x) \cdot z}{\ep}\) 
\\&\leq
 \phi\( \frac{ d(x)+ \nabla d(x) \cdot z+C_0|z|^2}{\ep}\) - \phi\( \frac{ d(x) + \nabla d(x) \cdot z}{\ep}\) 
 \end{align*}
 and 
 \begin{align*}
 \phi&\( \frac{ d(x+z)}{\ep}\) - \phi\( \frac{ d(x) + \nabla d(x) \cdot z}{\ep}\) 
\\&\geq \phi\( \frac{ d(x)+ \nabla d(x) \cdot z-C_0|z|^2}{\ep}\) - \phi\( \frac{ d(x) + \nabla d(x) \cdot z}{\ep}\).
 \end{align*}
Making the change of variables $z=Ty$ with $T$ as in \eqref{changevar_s<12}, 
and then taking $p=|y'|$, $t=y_n/p$, 
 we get
\begin{align*}
&\int_{\{|\nabla d(x) \cdot z |<\tau\}} 
 \( \phi\( \frac{ d(x)+ \nabla d(x) \cdot z+C_0|z|^2}{\ep}\) - \phi\( \frac{ d(x) + \nabla d(x) \cdot z}{\ep}\) \) \frac{dz}{\abs{z}^{n+2s}}\\&
 =\int_{\{|y_n| <\tau\}} 
 \( \phi\(  \frac{ d(x)+ y_n+ C_0(|y'|^2+y_n^2)}{\ep}\) - \phi\( \frac{ d(x)+ y_n}{\ep}\) \) \frac{dy}{\abs{y}^{n+2s}}\\&
 =\mathcal{H}^{n-2}(\mathcal{S}^{n-2})\int_0^\infty \frac{dp}{p^{1+2s}}\int_{-\frac{\tau}{p}}^{\frac{\tau}{p}}\( \phi\(   \frac{ d(x) + tp+ C_0p^2(1+t^2)}{\ep}\) - \phi\(\frac{d(x) + tp}{\ep}\) \) \frac{dt}{(1+t^2)^\frac{n+2s}{2}}\\&
 =\int_0^{r}\frac{dp}{p^{1+2s}}\,(\ldots)+\int_{r}^\infty \frac{dp}{p^{1+2s}}\,(\ldots)\\&
 =: I_1+I_2,
 \end{align*}
 with $r>0$ to be determined. 
  For the first term above, using that $\dot\phi>0$, we have
 \begin{align*}
  I_1
 &=\frac{C}{\ep}\int_0^{ r} dp\,p^{1-2s} \int_{-\frac{\tau}{p}}^{\frac{\tau}{p}} \frac{dt}{(1+t^2)^\frac{n+2s-2}{2}}\int_0^1\dot \phi\(  \frac{d(x) + tp+ \theta C_0p^2(1+t^2)}{\ep}\) d\theta\\&
 =\frac{C}{\ep}\int_0^{ r } dp\,p^{1-2s} \int_{-\frac{\tau}{p}}^{\frac{\tau}{p}} \frac{dt}{(1+t^2)^\frac{n+2s-2}{2}}\int_0^1\partial_t \left[\phi\(\frac{d(x) + tp+ \theta C_0p^2(1+t^2)}{\ep}\)\right]
 \frac{\ep}{p(1 + 2t\theta C_0p)} d\theta\\&
 \leq C\int_0^{ r } \frac{dp}{p^{2s}} \int_0^1d\theta \int_{-\frac{\tau}{p}}^{\frac{\tau}{p}}\partial_t \left[\phi\( \frac{d(x) + tp+ \theta C_0 p^2(1+t^2)}{\ep}\)\right]dt,
  \end{align*}
  choosing $\tau >0$ so small  that if  $ |tp| < \tau$ then $p(1+ 2t\theta C_0p)\geq p(1-2C_0\tau)\geq p/2$. 
 Integrating with respect to $t$, we obtain
 \begin{align*}
 I_1&\leq C\int_0^{ r} \frac{dp}{p^{2s}} \int_0^1\left[\phi\(\frac{ d(x)+ \tau +\theta C_0(p^2+\tau^2)}{\ep}\) -\phi\(\frac{ d(x)- \tau +\theta C_0(p^2+\tau^2)}{\ep}\)\right] d \theta\\&
 \leq  C\int_0^{r} \frac{dp}{ p^{2s}}=C r^{1-2s}. 
 \end{align*}
We also estimate
  \begin{align*}
 I_2 \leq C\int_{r}^\infty \frac{dp}{p^{1+2s}}\int_0^{\frac{\tau}{p}}dt=C\frac{\tau}{r^{1+2s}}. 
  \end{align*}
Choosing $r=\tau^\frac12$, we  finally obtain
$$\int_{\{|\nabla d(x) \cdot z |<\tau\}} 
 \( \phi\( \frac{ d(x)+ \nabla d(x) \cdot z+C_0|z|^2}{\ep}\) - \phi\( \frac{ d(x) + \nabla d(x) \cdot z}{\ep}\) \) \frac{dz}{\abs{z}^{n+2s}}\leq C\tau^{\frac12-s}.  $$
 Similarly, one can prove
 $$\int_{\{|\nabla d(x) \cdot z |<\tau\}} 
 \( \phi\( \frac{ d(x)+ \nabla d(x) \cdot z-C_0|z|^2}{\ep}\) - \phi\( \frac{ d(x) + \nabla d(x) \cdot z}{\ep}\) \) \frac{dz}{\abs{z}^{n+2s}}\geq -C\tau^{\frac12-s}.  $$
The lemma is then proven. 
\end{proof}

We now proceed with the proof of Lemma \ref{thm: b_ep fmc}. 
Assume $|d(x)|<\delta<\rho.$ By taking $\delta$ larger if necessary, we may assume that
\begin{equation}\label{deltacond:lemmaK}\frac{\ep}{\delta^2}=o_\ep(1).
\end{equation}
We have 
 \begin{equation}\begin{split}\label{aepsplits<1/2}
\overline{b}_\ep&= \int_{\{|z|<R\}} \left( \phi\left(\frac{d(x+ z)}{\ep}\right) - \phi\left(\frac{ d(x) + \nabla d(x) \cdot z}{\ep}\right) \right) \frac{dz}{\abs{z}^{n+2s}}\\&
= \int_{\R^n} \left( \phi\left(\frac{d(x+ z)}{\ep}\right) - \phi\left(\frac{ d(x) + \nabla d(x) \cdot z}{\ep}\right) \right) \frac{dz}{\abs{z}^{n+2s}}+O(R^{-2s}).\\&
 \end{split}
\end{equation}
We then split
\begin{equation}\begin{split}\label{aepsplits<1/2_bis}
& \int_{\R^n} \left( \phi\left(\frac{d(x+ z)}{\ep}\right) - \phi\left(\frac{ d(x) + \nabla d(x) \cdot z}{\ep}\right) \right) \frac{dz}{\abs{z}^{n+2s}}\\&
=\int_{\{ d(x+z) > d(x),~\nabla d(x) \cdot z < 0\}}(\ldots)+\int_{\{d(x+z) < d(x),~\nabla d(x) \cdot z > 0\}}(\ldots)\\&
\quad+\int_{\{d(x+z) > d(x),~\nabla d(x) \cdot z > 0\}}(\ldots)+\int_{\{d(x+z) < d(x),~\nabla d(x) \cdot z < 0\}}(\ldots)\\
&=:I_1+I_2+I_3+I_4.
 \end{split}
\end{equation}
We begin by estimating $I_1$. We further split
\begin{align*}
    I_1 = & \int_{ \left\{ d(x+z)-d(x) >2 \delta, ~\nabla d(x) \cdot z < -2\delta \right\} } (\ldots) \\
    &+ \int_{ \left\{ d(x+z)-d(x) > 0, ~-2\delta<\nabla d(x) \cdot z < 0 \right\} } (\ldots) \\
    &+ \int_{ \left\{ 0<d(x+z)-d(x) <2 \delta, ~\nabla d(x) \cdot z < - 2\delta \right\} } (\ldots) \\
    =& : J_1 + J_2 + J_3.
\end{align*}
We first estimate $J_1$. If $d(x+z)-d(x) > 2\delta$ and $\nabla d(x) \cdot z < -2\delta$, for $|d(x)|<\delta$ we have $d(x+z)> \delta$ and $\nabla d(x) \cdot z + d(x) < - \delta$. Thus,   by \eqref{eq:asymptotics for phi},
\begin{align*}
    \phi&\left(\frac{d(x+ z)}{\ep}\right) - \phi\left(\frac{ d(x) + \nabla d(x) \cdot z}{\ep}\right) \\ & = H\left(\frac{d(x+ z)}{\ep}\right) - H\left(\frac{ d(x) + \nabla d(x) \cdot z}{\ep}\right) \\
    & + O\left( \left| \frac{d(x+ z)}{\ep} \right|^{-2s} \right) + O\left( \left| \frac{ d(x) + \nabla d(x) \cdot z}{\ep} \right|^{-2s} \right) \\
    &= 1 + O \left( \frac{\ep^{2s}}{\delta^{2s}} \right).
\end{align*}
Consequently, using that, by Proposition \ref{prop:fmc_finite_result},
\begin{equation}\label{klemma:integkernel}\one_{ \{ d(x+z)-d(x) > 2\delta, ~\nabla d(x) \cdot z <-2\delta\}}\leq \one_{ \{ d(x+z)-d(x) > 0, ~\nabla d(x) \cdot z <0\}}\in L^1(\R^n), \end{equation}
and recalling the definition of $\kappa^+$ in \eqref{kappa+de}, we get
\begin{equation}\label{J1_in_I1}
    J_1 = \int_{ \left\{ d(x+z)-d(x) > 2\delta, ~\nabla d(x) \cdot z <-2\delta \right\} } \frac{dz}{|z|^{n+2s}} + O \left( \frac{\ep^{2s}}{\delta^{2s}} \right) = \kappa^{+}[x,d] + O \left( \frac{\ep^{2s}}{\delta^{2s}} \right) + o_\delta(1).
\end{equation}
Next, by Lemma \ref{lemm2s<1/2_est1} with $\sigma =0$ and $\tau = 2\delta$, for $\delta$ small enough, we have
\begin{equation}\label{J2_in_I1}
    J_2 = o_\delta(1).
\end{equation}
Finally, we estimate
\begin{align*} 
|J_3|&\leq 2\int_{\{d(x+z) -d(x)>0 ,~\nabla d(x) \cdot z < 0\}} \one_{\{0< d(x+z) -d(x)<2\delta\}}(z) \,\frac{dz}{|z|^{n+2s}}.
\end{align*}
Since, the set $\{d=0\}$ is a smooth surface, we have that 
\begin{equation}\label{chiclosetogamma}\one_{\{0< d(x+z) -d(x)<2\delta\}}(z)\to0\quad\text{ a.e. as }\delta\to0.\end{equation}
Therefore,  by \eqref{klemma:integkernel} and the Dominated Convergence Theorem,
\begin{equation}\label{J3_in_I1}
   J_3=o_\delta(1).
\end{equation}
From \eqref{J1_in_I1}, \eqref{J2_in_I1} and \eqref{J3_in_I1}, and using \eqref{deltacond:lemmaK}, we obtain
\begin{equation}\label{I1_estimate}
    I_1 = \kappa^+[x,d] + o_\ep(1) + o_\delta(1).
\end{equation}
Recalling the definition of $\kappa^-$ in  \eqref{kappa+de} and arguing similarly to the case of 
 $I_1$, we obtain
\begin{equation}\label{I2_estimate}
    I_2 = -\kappa^-[x,d] + o_\ep(1) + o_\delta(1).
\end{equation}
Next, we estimate $I_3$ and $I_4$. We further split
\begin{align*}
   I_3&=\int_{\{d(x+z) -d(x)>2\delta ,~\nabla d(x) \cdot z >4\delta\}}  (\ldots)
\\&\quad+\int_{\{d(x+z) -d(x)>0 ,~0< \nabla d(x) \cdot z < 4\delta \}} (\ldots)
\\&\quad+\int_{\{0<d(x+z) -d(x)<2\delta ,~\nabla d(x) \cdot z >4\delta\}} (\ldots)\\&
=: J_1+J_2+J_3.
\end{align*}
We first estimate $J_1$. If $d(x+z) -d(x)>2\delta$ and $\nabla d(x) \cdot z>4\delta$, then  for $|d(x)| < \delta$, we have $d(x+z) >  \delta$ and $d(x) + \nabla d(x) \cdot z >3\delta$. Then, by \eqref{eq:asymptotics for phi},

\begin{align*} 
&\phi\(\frac{d(x+ z)}{\ep}\) - \phi\( \frac{ d(x) + \nabla d(x) \cdot z}{\ep}\)\\&=H\(\frac{d(x+ z)}{\ep}\)-H\(  \frac{d(x)+\nabla d(x) \cdot z}{\ep}\)
\\&\quad+O\(\left|\frac{d(x+ z)}{\ep}\right|^{-2s}\)+O\(\left|\frac{ d(x) + \nabla d(x) \cdot z}{\ep}\right|^{-2s}\)\\&
= O\(\frac{\ep^{2s}}{\delta^{2s}}\). 
\end{align*}
This implies
\begin{equation*}\begin{split}
    |J_1| &\le   O\(\frac{\ep^{2s}}{\delta^{2s}}\)\int_{\{\nabla d(x) \cdot z >4\delta\}}   \frac{dz}{|z|^{n+2s}}\\&
    =O\(\frac{\ep^{2s}}{\delta^{2s}}\)\int_{\{y_n >4\delta\}}   \frac{dy}{|y|^{n+2s}}\\& \leq O\(\frac{\ep^{2s}}{\delta^{2s}}\)\int_{\{|y| >4\delta\}}   \frac{dy}{|y|^{n+2s}}
    \\&=O\(\frac{\ep^{2s}}{\delta^{4s}}\),
    \end{split}
\end{equation*}
where we used the change of  variables $z=Ty$ with $T$ as in \eqref{changevar_s<12}. Recalling \eqref{deltacond:lemmaK}, we obtain
\begin{equation}\label{J1_for_I3}
J_1=o_\ep(1). 
\end{equation}
Next, we estimate $J_2$. Let $\delta > 0$ be small enough so that $0 < 4\delta < \tau_0$, where $\tau_0 > 0$ is as in Lemma \ref{0<y_n<tau_lemma}. Recalling that $|\nabla d(x)| = 1$ whenever $|d(x)| < \delta$, it then follows from Lemma \ref{0<y_n<tau_lemma} that
\begin{equation}\label{J2_for_I3}
    J_2 = o_\delta(1).
\end{equation}
Finally, we estimate $J_3$. For $\tau_0$ as in Lemma \ref{lemm2s<1/2_1} and $\delta$ so small  that $4\delta < \tau_0$, we write
\begin{align*}
|J_3|&\leq 2 \int_{\{0<d(x+z) -d(x)<2\delta ,~\nabla d(x) \cdot z >4\delta\}} \frac{dz}{|z|^{n+2s}}\\&
=2 \int_{\{d(x+z) -d(x)<2\delta,~\nabla d(x) \cdot z>4\delta \}}\one_{\{0<d(x+z) -d(x)<2\delta\}}(z) \,\frac{dz}{|z|^{n+2s}}.\\&
=2 \int_{\{d(x+z) -d(x)<2\delta,~4\delta<\nabla d(x) \cdot z< \tau_0\}} \one_{\{0<d(x+z) -d(x)<2\delta\}}(z)\frac{dz}{|z|^{n+2s}}\\&
\quad + 2 \int_{\{\nabla d(x) \cdot z>\tau_0\}} \one_{\{0<d(x+z) -d(x)<2\delta\}}(z)\frac{dz}{|z|^{n+2s}}.
 \end{align*} 
 By Lemma \ref{lemm2s<1/2_1} with $\sigma=2\delta$, we have that 
 $$\one_{\{d(x+z) -d(x)<2\delta,~4\delta<\nabla d(x) \cdot z<\tau_0 \}}\in L^1(\R^n)$$
 uniformly in $\delta$. Therefore, by
\eqref{chiclosetogamma} and the Dominated Convergence Theorem,
\begin{equation}\label{J3_for_I3}
    J_3 = o_\delta(1).
\end{equation}
From \eqref{J1_for_I3}, \eqref{J2_for_I3} and \eqref{J3_for_I3}, we get
\begin{equation}\label{I3_estimate}
    I_3 = o_\ep(1) + o_\delta(1).
\end{equation}
With a similar argument, we also get
\begin{equation}\label{I4_estimate}
    I_4 = o_\ep(1) + o_\delta(1).
\end{equation}
Combining \eqref{aepsplits<1/2},  \eqref{aepsplits<1/2_bis}, \eqref{I1_estimate}, \eqref{I2_estimate}, \eqref{I3_estimate} and \eqref{I4_estimate}, we obtain
\begin{equation}\label{bbarestimate:final:lemmak}
\bar b_\ep=\kappa[x,d] + o_\ep(1) + o_\delta(1)+ O(R^{-2s}).
\end{equation}
It remains to estimate $\bar c_\ep$. Since $|\nabla d(x)|=1$, we can write
\begin{equation*}
    \bar c_\ep=\frac{1}{\ep^{2s}}\left[ \left(1+ \ep^{2+\frac{2s}{1-2s}} \right)^s -1 \right]W' \left( \phi \left( \frac{d(t,x)}{\ep} \right) \right),
\end{equation*}
and by H{\"o}lder continuity, 
\begin{equation*}
   | \bar c_\ep|\leq C\ep^{\frac{2s^2}{1-2s}}.
    \end{equation*}
The estimate on $\bar c_\ep$, combined with  \eqref{bbarestimate:final:lemmak}, yields  
 the desired result.

\section{Proof of Lemma \ref{lem:a_ep and frac laplacians}}\label{sec: proof of a_ep and frac laplacians}
\begin{proof}
 For simplicity, we drop the dependence on $t$. By Lemma  \ref{lem:1 to n facrional Laplacian} applied to $v=\phi$  with $e = |\nabla d(x)|$, we have 
\begin{align*}
     \int_{\mathbb{R}^n} \left( \phi \left( \frac{d(x)}{\ep} + \nabla d(x) \cdot z  \right) - \phi \left( \frac{d(x)}{\ep}\right) \right) \frac{dz}{|z|^{n+2s}} = |\nabla d(x)|^{2s} C_{n,s} \mathcal{I}_1^s[\phi]\left( \frac{d(x)}{\ep} \right). 
\end{align*}
Therefore (recall the definition of $\bar b_\ep$ in \eqref{b_epsilon}), we obtain
\begin{align*}
     \mathcal{I}_n^s \left[ \phi \left( \frac{d( \cdot)}{\ep} \right) \right](x) = &\int_{\mathbb{R}^n} \left( \phi \left( \frac{d(x+z)}{\ep} \right) - \phi \left( \frac{d(x)}{\ep} \right) \right) \frac{dz}{|z|^{n+2s}} \\
    = & \int_{\mathbb{R}^n} \left( \phi \left( \frac{d(x)+ \nabla d(x) \cdot z  }{\ep} \right) - \phi \left( \frac{d(x)}{\ep}\right) \right) \frac{dz}{|z|^{n+2s}} \\
    & + \int_{\mathbb{R}^n} \left( \phi \left( \frac{d(x+ z)}{\ep}\right) - \phi \left( \frac{d(x)+ \nabla d(x) \cdot z}{\ep}   \right) \right) \frac{dz}{|z|^{n+2s}} \\
    =&\frac{1}{\ep^{2s}} \int_{\mathbb{R}^n} \left( \phi \left( \frac{d(x) }{\ep} + \nabla d(x) \cdot z \right) - \phi \left( \frac{d(x)}{\ep}\right) \right) \frac{dz}{|z|^{n+2s}}+ \bar b_\ep + O(R^{-2s})\\
    = & \frac{1}{\ep^{2s}} |\nabla d(x)|^{2s} C_{n,s} \mathcal{I}_1^s[\phi]\left( \frac{d(x)}{\ep} \right) + \bar b_\ep + O(R^{-2s}). 
\end{align*}
 Subtracting $ \frac{C_{n,s}}{\ep^{2s}} \mathcal{I}_1^s[\phi]\left( \frac{d(\cdot)}{\ep} \right)$ from both sides, and using  that $\phi$ solves \eqref{eq:standing wave} (recall the  definition of $\bar c_\ep$ in \eqref{c_epsilon}), we get 
\begin{align*}
     \mathcal{I}_n^s& \left[ \phi \left( \frac{d( \cdot)}{\ep} \right) \right](x) - \frac{C_{n,s} }{\ep^{2s}}\mathcal{I}_1^s[\phi]\left( \frac{d(x)}{\ep} \right) \\
     &=\frac{1}{\ep^{2s}} (|\nabla d(x)|^{2s} -1)C_{n,s} W'\left(\phi\left( \frac{d(x)}{\ep} \right)\right)  + \bar b_\ep + O(R^{-2s}) \\
     &=  \bar a_\ep + O(R^{-2s}) + \tilde{c}_\ep,
\end{align*}
where 
\begin{align*}
    \tilde{c}_\ep :=\frac{1}{\ep^{2s}} \left[ |\nabla d(x)|^{2s} - \left(|\nabla d(x)|^{2} + \ep^{2+\frac{2s}{1-2s}} \right)^s \right] W' \left( \phi \left( \frac{d(x)}{\ep} \right) \right).
\end{align*}
By  H\"older continuity we get
\begin{align*}
    |\tilde{c}_\ep| \le C\ep^{\frac{2s^2}{1-2s}} 
\end{align*}
and  thus the desired result follows.
\end{proof}

\section{Proof of Lemma \ref{lem:a_estimates}}\label{sec:proof of a_estimates}
For ease of notation, throughout this section we omit the dependence on $t$.
First note that by the regularity of $\phi$ and $d$, there is some $\theta \in (0,1)$ and $C>0$ such that
\begin{equation}\label{mean_taylor_est}
    \begin{aligned}
        &\left| \phi \left( \frac{d(x+z) }{\ep} \right) - \phi \left( \frac{d(x)+ \nabla d(x) \cdot z}{\ep}  \right) \right| \\
        &\le \dot{\phi} \left( \theta \frac{d(x+ z)}{\ep} + (1-\theta)  \frac{d(x) + \nabla d(x)\cdot z}{\ep}  \right) C  \frac{|z|^2}{\ep}\\&
        =\dot{\phi} \left( \frac{d( x)}{\ep}+\theta \frac{d(x+z)-d(x)}{\ep} + (1-\theta)  \frac{ \nabla d(x)\cdot z }{\ep} \right) C \frac{|z|^2}{\ep}.
    \end{aligned}
\end{equation}
We will make several times  the change of variables $z=Ty$, where $T$ is an orthonormal matrix such that
\begin{equation}\label{changevar_s<12_bis}
\nabla d(x) \cdot (Ty) = c_1y_n,
\end{equation}
with $c_1=|\nabla d(x)|$ and  $y = (y', y_n)$,  $y' \in \mathbb{R}^{n-1}$.
Moreover, we will need the following preliminary results.
\begin{lem}\label{gamma_delta_int}
There exists $C>0$ such that for all $\tau,\, \gamma > 0$,
\begin{equation}\label{eq: gamma_delta_int}
    \int_{ \{ |y'| > \gamma, |y_n| < \tau \} } \frac{dy}{|y|^{n+2s}} \le C \frac{\tau}{\gamma^{1+2s}}.
\end{equation}
\end{lem}
\begin{proof}
Making the change of variable $w' = \frac{y'}{|y_n|}$, we have
    \begin{align*}
         \int_{ \{ |y'| > \gamma, |y_n| <\tau \} } \frac{dy}{|y|^{n+2s}}  & = \int_{ \{ |y_n| <\tau \}} \frac{dy_n}{|y_n|^{n+2s}} \int_{ \{|y'| > \gamma \} } \frac{dy'}{ \left( 1 + \frac{|y'|^2}{|y_n|^2} \right)^{\frac{n+2s}{2}}} \\
         & = \int_{ \{ |y_n| < \tau \} } \frac{d y_n}{|y_n|^{1+2s}} \int_{ \{ |w'|> \frac{\gamma}{|y_n|} \} } \frac{dw'}{(1 +|w'|^2)^{\frac{n+2s}{2}}} \\ 
         & \le \int_{ \{ |y_n| < \tau \} } \frac{dy_n}{|y_n|^{1+2s}} \int_{ \{ |w'|> \frac{\gamma}{|y_n|} \} } \frac{dw'}{|w'|^{n+2s}} \\
         & = C \int_{ \{ |y_n| < \tau \} } \frac{dy_n}{|y_n|^{1+2s}} \int_{ \frac{\gamma}{|y_n|}}^{\infty} \frac{d \rho}{|\rho|^{2+2s}} \\
         & = C  \int_{ \{ |y_n| <\tau \} } \frac{1}{|y_n|^{1+2s}} \frac{|y_n|^{1+2s}}{\gamma^{1+2s}} dy_n\\
         & = C\frac{\tau}{\gamma^{1+2s}}.
    \end{align*}
\end{proof}

\begin{lem}\label{0<y_n<tau_lemmadotphi}
Assume $|\nabla d(x)|=1$. Then there exist $\tau_0>0$ and $C>0$ such that for all $0<\tau\leq\tau_0$, and $1\leq R\leq \infty$, 
\begin{equation*}\label{0<y_n<tau_lemma_eqdotphi}
\left| \int_{\{|\nabla d(x) \cdot z |<\tau,\, |z|<R\}} 
 \( \dot\phi\( \frac{ d(x+z)}{\ep}\) - \dot\phi\( \frac{ d(x) + \nabla d(x) \cdot z}{\ep}\) \) \frac{dz}{\abs{z}^{n+2s}}\right|\leq C\tau^{\frac12-s}. 
\end{equation*}
\end{lem}
\begin{proof}
The proof  is similar to that of Lemma \ref{0<y_n<tau_lemma}, but it is  more involved due to the fact that  $\dot\phi$ is not a monotone function. 
We perform the usual  Taylor expansion  of $d$,  but we make the error term explicit, for $\lambda \in (0,1)$,
$$d(x+z) - d(x)=  \nabla d(x) \cdot z
  + \int_0^1D^2d(x+\lambda z)(1-\lambda)\,d\lambda\, z\cdot z.
 $$
Assume $\tau<1/2$ and let  $0<r<1/2$ to be determined. Making  the change of variables $z=Ty$ with $T$ as in \eqref{changevar_s<12} (and $c_1=1$), we get
 \begin{align*}
&\int_{\{|\nabla d(x) \cdot z |<\tau,\,|z|<R\}} 
 \( \dot\phi\( \frac{ d(x+z)}{\ep}\) - \dot\phi\( \frac{ d(x) + \nabla d(x) \cdot z}{\ep}\) \) \frac{dz}{\abs{z}^{n+2s}}\\&
=\int_{\{|y_n|<\tau,\,|y|<R\}} 
 \( \dot\phi\( \frac{ d(x+Ty)}{\ep}\) - \dot\phi\( \frac{ d(x) + y_n}{\ep}\) \) \frac{dy}{\abs{y}^{n+2s}}\\&
 =\int_{\{|y_n|<\tau,\,|y'|<r\}} (\ldots)+\int_{\{|y_n|<\tau,\,|y'|>r, |y|<R\}} (\ldots)\\&
 =:I_1+I_2. 
 \end{align*}
Next, for $I_1$ we make the further change of variable $t=y_n/p$, and use  polar coordinates  $y'=p\theta$ with $p>0$ and $\theta\in \mathcal{S}^{n-2}$. This gives
 \begin{align*}
I_1&
=\int_{\mathcal{S}^{n-2}}d\theta\int_0^r \frac{dp}{p^{1+2s}}\int_{-\frac{\tau}{p}}^{\frac{\tau}{p}}\(\dot\phi\(   \frac{ d(x) + tp+A(x,p,\theta,t)}{\ep}\) - \dot\phi\(\frac{d(x) + tp}{\ep}\) \) \frac{dt}{(1+t^2)^\frac{n+2s}{2}},
 \end{align*}
 where the function $A$ has the form 
 $$A(x,p,\theta,t)=p^2\int_0^1D^2d(x+\lambda pT(\theta,t))(1-\lambda)\,d \lambda \, T(\theta, t)\cdot T(\theta,t).$$
 Note that  for $0<p<1$ and $p|t|<1$, 
 \begin{equation}\label{ApartialAtime}A =O(p^2(1+t^2)),\quad\partial_tA =O(p^2(1+|t|))\quad\text{and}\quad\partial^2_tA =O(p^2). \end{equation}
$I_1$ can be rewritten as
 \begin{align*}
  I_1
 &=\frac{1}{\ep}\int_{\mathcal{S}^{n-2}}d\theta\int_0^r \frac{dp}{p^{1+2s}}\int_{-\frac{\tau}{p}}^{\frac{\tau}{p}} dt\, \frac{A(x,p,\theta,t)}{(1+t^2)^\frac{n+2s}{2}}\int_0^1\ddot \phi\(  \frac{d(x) + tp+\lambda A(x,p,\theta,t)}{\ep}\)\, d\lambda\\&
 =\int_{\mathcal{S}^{n-2}}d\theta\int_0^r \frac{dp}{p^{1+2s}}\int_{-\frac{\tau}{p}}^{\frac{\tau}{p}} dt\, \frac{A(x,p,\theta,t)}{(1+t^2)^\frac{n+2s}{2}}\\&
 \quad \cdot \int_0^1\partial_t 
 \left[\dot \phi\(\frac{d(x) + tp+\lambda A(x,p,\theta,t)}{\ep}\)\right] \frac{d \lambda}{p +\partial_t A(x,p,\theta,t)}.
   \end{align*}
  By \eqref{ApartialAtime}, for $0<p< r$ and $p|t|< \tau$, with $r$ and $\tau$  sufficiently small, we have
  \begin{equation}\label{ApartialAtime-part2}
  p +\partial_t A(x,p,\theta,t)\geq p[1-C(p|t|+p)]\geq \frac{p}{2}.
  \end{equation} 
   Integrating by parts with respect to $t$, we obtain
    \begin{align*}
I_1&= \int_{\mathcal{S}^{n-2}}d\theta\int_0^r \frac{dp}{p^{1+2s}}\int_0^1d\lambda \left\{ \dot \phi\(\frac{d(x) + tp+\lambda A(x,p,\theta,t)}{\ep}\)\right.\\&
\quad\cdot\left.\frac{A(x,p,\theta,t)}{(1+t^2)^\frac{n+2s}{2}(p +\partial_t A(x,p,\theta,t))}
\Big\rvert_{t=-\frac{\tau}{p}}^{t=\frac{\tau}{p}}\right.\\&
\quad-\left. \int_{-\frac{\tau}{p}}^{\frac{\tau}{p}}\dot \phi\(\frac{d(x) + tp+\lambda A(x,p,\theta,t)}{\ep}\)\partial_t\left[\frac{A(x,p,\theta,t)}{(1+t^2)^\frac{n+2s}{2}(p +\partial_t A(x,p,\theta,t))}\right]dt\right\}. 
    \end{align*}
  By   \eqref{ApartialAtime} and \eqref{ApartialAtime-part2}, for $0<p< r$, 
  $$\frac{A}{(1+t^2)^\frac{n+2s}{2}(p +\partial_t A)}=O\(\frac{p}{(1+t^2)^\frac{n+2s-2}{2}}\),$$
  and 
   \begin{align*}
   \partial_t\left[\frac{A}{(1+t^2)^\frac{n+2s}{2}(p +\partial_t A)}\right]&=\frac{\partial_t A}{(1+t^2)^\frac{n+2s}{2}(p +\partial_t A)}-(n+2s)\frac{tA}{(1+t^2)^\frac{n+2s+2}{2}(p +\partial_t A)}\\&
   \quad-\frac{A\partial_{tt}A}{(1+t^2)^\frac{n+2s}{2}(p+\partial_tA)^2}\\&
   =O\(\frac{p}{(1+t^2)^\frac{n+2s-1}{2}}\)+O\(\frac{p^2}{(1+t^2)^\frac{n+2s-2}{2}}\). 
  \end{align*}
  Therefore, 
   \begin{align*}
   &\int_{\mathcal{S}^{n-2}}d\theta\int_0^r \frac{dp}{p^{1+2s}}\int_0^1\dot \phi\(\frac{d(x) + tp+\lambda A(x,p,\theta,t)}{\ep}\)\frac{A(x,p,\theta,t)}{(1+t^2)^\frac{n+2s}{2}(p +\partial_t A(x,p,\theta,t))}
\Big\rvert_{t=-\frac{\tau}{p}}^{t=\frac{\tau}{p}}d\lambda \\&
 \leq  C\int_0^{r} \frac{dp}{ p^{2s}}=Cr^{1-2s},
\end{align*}
and 
 \begin{align*}
 &\int_{\mathcal{S}^{n-2}}d\theta\int_0^r \frac{dp}{p^{1+2s}}\int_0^1d \lambda \int_{-\frac{\tau}{p}}^{\frac{\tau}{p}}\dot \phi\(\frac{d(x) + tp+\lambda A(x,p,\theta,t)}{\ep}\)\\&
 \quad\cdot\partial_t\left[\frac{A(x,t,p,\theta)}{(1+t^2)^\frac{n+2s}{2}(p +\partial_t A(x,p,\theta,t))}\right]dt\\&
 \leq C\int_0^r \frac{dp}{p^{2s}} \int_{-\frac{\tau}{p}}^{\frac{\tau}{p}}\left\{\frac{1}{(1+t^2)^\frac{n+2s-1}{2}}+\frac{p}{(1+t^2)^\frac{n+2s-2}{2}}\right\}dt\\&
 \leq   C\int_0^r \frac{dp}{p^{2s}} \left\{1+p\left|\int_{1}^{\frac{\tau}{p}}\frac{dt}{t^{2s}}\right|\right\}\\&
 = C\int_0^r \frac{dp}{p^{2s}} \left\{1+\tau^{1-2s}p^{2s}\right\}\\&
 \leq Cr^{1-2s}. 
 \end{align*}
 We infer that 
 \begin{align*}
 |I_1|& \leq  C r^{1-2s}. 
 \end{align*}
We also estimate
  \begin{align*}
| I_2| \leq C\int_{r}^\infty \frac{dp}{p^{1+2s}}\int_{-\frac{\tau}{p}}^{\frac{\tau}{p}}dt=C\frac{\tau}{r^{1+2s}}. 
  \end{align*}
Choosing $r=\tau^\frac12$, we obtain
$$\int_{\{|\nabla d(x) \cdot z |<\tau,\,|z|<R\}} 
 \( \dot\phi\( \frac{ d(x+z)}{\ep}\) - \dot\phi\( \frac{ d(x) + \nabla d(x) \cdot z}{\ep}\) \) \frac{dz}{\abs{z}^{n+2s}}\leq C\tau^{\frac12-s}. $$
The lower bound can be  proven similarly. This concludes the proof of the lemma.   
\end{proof}

Now we are ready to prove Lemma \ref{lem:a_estimates}.
\subsection{Proof of \eqref{a_near}} We consider two cases: $|d(x)|<\rho$ and  $|d(x)|\geq \rho$, with $\rho $ as in Definition \ref{defn:extension}. 
First, assume $|d(x)|<\rho$, then $|\nabla d(x)|=1$. 
We begin by estimating $\bar b_\ep$ as in \eqref{b_epsilon} from the definition of $\bar a_\ep$ in  \eqref{aepsilondef}. We split
    \begin{align*}
       \bar  b_\ep = &\int_{ \left\{|z| < R \right\} } \left( \phi \left( \frac{d( x+z)}{\ep} \right) - \phi \left( \frac{d(x) + \nabla d(x) \cdot z }{\ep}\right) \right) \frac{dz}{|z|^{n+2s}} \\
        =&\int_{\{ d(x+z) > d(x),~\nabla d(x) \cdot z < 0\}}(\ldots)+\int_{\{d(x+z) < d(x),~\nabla d(x) \cdot z > 0\}}(\ldots)\\
         &+\int_{\{d(x+z) > d(x),~\nabla d(x) \cdot z > 0\}}(\ldots)+\int_{\{d(x+z) < d(x),~\nabla d(x) \cdot z < 0\}}(\ldots)\\
=:& I_1+I_2+I_3+I_4.
    \end{align*}
By Proposition \ref{prop:fmc_finite_result}, $I_1$ and $I_2$ are bounded uniformly in $\ep$. Thus, it is enough to show that
\begin{equation}\label{I3I4bounda_near}|I_3|,\,|I_4|\leq C.\end{equation}
Let $\tau>0$ to be chosen, and further split,
\begin{align*}
    I_3 &= \int_{\{d(x+z) > d(x),~\nabla d(x) \cdot z > 0\}} \left( \phi\left( \frac{d(x+ z)}{\ep}\right) - \phi\left( \frac{ d(x) + \nabla d(x) \cdot z}{\ep}\right) \right) \frac{dz}{\abs{z}^{n+2s}} \\
    & =  \int_{\{d(x+z) > d(x),~\nabla d(x) \cdot z > \tau \}}(\ldots) + \int_{\{d(x+z) > d(x),~ 0 <\nabla d(x) \cdot z <\tau\}}(\ldots) \\& =: J_1 + J_2. 
\end{align*}
Making the change of variables $z=Ty$ with $T$ as in \eqref{changevar_s<12_bis} (and  $c_1=1$), we get
\begin{align*}
    |J_1| \le 2\int_{ \{y_n > \tau \} } \frac{dy}{|y|^{n+2s}} \le 2\int_{ \{ |y| > \tau\}} \frac{dy}{|y|^{n+2s}} \le C \tau^{-2s}.
\end{align*}
By Lemma \ref{0<y_n<tau_lemma}, choosing $\tau=\tau_0$ with $\tau_0$ as in the lemma, we have  $$|J_2|\leq C\tau^{\frac12-s}.$$
The estimates on $J_1$ and  $J_2$  imply \eqref{I3I4bounda_near} for  $I_3$. 

The  bound for  $I_4$ follows similarly. Thus, we have shown that 
\begin{equation}\label{b_ep bound_a_near}
|\bar b_\ep|\leq C.
\end{equation}
  Finally, we estimate $\bar c_\ep$ as in \eqref{c_epsilon}. Since $ |\nabla d(x)| =1$, 
we have 
$$\bar  c_\ep[d](t,x) :=\frac{1}{\ep^{2s}}\left[ \left(1 + \ep^{2+\frac{2s}{1-2s} } \right)^s -1 \right]W' \left( \phi \left( \frac{d(t,x)}{\ep} \right) \right),$$
and  by H\"older continuity,
\begin{align*}
    |\bar c_\ep| \le C \ep^{\frac{2s^2}{1-2s}} .
\end{align*}
 This estimate, combined with \eqref{b_ep bound_a_near}, gives  \eqref{a_near} for the case   $|d(x)|<\rho$.
 
Next,  assume $|d(x)|\geq\rho$. Again, we begin by estimating $\bar b_\ep$ first. Let $c>0$ be so small   that if $|z| \le c\rho$, then
   $\left| d(x+ z) -d(x)\right| \leq \rho/4$ and $\left|\nabla d(x) \cdot z  \right| \leq \rho/4$. We  write
\begin{align*}
   \bar  b_\ep &=  \int_{ \left\{|z| < R \right\} } \left( \phi \left( \frac{d( x+ z)}{\ep} \right) - \phi \left( \frac{d(x)+ \nabla d(x) \cdot z}{\ep}  \right) \right) \frac{dz}{|z|^{n+2s}} \\
    &=  \int_{\{|z| <c\rho\} }(\ldots) + \int_{ \{ c\rho < |z| < R \}} (\ldots)  \\
    & =: I + II.
\end{align*}
Using \eqref{mean_taylor_est} and  estimate \eqref{eq:asymptotics for phi dot} for $\dot{\phi}$,   we get 
\begin{align*}
    |I|&\le   C  \int_{\{|z| <c\rho\} }\frac{\ep^{2s}}{\left|d(x)+\theta (d(x+z)-d(x))+ (1-\theta) \nabla d(x) \cdot z \right|^{2s+1}}
  \frac{dz}{|z|^{n+2s-2}} \\&
   \leq C\frac{\ep^{2s}}{\left(|d(x)|-\frac{\rho}{2}\right)^{2s+1}}  \int_{\{|z| <c\rho\} } \frac{dz}{|z|^{n+2s-2}}\\&
 \leq C \rho^{1-4s} \ep^{2s}.
\end{align*}
For $II$, we have
\begin{align*}
    |II| \le 2\int_{ \{ c\rho < |z| < R \}} \frac{dz}{|z|^{n+2s}} \le \frac{C}{\rho^{2s}}. 
\end{align*}
Combining the estimates of $I$ and $II$, we obtain 
\begin{equation}\label{bepestimatea_far}|\bar b_\ep|\leq \frac{C}{\rho^{2s}}.
\end{equation}
Next, we estimate $\bar c_\ep$. Let $H(\cdot)$ denote the Heaviside function. 
Using that $W'\(H \left( \frac{d(x)}{\ep} \right)\)=0$,  by a Taylor's expansion around $H \left( \frac{d(x)}{\ep} \right)$, and for some $\xi_0\in\R$, we get
\begin{align}\label{W'estimatecepestafar}
   \left| W'\left( \phi \left( \frac{d(x)}{\ep}\right) \right)\right|= \left|W''(\xi_0) \left( \phi \left( \frac{d(x)}{\ep}\right) - H \left( \frac{d(x)}{\ep} \right)\right)\right|\leq \frac{C\ep^{2s}}{|d(x)|^{2s}},
\end{align}
where we used  estimate \eqref{eq:asymptotics for phi} for the last inequality. From  \eqref{W'estimatecepestafar}, we finally get
\begin{align*}
    |\bar c_\ep| \le  \frac{C}{|d(x)|^{2s}} \le  \frac{C}{\rho^{2s}}. 
\end{align*}
From the  estimate on $\bar c_\ep$ and \eqref{bepestimatea_far}, 
\eqref{a_near} for the case   $|d(x)|\geq \rho$ follows.

\subsection{Proof of \eqref{partial_a_near}}
We consider two cases: $|d(x)|<\delta_0$ and  $|d(x)|\geq \delta_0$, with $0<\delta_0<\rho$ to be determined, and $\rho$ as in Definition \ref{defn:extension}. First, assume  $|d(x)|< \delta_0$. We will establish the estimate for  $\nabla_x \bar a_\ep$; the estimate for  $\partial_t \bar a_\ep$ follows by a similar argument. 
 Since $\delta_0<\rho$, we have that   $|\nabla d(x)|=1$. 
 We begin by estimating  $\nabla_x \bar b_\ep$,  and compute
 

\begin{equation}\label{partialxba-close}\begin{split}
   \partial_{x_i}\bar  b_\ep = & \int_{\{ |z| < R\}} \bigg[ \dot{\phi} \left( \frac{d(x+ z)}{\ep}\right) \frac{\partial_{x_i}d(x+z)}{\ep} \\
   &- \dot{\phi}\left( \frac{d(x)+ \nabla d(x) \cdot z}{\ep}  \right) \frac{\partial_{x_i}d(x) + \nabla \partial_{x_i}d(x) \cdot z }{\ep} \bigg] \frac{dz}{|z|^{n+2s}}  \\
   =  &\ep^{-1}\left\{ \int_{\{ |z| < R \}} \bigg [ \dot{\phi} \left( \frac{d(x+z)}{\ep}\right) (\partial_{x_i}d(x+ z) - \partial_{x_i}d(x))\right.  \\
    & + \left(  \dot{\phi} \left( \frac{d(x+ z)}{\ep}\right) -  \dot{\phi}\left( \frac{d(x) + \nabla d(x) \cdot z }{\ep}\right) \right) \partial_{x_i}d(x) \\
    &\left.- \dot{\phi}\left( \frac{d(x) + \nabla d(x) \cdot z}{\ep} \right)  \nabla \partial_{x_i}d(x) \cdot z  \bigg]  \frac{dz}{|z|^{n+2s}}\right\} \\
     =:& \ep^{-1}(I + II+ III).
\end{split}
\end{equation}
We first estimate $I$.  
Let $0<\gamma <1$ to be chosen, and split the integral as follows
\begin{align*}
    I= \int_{\left\{  |z| <\gamma \right\}} (\ldots) + \int_{\left\{ \gamma < |z| <R \right\}} (\ldots) =: I_1+ I_2.
\end{align*}
For $I_1$, we obtain
\begin{align*}
    |I_1| \le C  \int_{ \{|z| < \gamma\} } \frac{dz}{|z|^{n+2s-1}} \le C \gamma^{1-2s}.
\end{align*}
For $I_2$,  using that $\{z\,:\,d(x+z)=0\}$ is a smooth surface,  and applying estimate \eqref{eq:asymptotics for phi dot} for $\dot \phi$, we have
\begin{align*}
    |I_2| &\le C\int_{ \{ |d(x+z)|< \delta_0,~|z|> \gamma\} } \frac{dz}{|z|^{n+2s}}  + \int_{ \{ |d(x+z)|\geq \delta_0,~ |z|>\gamma \} } \dot{\phi} \left(\frac{d(x+z)}{\ep} \right)   \frac{dz}{|z|^{n+2s}} \\
    &\le  \frac{C}{\gamma^{n+2s}} \int_{ \{ |d(x+z)| < \delta_0 \} } dz + C \frac{\ep^{2s+1}}{\delta_0^{2s+1}} \int_{\{ |z|>\gamma\}} \frac{dz}{|z|^{n+2s}}\\
    &\le C \left(\frac{\delta_0}{\gamma^{n+2s}} +\frac{\ep^{2s+1}}{\delta_0^{2s+1} \gamma^{2s}}\right).
\end{align*}
Combining the estimates for $I_1$ and $I_2$, we obtain
\begin{equation}\label{Iestnablaaxnear}
|I|\leq C\left(\gamma^{1-2s}+\frac{\delta_0}{\gamma^{n+2s}}+\frac{\ep^{2s+1}}{\delta_0^{2s+1} \gamma^{2s}}\right).
\end{equation} 
Next, we estimate $II$. 
We split
\begin{align*}
    II& =  \int_{\left\{  |\nabla d(x)\cdot z| <\gamma, |z|<R\right\}} (\ldots) + \int_{\left\{ |\nabla d(x)\cdot z| >\gamma,\,  |z| <R \right\}} (\ldots) =: II_1+ II_2.
    \end{align*}
  By Lemma \ref{0<y_n<tau_lemmadotphi}, for $\gamma\leq \tau_0$ and $\tau_0$ as in the lemma,
  $$|II_1|\leq C\gamma^{\frac12-s}.$$  
For $II_2$, we  further split
\begin{align*}
    |II_2| &\le C \int_{\left\{ |\nabla d(x)\cdot z| >\gamma  \right\}}\dot{\phi} \left( \frac{d(x+z)}{\ep}\right)\frac{dz}{|z|^{n+2s}}  
    +C \int_{\left\{ |\nabla d(x)\cdot z| >\gamma  \right\}} \dot{\phi}\left( \frac{d(x)+\nabla d(x)\cdot z}{\ep} \right)  \frac{dz}{|z|^{n+2s}}.\\
        & = : J_1 + J_2.
\end{align*}
 For $J_1$, similarly to the estimate of  $I_2$, we get
$$|J_1|\leq  C\left( \frac{\delta_0}{\gamma^{n+2s}} +\frac{\ep^{2s+1}}{\delta_0^{2s+1} \gamma^{2s}}\right).$$
For $J_2$,   note that $|d(x) + \nabla d(x)\cdot z|  >\delta_0$ if  $|\nabla d(x)\cdot z|>\gamma$, $|d(x)|<\delta_0$, and we choose $\gamma\geq 2\delta_0$. 
Therefore, by estimate \eqref{eq:asymptotics for phi dot} for $\dot \phi$, we have 
\begin{align*}
    |J_2| \le C \frac{\ep^{2s+1}}{\delta_0^{2s+1}} \int_{ \{ |z| > \gamma \} } \frac{dz}{|z|^{n+2s}  } = C \frac{\ep^{2s+1}}{\delta_0^{2s+1} \gamma^{2s}}.
\end{align*}
From  the estimates of $II_1$, $J_1$ and $J_2$ we obtain
\begin{equation}\label{Iestnablaaxnear_II} |II|\leq  C \left(\gamma^{\frac12-s}+\frac{\delta_0}{\gamma^{n+2s}} +\frac{\ep^{2s+1}}{\delta_0^{2s+1} \gamma^{2s}}\right). 
\end{equation}
We finally estimate $III$.   We split
$$III=\int_{\{| z| <2\gamma,\, |z|<R  \}}(\ldots)+\int_{\{|z| >2\gamma ,\, |z|<R   \}}(\ldots):=III_1+III_2.$$
 Similarly to  the estimate of $I_1$,  we have
 $$|III_1| \le C  \int_{ \{|z| <2 \gamma\} } \frac{dz}{|z|^{n+2s-1}} \le C \gamma^{1-2s}.$$
 We further split
 $$III_2=\int_{\{2\gamma<|z|<R,\, |\nabla d(x)\cdot z|<2\delta_0  \}}(\ldots)+\int_{\{2\gamma<|z|<R,\, |\nabla d(x)\cdot z|>2\delta_0  \}}(\ldots)=:J_1+J_2.$$
 To estimate $J_1$, we make the change of variables $z=Ty$ with $T$ as in \eqref{changevar_s<12} (and $c_1=1$). Note that if $|y|>2\gamma$, $|y_n|<2\delta_0$ and 
 $\gamma\geq 2\delta_0$, then $|y'|>\gamma$. Therefore, by Lemma \ref{gamma_delta_int}, we have
  \begin{equation*}\begin{split}|J_1|&\leq C\int_{\{2\gamma<|z|<R,\,  |\nabla d(x)\cdot z|<2\delta_0\} }|z|  \frac{dz}{|z|^{n+2s}}\\&
  \leq CR \int_{\{|z|>2\gamma,\,  |\nabla d(x)\cdot z|<2\delta_0\} }  \frac{dz}{|z|^{n+2s}}\\&
  =  CR \int_{\{|y|>2\gamma,\,  |y_n|<2\delta_0\} }  \frac{dy}{|y|^{n+2s}}\\&
  \leq CR \int_{\{|y'|>\gamma,\,  | y_n|<2\delta_0\} }  \frac{dy}{|y|^{n+2s}}\\&
  \leq \frac{CR\delta_0}{\gamma^{1+2s}}. 
  \end{split} \end{equation*}
We finally estimate $J_2$.  For $|d(x)|<\delta_0$ and $|\nabla d(x)\cdot z|> 2\delta_0$, we have that  $|d(x)+\nabla d(x)\cdot z|>\delta_0$. Therefore, by  estimate \eqref{eq:asymptotics for phi dot} for $\dot \phi$, we get
 \begin{equation*}\begin{split}|J_2|&\leq C\int_{\{\gamma< |z|<R,\,|\nabla d(x)\cdot z| >2\delta_0   \}}\dot{\phi}\left( \frac{d(x)+\nabla d(x)\cdot z}{\ep} \right)|z|  \frac{dz}{|z|^{n+2s}}
 \\&\leq CR\int_{\{|z|>\gamma,\,|\nabla d(x)\cdot z| >2\delta_0\}}\dot{\phi}\left( \frac{d(x)+\nabla d(x)\cdot z}{\ep} \right) \frac{dz}{|z|^{n+2s}}\\&
 \leq  C R\frac{\ep^{2s+1}}{\delta_0^{2s+1}}\int_{\{|z|>\gamma\}}\frac{dz}{|z|^{n+2s}}\\&
 =  C R\frac{\ep^{2s+1}}{\delta_0^{2s+1} \gamma^{2s}}.
\end{split} \end{equation*}
From the estimates on $J_1$ and $J_2$, we infer that 
$$|III_2|\leq CR\( \frac{\delta_0}{\gamma^{1+2s}}+\frac{\ep^{2s+1}}{\delta_0^{2s+1} \gamma^{2s}}\),$$
which, together  with the estimate on $III_1$, gives 
$$|III|\leq C \left(\gamma^{1-2s}+R \frac{\delta_0}{\gamma^{1+2s}}+R\frac{\ep^{2s+1}}{\delta_0^{2s+1} \gamma^{2s}}\right). $$
Combining the   estimate for $III$ with the estimates  for $I$ and $II$ in \eqref{Iestnablaaxnear} and \eqref{Iestnablaaxnear_II} (recall \eqref{partialxba-close}), we   obtain
\begin{align*}
    |\partial _{x_i} \bar b_\ep| \le C \ep^{-1}\left( \gamma^{\frac12-s} + R\frac{\delta_0}{\gamma^{n+2s}} + R\frac{\ep^{2s+1}}{\delta_0^{2s+1} \gamma^{2s}} \right).
\end{align*}
Choosing $\gamma=o_{\delta_0}(1)\geq 2\delta_0 $ such that  $\frac{\delta_0}{\gamma^{n+2s}} =o_{\delta_0}(1)$,    we obtain
\begin{equation}\label{partial_b_estimate}
    |\partial_{x_i} \bar b_\ep| \leq C\ep^{-1}R \left(o_{\delta_0}(1)+ \frac{\ep^{2s+1}}{\delta_0^{4s+1} } \right).
\end{equation}

\noindent Finally, we estimate $\nabla_x\bar c_\ep$. Since $|\nabla d| = 1$ in a neighborhood of $x$, we have 
\begin{align*}
    \partial_{x_i} \bar c_\ep = \frac{1}{\ep^{2s}}\left[ \(1+  \ep^{2+\frac{2s}{1-2s}}\)^s - 1 \right] W'' \left( \phi \left( \frac{d(x)}{\ep} \right) \right) \dot{\phi} \left( \frac{d(x)}{\ep} \right) \frac{\partial_{x_i }d(x)}{\ep},
\end{align*}
and by H{\"o}lder continuity  we obtain
\begin{equation}\label{partial_c_estimate}
    |\partial_{x_i}\bar  c_\ep| \le C\ep^{\frac{2s^2}{1-2s}-1} .
\end{equation}
Combining \eqref{partial_b_estimate} and \eqref{partial_c_estimate}, we obtain 
\begin{equation}\label{partial_a_near_new}
 |\partial_{x_i} \bar a_\ep| \le C\ep^{-1}R\left(o_\ep(1)+o_{\delta_0}(1)+ \frac{\ep^{2s+1}}{\delta_0^{4s+1} } \right)\quad \text{if }|d(x)|<\delta_0.
\end{equation}

Next, assume $|d(x)|\geq \delta_0$. 
 As before, we treat the two terms $\bar b_\ep$ and $\bar c_\ep$ in the definition of $\bar a_\ep$ in \eqref{aepsilondef} separately, starting with  $\bar b_\ep$. Recalling \eqref{partialxba-close}, 
we split the domain of integration into two parts: $|z|\le c\delta_0$ and $|z|>c\delta_0$ where $c>0$ is a small constant to be chosen.
Choose  $c>0$ sufficiently small such that,  for all $|z| \le c\delta_0$, it holds that  $\left| d(x+ z)-d(x)\right| \leq \delta_0/4$ and $\left|\nabla d(x) \cdot z \right| \leq \delta_0/4$. In particular, since  $|d(x)|\geq \delta_0$, 
this implies that $|d(x+z)|\geq\delta_0/2$ and $\left|d(x)+\nabla d(x) \cdot z \right| \geq\delta_0/2$ if $|z| \le c\delta_0$. 
Thus, we split 
\begin{align}\label{partialxbnablaafar}
   \partial_{x_i}\bar  b_\ep &=\ep^{-1}\left(\int_{\{|z|<c\delta_0\}}(\ldots)+\int_{\{|z|>c\delta_0\}}(\ldots)\right)=:\ep^{-1}(I+II).
   \end{align}
For $I$, using  estimate \eqref{mean_taylor_est} with $\phi$ replaced by $\dot \phi$,  and estimate  \eqref{eq:asymptotics for phi dot} for $\dot{\phi}$ and $\ddot{\phi}$,  we get 
\begin{equation}\label{Inabla_xafar}\begin{split}
    |I| \le & C \int_{\{ |z| < c \delta_0 \}} \bigg[  \dot{\phi} \left( \frac{d(x+z)}{\ep}\right)|z|  + \left|\ddot{\phi}\left(  \frac{d(x)}{\ep}  +\theta \frac{d(x+z)-d(x)}{\ep} + (1-\theta)\frac{\nabla d(x) \cdot z }{\ep} \right)\right|\frac{|z|^2}{\ep} \\
    & +  \dot{\phi} \left( \frac{d(x)+\nabla d(x) \cdot z }{\ep} \right)|z| \bigg] \frac{dz}{|z|^{n+2s}} \\
    \le & C \frac{\ep^{2s+1}}{\delta_0^{2s+1}} \int_{ \{ |z| <c \delta_0 \} } \frac{dz}{|z|^{n+2s-1}} + C \frac{\ep^{2s}}{\(|d(x)|-\frac{\delta_0}{2}\)^{2s+1}} \int_{ |z| < c \delta_0} \frac{dz}{|z|^{n+2s-2}} \\
    \leq& C\left( \frac{\ep^{2s+1}}{\delta_0^{4s}}+\frac{\ep^{2s}}{\delta_0^{4s-1}}\right). 
\end{split}
\end{equation}
For $II$, using that  $| \nabla \partial_{x_i}d(x) \cdot z |\leq CR$ for $|z|<R$, we get
\begin{equation}\label{IIsplitnablaxfar}\begin{split}
    |II|& \le  C \int_{\{ c \delta_0 < |z| <R  \}} \dot{\phi} \left( \frac{d(x+ z)}{\ep}\right) \frac{dz}{|z|^{n+2s}} + CR\int_{\{ c\delta_0  <|z| < R\}} \dot{\phi} \left( \frac{d(x)+ \nabla d(x) \cdot z }{\ep} \right)  \frac{dz}{|z|^{n+2s}} \\&=:II_1 + RII_2. 
\end{split}
\end{equation}
We first estimate $II_1$. To this end, we further split 
\begin{align*}
    II_1 & =\int_{ \{ c\delta_0 < |z| < R,\, |d(x+z)| \le\ep^{1\!/2}\} }(\dots) + \int_{ \{ c\delta_0 < |z| < R,\, |d(x+z)| > \ep^{1\!/2}\} }(\dots) \\
    & =: J_1+J_2.
\end{align*} 
Using that $\{z\,:\,d(x+z)=0\}$ is a smooth surface, we get
\begin{align*}
     J_1 \le   C \int_{ \{ c\delta_0 < |z| < R,\, |d(x+z)| \le \ep^{1\!/2}\} } \frac{dz}{|z|^{n+2s}} \le \frac{C}{\delta_0^{n+2s}} \int_{ \{|d(x+z)| \le \ep^{1\!/2}\} } dz \le 
    \frac{C \ep^{\frac12}}{\delta_0^{n+2s}}.
\end{align*}
Using estimate \eqref{eq:asymptotics for phi dot} for $\dot{\phi}$, we also have
\begin{align*}
     J_2 \le C\ep^{s+\frac12} \int_{ \{ c\delta_0 < |z| < R\} } \frac{dz}{|z|^{n+2s}} = C \frac{\ep^{s+\frac12}}{\delta_0^{2s}}. 
\end{align*}
 From the above estimates on $J_1$ and $J_2$ we infer that 
   \begin{equation}\label{II_1estimnablaax_far}II_1\leq \frac{C \ep^{\frac12}}{\delta_0^{n+2s}}.
   \end{equation}
We finally estimate $II_2$. If $\nabla_x d(x)=0$, then since $|d(x)|\geq\delta_0$, by  \eqref{eq:asymptotics for phi dot}  for $\dot{\phi}$ we have 
\begin{equation}\label{II_2est_nablax_far_zerpgra}II_2= \dot{\phi} \left( \frac{d(x)}{\ep} \right)  \int_{\{ c\delta_0  <|z| < R\}}\frac{dz}{|z|^{n+2s}}\leq \frac{C\ep^{2s+1}}{\delta_0^{2s+1} }\int_{\{ c\delta_0  <|z| < R\}}\frac{dz}{|z|^{n+2s}}\leq\frac{C\ep^{2s+1}}{\delta_0^{4s+1} }.
\end{equation}
Next, assume  $\nabla d(x)\neq 0$.
While keeping in mind that the integration is performed over the set $\{c\delta_0 < |z| < R\}$, we omit explicit  reference to this domain  in the following integrals for ease of notation, and split as follows
\begin{align*}
   II_2 &=\int_{\{ \nabla d(x)\cdot z < -d(x) -\ep^{1\!/2}\}} (\ldots) + \int_{\{|\nabla d(x)\cdot z +d(x)|\leq   \ep^{1\!/2}\}} (\ldots) + \int_{\{ {\nabla d(x)}\cdot z >-d(x) + \ep^{1\!/2} \}} (\ldots) \\
    & =:J_1+J_2+J_3.
\end{align*}
Notice that for both $J_1$ and $J_3$, we have
\begin{align*}
    \frac{\left| d(x) + \nabla d(x) \cdot z \right|}{\ep} \geq \ep^{-\frac12},
\end{align*}
and so for both $J_1$ and $J_3$,  we can use estimate  \eqref{eq:asymptotics for phi dot} for $\dot{\phi}$ and integrate for $|z| > c\delta_0$, to get
\begin{align*}
    J_1,\, J_3 \le C \ep^{s+\frac12} \int_{ \{  |z|>c\delta_0  \} } \frac{dz}{|z|^{n+2s}} \le C \frac{\ep^{s+\frac12}}{\delta_0^{2s}}.
\end{align*}
For $J_2$,  
performing the change of variable $z=Ty$ with $T$  as  in \eqref{changevar_s<12_bis}, we get
\begin{align*}
       J_2& 
 \le C \int_{ \left\{ \left|y_n+\frac{d(x) }{c_1}\right|\le   \frac{\ep^{1\!/2}}{c_1}\right\} } \frac{dy}{|y|^{n+2s}}.
\end{align*}
Notice that since $|d(x)|\geq \delta_0$, this last integral is well defined provided $\ep^{1\!/2}<\delta_0$  so that $y_n$ stays away from zero. Integrating first in $y'$ and then in $y_n$, and using H\"older continuity, we obtain
\begin{align*}
    J_2&\leq C \int_{ \left\{ \left|y_n+\frac{d(x) }{c_1}\right|\le   \frac{\ep^{1\!/2}}{c_1}\right\} } \frac{dy_n}{|y_n|^{1+2s}}
    = Cc_1^{2s} \frac{\left||d(x)+\ep^{\frac12}|^{2s} - |d(x)-\ep^{\frac12}|^{2s}\right|}{|d(x)+\ep^{\frac12}|^{2s}|d(x)-\ep^{\frac12}|^{2s}}\le \frac{C\ep^{s}}{\delta_0^{4s}},  
\end{align*} provided $\ep^{1\!/2}<\delta_0/2$.

From the estimates on $J_1$, $J_2$, $J_3$ and \eqref{II_2est_nablax_far_zerpgra}, we get
$$II_2\leq \frac{C\ep^{s}}{\delta_0^{4s}}. $$   Combining this with \eqref{II_1estimnablaax_far} (recall \eqref{IIsplitnablaxfar}) gives
$$|II|\leq C\left( \frac{ \ep^{\frac12}}{\delta_0^{n+2s}}+R\frac{\ep^{s}}{\delta_0^{4s}}\right).$$ From \eqref{partialxbnablaafar}, \eqref{Inabla_xafar} and the estimate for $II$, 
we finally obtain, for $\ep$ sufficiently small,  
\begin{align}\label{partialxbestnablaafar}
    |\partial_{x_i} \bar b_\ep| \leq  CR \ep^{-1} \frac{\ep^{s}}{\delta_0^{n+2s} }. 
\end{align}
It remains to estimate $\nabla_x \bar c_\ep$.  We compute
\begin{align*}
\partial_{x_i} \bar c_\ep (x)=&\frac{2s}{\ep^{2s}}\left(|\nabla d(x)|^{2} + \ep^{2+\frac{2s}{1-2s}} \right)^{s-1}W' \left( \phi \left( \frac{d(x)}{\ep} \right) \right)\sum_{j=1}^n\partial_{x_ix_j}d(x)\partial_{x_j} d(x)\\&
+\frac{1}{\ep^{2s}}\left(|\nabla d(x)|^{2} + \ep^{2+\frac{2s}{1-2s}} \right)^sW''\left( \phi \left( \frac{d(x)}{\ep} \right) \right) \dot{\phi} \left( \frac{d(x)}{\ep} \right) \frac{\partial_{x_i}d(x)}{\ep}.
\end{align*}
 Using   \eqref{W'estimatecepestafar}  and estimate \eqref{eq:asymptotics for phi dot} for $\dot \phi$, we get
\begin{align*}
        |\partial_{x_i} \bar c_\ep| \le & \frac{C}{\ep^{2s}} \left(|\nabla d(x)|^{2} + \ep^{2+\frac{2s}{1-2s}} \right)^{s-1} |\nabla d(x)| W' \left( \phi \left( \frac{d(x)}{\ep} \right) \right) + \frac{C}{\ep^{2s+1}} \dot{\phi} \left( \frac{d(x)}{\ep} \right) \\
        \le & C \left( \frac{|\nabla d(x)|^2}{|\nabla d(x)|^2 + \ep^{2+\frac{2s}{1-2s}}} \right)^{\frac{1}{2}} \left( |\nabla d(x)|^2 + \ep^{2+\frac{2s}{1-2s}} \right)^{s-\frac{1}{2}} \frac{1}{\delta_0^{2s}} 
        +\frac{C}{\delta_0^{2s+1}} \\
        \le &\frac{C}{ \delta_0^{2s}\left( |\nabla d(x)|^2 + \ep^{2+\frac{2s}{1-2s}} \right)^{\frac{1}{2}-s}}+\frac{C}{\delta_0^{2s+1}} \\
        \le & \frac{C}{ \delta_0^{2s}\left(\ep^{2+\frac{2s}{1-2s}} \right)^{\frac{1}{2}-s}}+\frac{C}{\delta_0^{2s+1}}
        \\ =& C \ep^{-1}  \left( \frac{\ep^{s}}{\delta_0^{2s}} + \frac{\ep}{\delta_0^{2s+1}} \right).
     \end{align*}
From this  last estimate  and \eqref{partialxbestnablaafar}, we  obtain  

\begin{equation}\label{partial_a_far_new}
 |\partial_{x_i} \bar a_\ep| \le CR\ep^{-1} \frac{\ep^s}{\delta_0^{n+2s}}\quad \text{if }|d(x)|\geq \delta_0.
\end{equation}
From  \eqref{partial_a_near_new} and  \eqref{partial_a_far_new}, choosing $\delta_0>2\ep^\frac12$ such that $\delta_0=o_\ep(1)$ and $\ep^s/\delta_0^{n+2s}=o_\ep(1)$, estimate
\eqref{partial_a_near} follows.

\section{Proof of Lemma \ref{lem:ae psi estimate}}\label{sec:proof of ae psi estimate}
Lemma \ref{lem:ae psi estimate} is a consequence of the following three lemmas.
\begin{lem}\label{lem:psi estimate near}
 Assume $\ep/\delta^2=o_\ep(1)$. Then, for all  $(t,x) \in [t_0,t_0+h] \times \mathbb{R}^n$, if $|d(t,x)|< \delta/2$,  
    \begin{align*}
    \ep^{2s} \mathcal{I}_n^s\left[ \psi \( \frac{d(t,\cdot)}{\ep};t,\cdot\) \right](x) 
 - C_{n,s} \mathcal{I}_1^s[\psi\(\cdot;t,x\)]\(\frac{d(t,x)}{\ep}\)=Ro_\ep(1).
 \end{align*}
\end{lem}

\begin{lem}\label{lem:psi estimate far part a}
Assume $\ep/\delta^2=o_\ep(1)$. Then,  for all $(t,x) \in [t_0,t_0+h] \times \mathbb{R}^n$, if $|d(t,x)|\geq  \delta/2$,  
\begin{align*}
    \abs{ C_{n,s} \mathcal{I}_1^s[\psi\(\cdot;t,x\)]\(\frac{d(t,x)}{\ep}\)} =o_\ep(1). 
    \end{align*}
\end{lem}

\begin{lem}\label{lem:psi estimate far part b}
    Assume $\ep/\delta^2=o_\ep(1)$. Then  for all $(t,x) \in [t_0,t_0+h] \times \mathbb{R}^n$, if $|d(t,x)|\geq  \delta/2$,  
  \begin{align*}
    \ep^{2s} \mathcal{I}_n^s\left[ \psi \( \frac{d(t,\cdot)}{\ep};t,\cdot\) \right](x)=R o_\ep(1).
    \end{align*}
\end{lem}

For simplicity of notation, we drop the dependence on  $t$ in the following proofs. 

\subsection{Proof of Lemma \ref{lem:psi estimate near}} 
Using Lemma \ref{lem:1 to n facrional Laplacian}  for $v=\psi(\cdot;x)$  with $e=\nabla d(x)$, and recalling that $|\nabla d(x)|=1$ when $\abs{d(x)}<\delta/2$, we obtain
\begin{align*}
\ep^{2s}& \mathcal{I}_n^s\left[ \psi \( \frac{d(\cdot)}{\ep};\cdot\) \right](x) 
 - C_{n,s} \mathcal{I}_1^s[\psi\(\cdot;x\)]\(\frac{d(x)}{\ep}\) \\&
 =\int_{\R^n}\left(\psi\(\frac{d(x+\ep z)}{\ep};x+\ep z\)-\psi\(\frac{d(x)}{\ep}+\nabla d(x)\cdot z;x\)\right)\frac{dz}{|z|^{n+2s}}\\&
 =\int_{\R^n}\left(\psi\(\frac{d(x+\ep z)}{\ep};x+\ep z\)-\psi\(\frac{d(x)}{\ep}+\nabla d(x)\cdot z;x+\ep z \)\right)\frac{dz}{|z|^{n+2s}}\\&
 \quad+\int_{\R^n}\left(\psi\(\frac{d(x)}{\ep}+\nabla d(x)\cdot z;x+\ep z \)-\psi\(\frac{d(x)}{\ep}+\nabla d(x)\cdot z;x\)\right)\frac{dz}{|z|^{n+2s}}\\&
=:I+II.
 \end{align*}
 \noindent First, let us estimate $I$. We  split
\begin{align*}
 I&
 = \int_{ \{ |z| < \ep^{-1\!/2} \} } (\ldots) + \int_{ \{ |z| > \ep^{-1\!/2} \} } (\ldots) 
  =: I_1 + I_2.
\end{align*}
Using that 
\begin{equation}\label{I_1-psilemsecondorder}\abs{\psi\(\frac{d(x+\ep z)}{\ep};x+\ep z\)-\psi\(\frac{d(x)}{\ep}+\nabla d(x)\cdot z;x+\ep z\)}\leq C\|\dot\psi\|_\infty\ep|z|^2, 
\end{equation}
and estimates \eqref{psi_near} and \eqref{psi_far} for $\dot{\psi}$,  we get
\begin{align*}
    |I_1| \le C  \frac{\ep}{\delta^{2s}} \int_{ \{ |z| < \ep^{-1\!/2} \} } \frac{dz}{|z|^{n+2s-2}} \le C \frac{\ep^s}{\delta^{2s}}.
\end{align*}
Using  \eqref{psi_near}  and \eqref{psi_far}  for $\psi$, we obtain
\begin{align*}
    |I_2| \le 2 \| \psi \|_{\infty}  \int_{ \{ |z| > \ep^{-1\!/2} \} } \frac{dz}{|z|^{n+2s}} \leq\frac{C}{\delta^{2s}} \int_{ \{ |z| > \ep^{-1\!/2} \} } \frac{dz}{|z|^{n+2s}} \le C \frac{\ep^s}{\delta^{2s}}.
\end{align*}
The estimates on $I_1$ and $I_2$ imply
\begin{equation}\label{Iest-psilemma-near}
|I|\leq  C\frac{\ep^s}{\delta^{2s}}.
\end{equation}
Next, we  estimate $II$. 
By \eqref{partial_psi_near} there exists $\tau=o_\ep(1)$ such that $|\nabla_x\psi(\xi; y)|\leq C\ep^{-1}\tau R$ if $|d(y)|<\delta$.
By taking   $\tau$ larger if necessary, we may assume $\ep$  smaller than   $\tau\delta/2$. Then, we can split $II$ as follows
\begin{align*}
    II = \int_{ \{ |z| < \tau^{-1} \} } (\ldots) + \int_{ \{  \tau^{-1}<|z| < \frac{\delta}{2\ep} \}}  (\ldots) +\int_{\{|z|> \frac{\delta}{2\ep} \}}(\ldots)=: II_1 + II_2+II_3.
\end{align*}
 Note that since  $|d(x)|<\delta/2$, for  $|z|<\delta/(2\ep)$ we have $|d(x+\ep z)|<\delta$. Therefore, using that for $|z|<\tau^{-1}<\delta/(2\ep)$,
\begin{equation}\label{II_1-psi_first_order_est}\begin{split}
    \abs{\psi\(\frac{d(x)}{\ep}+\nabla d(x)\cdot z;x+\ep z \)-\psi\(\frac{d(x)}{\ep}+\nabla d(x)\cdot z;x\)}& \le C\sup_{|d(y)|<\delta} \norm[\infty]{\nabla_x \psi(\cdot; y)} \ep |z|\\&
    \leq C\tau R|z|,
\end{split}\end{equation}
 we obtain
\begin{align*}
    |II_1| \le C\tau R \int_{ \{ |z| < \tau^{-1} \} } \frac{dz}{|z|^{n+2s-1}} \le C\tau^{2s}R. 
\end{align*}
Using \eqref{psi_near} for $\psi$,
\begin{align*}
    |II_2| \le 2  \sup_{|d(y)|\leq\delta}\| \psi (\cdot; y)\|_{\infty} \int_{ \{ |z| > \tau^{-1} \} } \frac{dz}{|z|^{n+2s}}  \le C \tau^{2s}. 
\end{align*}
Finally, using  \eqref{psi_far} for $\psi$,
\begin{align*}
    |II_3| \le 2 \| \psi \|_{\infty} \int_{ \{ |z| > \frac{\delta}{2\ep} \} } \frac{dz}{|z|^{n+2s}}  \le C \frac{\ep^{2s}}{\delta^{4s}}.
\end{align*}
The estimates on $II_1,II_2$ and $II_3$ imply 
\begin{equation*}\label{IIest-psilemma-near} |II|\leq C\(\tau^{2s}R+ \frac{\ep^{2s}}{\delta^{4s}}\). \end{equation*}
Assuming $\ep/\delta^2=o_\ep(1)$, from  the estimate on $I$ in \eqref{Iest-psilemma-near} and the estimate  on $II$ the lemma follows.
\qed

\subsection{Proof of Lemma \ref{lem:psi estimate far part a}}
Using \eqref{eq:linearized operator}, \eqref{eq:linearized wave}, estimate \eqref{a_near}, and recalling the definition of $\mu$ in \eqref{def:mu_definition}, we get
\begin{align*}
   \left| \mathcal{I}_1^s[\psi(\cdot ; t,x)]\left(\frac{d(t,x)}{\ep} \right) \right| & \le C\left| \psi \left( \frac{d(x)}{\ep} \right)  \right|+ \frac{C}{\delta^{2s}} \dot{\phi} \left( \frac{d(x)}{\ep} \right)  + \frac{C}{\delta^{2s}}\left| \frac{W''\left(\phi\left(\frac{d(x)}{\ep}\right)\right)-W''(0)}{W''(0)} \right|.
        \end{align*}
        Since $|d(x)|\geq\delta/2$, from estimate \eqref{psi_far} for $\psi$  and estimate  \eqref{eq:asymptotics for phi dot} for $\dot{\phi}$,
        $$\left| \psi \left( \frac{d(x)}{\ep} \right)  \right|\leq C\frac{\ep^{2s}}{\delta^{4s}}\quad\text{ and }\quad 
        0< \dot{\phi} \left( \frac{d(x)}{\ep} \right)  \leq  C\frac{\ep^{2s+1}}{\delta^{2s+1}}.$$
        Let $H$ be the Heaviside function. Using that 
        $W''\(H \left( \frac{d(x)}{\ep}\right)\)=W''(0)$,   by estimate \eqref{eq:asymptotics for phi} we have
\begin{align*}
    \left| W'' \left( \phi \left( \frac{d(x)}{\ep} \right) \right) - W''(0) \right| \le C \left| \phi \left( \frac{d(x)}{\ep} \right) - H \left( \frac{d(x)}{\ep} \right) \right| \le C \frac{\ep^{2s}}{\delta^{2s}}.
\end{align*}
Assuming $\ep/\delta^2=o_\ep(1)$, the lemma   follows.

\subsection{Proof of Lemma \ref{lem:psi estimate far part b}}
 Assume $\ep/\delta^2=o_\ep(1)$.

We write, 
\begin{equation}\label{I_npsi-I_1psi_far_est}\begin{split}
    \ep^{2s} \mathcal{I}_n^s \left[ \psi \left( \frac{d( \cdot)}{\ep}; \cdot \right) \right] (x)= &  \ep^{2s} \mathcal{I}_n^s \left[ \psi \left( \frac{d( \cdot)}{\ep}; \cdot \right) \right](x)- |\nabla d(x)|^{2s} C_{n,s} \mathcal{I}_n^s [\psi(\cdot;x)] \left( \frac{d(x)}{\ep} \right) \\
    & +|\nabla d(x)|^{2s} C_{n,s} \mathcal{I}_n^s [\psi(\cdot;x)] \left( \frac{d(x)}{\ep} \right).
\end{split}
\end{equation}
Using Lemma \ref{lem:1 to n facrional Laplacian} for $v=\psi(\cdot;x)$ with $e=\nabla d(x)$ and as in the proof of Lemma \ref{lem:psi estimate near}, we obtain
\begin{align*}
    \ep^{2s} & \mathcal{I}_n^s \left[ \psi \left( \frac{d( \cdot)}{\ep}; \cdot \right) \right](x) - |\nabla d(x)|^{2s} C_{n,s} \mathcal{I}_n^s [\psi(\cdot;x)] \left( \frac{d(x)}{\ep} \right) \\
    & =\int_{\R^n}\left(\psi\(\frac{d(x+\ep z)}{\ep};x+\ep z\)-\psi\(\frac{d(x)}{\ep}+\nabla d(x)\cdot z;x+\ep z \)\right)\frac{dz}{|z|^{n+2s}}\\&
 \quad+\int_{\R^n}\left(\psi\(\frac{d(x)}{\ep}+\nabla d(x)\cdot z;x+\ep z \)-\psi\(\frac{d(x)}{\ep}+\nabla d(x)\cdot z;x\)\right)\frac{dz}{|z|^{n+2s}}\\&
=:I+II, 
\end{align*}
with 
\begin{align}\label{Iest-psilemma-far} |I|\leq C\frac{\ep^{s}}{\delta^{2s}}.
\end{align} 
Next, we   estimate $II$.  
If $|d(x)|\geq \delta/2$ and $|z|<\delta/(4\ep\|\nabla d\|_\infty)$, then, for all $\theta\in (0,1)$, 
$$\frac{d(x)}{\ep}+\nabla d(x)\cdot z\geq \frac{\delta}{4\ep}\quad\text{and}\quad |d(x+\theta\ep z)|>\frac{\delta}{4}.$$
Therefore, by  \eqref{partial_psi_far} for $\nabla_x \psi$, there exists $\tau=o_\ep(1)$ such that
\begin{equation}\label{psilemma3der}\begin{split}&\left|\psi\(\frac{d(x)}{\ep}+\nabla d(x)\cdot z;x+\ep z \)-\psi\(\frac{d(x)}{\ep}+\nabla d(x)\cdot z;x\)\right|
\\&\quad\leq\sup_{\theta\in(0,1)}\left|\nabla_x \psi\(\frac{d(x)}{\ep}+\nabla d(x)\cdot z;x+\theta \ep z \)\right|\ep |z|
\\&\quad \leq C\(R\tau +\frac{\ep}{\delta^{2s+1}}\)\frac{\ep^{2s}}{\delta^{2s}}|z|.
\end{split}\end{equation}
By eventually taking $\tau$ larger, if necessary, we may assume $\tau>4\ep\|\nabla d\|_\infty/\delta$. 
Then, we split $II$ as follows
\begin{align*}
    II = \int_{ \{ |z| <\tau^{-1}\} } (\ldots) + \int_{ \{  |z|>\tau^{-1} \} } (\ldots) =: II_1 +II_2.
\end{align*}
  By \eqref{psilemma3der} and using that $\tau>4\ep\|\nabla d\|_\infty/\delta$, we get
\begin{align*}
    |II_1| &\le C\(R\tau +\frac{\ep}{\delta^{2s+1}}\)\frac{\ep^{2s}}{\delta^{2s}}
     \int_{ \{ |z| < \tau^{-1} \} } \frac{dz}{|z|^{n+2s-1}} 
     = C\(R\tau +\frac{\ep}{\delta^{2s+1}}\)\frac{\ep^{2s}}{\delta^{2s}}\tau^{2s-1}
     \\&\le  C\(R\tau^{2s} +\frac{\tau^{2s}}{\delta^{2s}}\)\frac{\ep^{2s}}{\delta^{2s}}. 
\end{align*}
Finally, using \eqref{psi_far} for $\psi$,
\begin{align*}
    |II_2| \le 2 \| \psi \|_{\infty} \int_{ \{ |z| > \tau^{-1} \} } \frac{dz}{|z|^{n+2s}}  \le  \frac{C\tau^{2s}}{\delta^{2s}}. 
\end{align*}
From the estimates on $II_1$ and $II_2$,  we obtain
\begin{equation*}\label{IIest-psilemma-near} |II|\leq   
C\left(R\tau^{2s}+\frac{\tau^{2s}}{\delta^{2s}}\right). \end{equation*}
Without loss of generality, we may assume $\tau/\delta=o_\ep(1)$. Since we also have that
 $\ep/\delta^2=o_\ep(1)$, the estimate on $II$ and the estimate on $I$ in  \eqref{Iest-psilemma-far} imply that 
\begin{align*}
\Bigg|\ep^{2s} & \mathcal{I}_n^s \left[ \psi \left( \frac{d( \cdot)}{\ep}; \cdot \right) \right] (x)- |\nabla d(x)|^{2s} C_{n,s} \mathcal{I}_n^s [\psi(\cdot;x)] \left( \frac{d(x)}{\ep} \right) \Bigg|\leq Ro_\ep(1). 
\end{align*}
Moreover,  by Lemma \ref{lem:psi estimate far part a} we also have
\begin{align*}
    |\nabla d(x)|^{2s}C_{n,s} \mathcal{I}_1^s[\psi(\cdot;x)] \(\frac{d(x)}{\ep}\) =o_\ep(1). 
\end{align*}
Recalling \eqref{I_npsi-I_1psi_far_est},  the lemma follows from the last two estimates.

 \section*{Acknowledgements}

The second author has been supported by the NSF Grant DMS-2155156 ``Nonlinear PDE methods in the study of interphases.'' 



\end{document}